\documentclass[reqno]{amsart}
\usepackage{amssymb}
\usepackage{amsmath}
\usepackage[mathscr]{euscript}
\usepackage[small]{caption}
\usepackage{mathtools}
\usepackage{dsfont}
\usepackage{graphicx}
\usepackage{stmaryrd}
\graphicspath{ {images/} }
\usepackage{color}
\usepackage{todonotes}
\usepackage{changes}
\colorlet{Changes@Color}{orange}
\usepackage{todonotes}

\makeatletter
\@addtoreset{equation}{section}
\makeatother

\newtheorem{thm}{Theorem}[section]
\newtheorem{lem}[thm]{Lemma}
\newtheorem{prop}[thm]{Proposition}
\newtheorem{cor}[thm]{Corollary}

\newtheorem{defn}[thm]{Definition}
\newtheorem{rem}[thm]{Remark}

\newtheorem{definition}[thm]{Definition}

\newcounter{as}[section]

\newcommand{\mc}[1]{{\mathcal #1}}
\newcommand{\mf}[1]{{\mathfrak #1}}
\newcommand{\mb}[1]{{\mathbf #1}}
\newcommand{\bb}[1]{{\mathbb #1}}
\newcommand{\bs}[1]{{\boldsymbol #1}}
\newcommand{\ms}[1]{{\mathscr #1}}

\newcommand{\<}{\langle}
\renewcommand{\>}{\rangle}
\renewcommand{\Cap}{{\rm cap}}

\definecolor{bblue}{rgb}{.2,0.2,.8}

\begin{document}

\title[Metastable behavior of weakly mixing Markov chains]
{Metastable behavior of weakly mixing Markov chains: the case of
reversible, critical zero-range processes}

\author{C. Landim, D. Marcondes, I. Seo}

\begin{abstract}

We present a general method to derive the metastable behavior of
weakly mixing Markov chains. This approach is based on properties of
the resolvent equations and can be applied to metastable dynamics
which do not satisfy the mixing conditions required in \cite{BL1, BL2}
or in \cite{LMS2}.

As an application, we study the metastable behavior of critical
zero-range processes.  Let $r: S\times S\to \bb R_+$ be the jump rates
of an irreducible random walk on a finite set $S$, reversible with
respect to the uniform measure. For $\alpha >0$, let
$g: \bb N\to \bb R_+$ be given by $g(0)=0$, $g(1)=1$,
$g(k) = [k/(k-1)]^\alpha$, $k\ge 2$. Consider a zero-range process on
$S$ in which a particle jumps from a site $x$, occupied by $k$
particles, to a site $y$ at rate $g(k) r(x,y)$. For $\alpha \ge 1$, in
the stationary state, as the total number of particles, represented by
$N$, tends to infinity, all particles but a negligible number
accumulate at one single site. This phenomenon is called
condensation. Since condensation occurs if and only if $\alpha\ge 1$,
we call the case $\alpha =1$ critical.  By applying the general method
established in the first part of the article to the critical case, we
show that the site which concentrates almost all particles evolves in
the time-scale $N^2 \log N$ as a random walk on $S$ whose transition
rates are proportional to the capacities of the underlying random
walk.

\end{abstract}

\address{IMPA, Estrada Dona Castorina 110, CEP 22460 Rio de Janeiro,
Brasil
and CNRS UMR 6085, Université de Rouen, France. \\
e-mail: \texttt{landim@impa.br} } \address{Institute of Mathematics
and Statistics, Universidade de S\~{a}o Paulo, R. do Mat\~{a}o, 1010 -
Butant\~{a}, S\~{a}o Paulo - SP,
05508-090, Brazil. \\
e-mail: \texttt{dmarcondes@ime.usp.br}} \address{Department of
Mathematical Sciences and RIM, Seoul National University, Gwanak-Ro
1, Gwanak-Gu 08826, Seoul, Republic of Korea. \\
e-mail: \texttt{insuk.seo@snu.ac.kr} }
\maketitle

\section{Introduction}
\label{intro}

More than twenty years ago, stochastic dynamics which exhibit
condensation \cite{DGC98, ev1} have been introduced. These dynamics
describe a conservative evolution of particles on a finite or
countably infinite set $S$. Despite the fact that no particles are
created nor annihilated, the system condensates in the sense that a
macroscopic proportion of particles sit on a single site provided the
density exceeds a critical value \cite{jmp, gss, al, agl, xu}.

We consider here zero-range processes evolving on a finite set
$S$. The dynamics can be described as follows. Let
$r: S\times S\to \bb R_+$ be the jump rates of an irreducible random
walk on $S$, reversible with respect to a probability measure $m$. For
$\alpha >0$, let $g: \bb N\to \bb R_+$ be given by $g(0)=0$, $g(1)=1$,
$g(k) = [k/(k-1)]^\alpha$, $k\ge 2$. A particle jumps from a site $x$,
occupied by $k$ particles, to a site $y$ at rate $g(k) r(x,y)$.

\subsection*{Phase transitions.}

This model exhibits phase transitions at $\alpha =1$ and at
$\alpha=2$. Indeed, let $Z_\alpha \colon \bb R_+ \to \bb R_+$ be the
partition function given by
\begin{equation}
\label{n-1}
Z_\alpha (\varphi) \;=\; 1\;+\; \sum_{k\ge 1} \frac{\varphi^k}{g(1) \cdots
g(k)} \;=\; 1\;+\; \sum_{k\ge 1} \frac{\varphi^k}{k^\alpha} \;\cdot
\end{equation}
The parameter $\varphi$ is usually called the fugacity.  It is clear
that $\varphi_c =1$ is the radius of convergence of the series for all
$\alpha>0$. The grand canonical stationary states, denoted by
$\{\nu^{(\alpha)}_\varphi : 0 \le \varphi <\varphi_c\}$, of the
zero-range processes are given by
\begin{equation*}
\nu^{(\alpha)}_\varphi (\eta) \;=\; \prod_{x\in S} \frac{1}{Z_\alpha
(\varphi)} \frac{\varphi^{\eta_x}}{a(\eta_x)}\;\cdot
\end{equation*}
 where $a(0) = 1$ and $a(\eta) = \eta^{\alpha}$ for $\eta \geq 1$. In this formula, $\eta=(\eta_x)_{x\in S}$ represents a configuration
of particles and $\eta_x$ the number of particles at site $x$ for the
configuration $\eta$.

The density of particles under the stationary state
$\nu^{(\alpha)}_\varphi$, denoted by $R_\alpha (\varphi)$, is given by
$R_\alpha (\varphi) = \varphi\, Z'_\alpha (\varphi)/ Z_\alpha
(\varphi)$, where $Z'_\alpha$ stands for the derivative of $Z_\alpha$.

By \eqref{n-1}, for $0<\alpha\le 1$, $Z_\alpha (\varphi)$ and
$R_\alpha (\varphi)$ diverge as $\varphi \to \varphi_c$. In
particular, for every density $\rho>0$ there exists (a unique)
fugacity whose corresponding grand canonical state has density
$\rho$. For $1<\alpha\le 2$, $Z_\alpha (\varphi)$ converges, but
$R_\alpha (\varphi)$ diverges. In this range, it still holds that for
every density $\rho>0$ there exists (a unique) fugacity whose
corresponding grand canonical state has density $\rho$. Finally, for
$\alpha>2$, $Z_\alpha (\varphi)$ and $R_\alpha (\varphi)$ converge as
$\varphi \to \varphi_c$, and there is a critical density $\rho_c$
above which there is no fugacity whose corresponding grand canonical
state has density $\rho$.

\subsection*{Condensation.}

By the previous considerations, in the thermodynamical limit,
condensation appears only for $\alpha>2$. However, in the context of a
fixed finite set $S$ with the total number of particles increasing to
infinity, condensation also occurs in the range $1\le \alpha \le
2$. As there is no condensation for $\alpha<1$ when $S$ is fixed and
finite, we call the parameter $\alpha=1$ \emph{critical} and
$\alpha>1$ \emph{super-critical}.

For each $N\ge 1$, representing the total number of particles, denote
by $\mu_N$ the unique stationary state of the zero-range dynamics with
$N$ particles evolving on $S$. Fix a sequence $(\ell_N : n\ge 1)$ of
integer numbers such that $\ell_N \to \infty$, $\ell_N/N\to 0$. Denote
by $\mc E^x_N$, $x\in S$, the set of configurations given by
\begin{equation*}
\mc E^x_N \;=\; \big\{ \eta : \eta_x \,\ge\, N - \ell_N\}\;.
\end{equation*}
Hence, $\mc E^x_N$ represents the set of configurations with at least
$N -\ell_N$ particles at site $x$, that is, the configurations in
which a condensate has been formed at site $x$.

By \cite{BL3} for $\alpha > 1$, and by Theorem \ref{t26} below for
$\alpha =1$ (under additional assumptions on the sequence $\ell_N$),
$\mu_N(\mc E^x_N)\to 1/|S|$. Therefore, under the stationary state,
essentially all particles sit on a single site.

\subsection*{The evolution of the condensate.}

Once condensation has been established, it becomes natural to consider
the time evolution of the model. One expects to observe two different
regimes. As particles accumulate on a single site in the stationary
state, starting from a homogeneous distribution of particles among all
sites, coarsening should occur in a certain time-scale, and particles
should gradually concentrate on fewer and fewer sites, until the
system saturates and almost all of them sit on a single site. This
regime is called in Physics literature the coarsening phase of the
dynamics. It has been established in \cite{bjl} for $\alpha>1$ and
shown to occur in the time-scale $N^2$.

Consider a configuration in which all particles sit on the same
site. Call condensate the site at which this occurs. On a longer
time-scale, one expects to observe an evolution of the condensate.
This has been quantitatively analyzed for zero-range processes
evolving on a finite set for $\alpha>1$ in \cite{BL3, Lan2, Seo, OR}
and in the thermodynamical limit (when the number of sites increases
together with the number of particles) for $\alpha>20$ in \cite{AGL2}.

In this article, we examine the evolution of the condensate on a
finite set in the critical case $\alpha=1$. In this context, there is
an important difference between the case $\alpha>1$, considered
previously, and the case $\alpha=1$ studied here.  In the former, on a
fixed number of sites, starting from a configuration in a set
$\mc E^x_N$, called from now on ``well'', the process visits all
configurations of $\mc E^x_N$ before hitting a new well $\mc E^y_N$,
$y\not = x$. Such dynamics are said to ``visit points''.  In contrast,
in the critical case $\alpha = 1$, this property does not hold because
the wells are much larger.

In \cite{BL1, BL2}, a general theory has been proposed to derive the
metastable behavior of dynamics which visit points. This approach has
been successfully applied to super-critical zero-range processes in
the aforementioned articles and to inclusion processes in \cite{BDG}.

\subsection*{A resolvent approach.}

Motivated by critical zero-range processes, in the first part of the
article, we present a general method to derive the metastable behavior
of dynamics which do not satisfy the assumptions of \cite{BL1, BL2}.

The approach consists in showing that the Markov chain fulfills two
conditions.  First, that the solution of some resolvent equations are,
in each well $\mc E^x_N$, close to a constant in the $L^1$
sense. Then, that starting from any point in $\mc E^x_N$, the process
does not jump immediately to another well $\mc E^y_N$, $y\not = x$.

In Section \ref{ns2}, we show that the first condition follows from a
spectral gap estimate for the dynamics obtained by reflecting the
process at the boundary of the metastable sets, and from a
characterization of the limits of the solution of the resolvent
equation over each well.

The fact that the process remains in a well for a reasonable amount of
time can be derived in two steps. One first show that the process
visits a deep region of the well before it reaches another well. This
part of the argument relies on the construction of a super-harmonic
function.  Then, one proves that starting from this deep region, the
process does not hit quickly another well.

The method proposed in \cite{LMS2}, which also relies on properties of
the resolvent equation, is designed for dynamics with good local
ergodic properties and requires either an estimate of the mixing time
of the reflected process or the property that the process visits a
specific point in the well in the metastable time-scale. In contrast,
the method proposed here is designed for dynamics where these
properties do not hold or cannot be proved. This is the case of
diffusions \cite{LLS22}, zero-range processes in the thermodynamical
limit \cite{RS22} and of many other dynamics in which the entropy, and
not only the energy landscape, plays a role in the metastable behavior
\cite{BL4, l-review} and references therein.

\subsection*{Back to critical zero-range processes.}

In the second part of the article, we apply the method described above
to critical zero-range processes.  All estimates are delicate in the
critical case due to the small difference between time-scales. While
coarsening occurs in the diffusive scale $N^2$, the evolution of the
condensate is observed in the time-scale $N^2 \log N$, and the
equilibration inside the sets $\mc E^x_N$ in a time-scale
$(N/\log N)^2$.

All these time-scales do not depend on the lattice geometry. The
asymptotic jump rates of the condensate, however, depend on the
geometry.

An interesting problem, left for future investigations, is the
description of the coarsening phase of this model.

\section{Metastability of weakly mixing Markov chains}
\label{prelim-1}

In this and the next section, we present the main results of the
article. We provide here a set of sufficient conditions for a sequence
of continuous-time Markov chains with poor local mixing conditions to
exhibit a metastable behavior. All new notation introduced in the text
and not in a displayed equation is presented in blue.

We start by introducing the general framework proposed in \cite{BL1, BL2}
to describe the metastable behavior of a Markovian dynamics as a
Markov chain model reduction.  Let $\color{bblue} \{\mc{H}_{N}:N\ge1\}$
be a collection of finite sets. Elements of the set $\mc{H}_{N}$ are
designated by the letters $\eta$, $\xi$, and $\zeta$.

Consider a sequence $\color{bblue} \{\xi_{N}(t):t\ge0\}$ of
$\mc{H}_{N}$-valued, irreducible, continuous-time Markov chains, whose
generator is represented by $\color{bblue} \ms{L}_{N}$.  Therefore,
for every function $f: \mc H_N \to \bb R$,
\begin{equation*}
(\ms L_N\, f)(\eta) \;=\; \sum_{\xi\in \mc H_N} R_N(\eta,\xi)\,
\big[\, f(\xi) - f(\eta)\, \big]\;,
\end{equation*}
where $\color{bblue} R_N(\eta,\xi)$ stands for the jump rates.  Denote
by $\lambda_N(\eta)$ the holding times of the Markov chain,
$\color{bblue} \lambda_N(\eta) = \sum_{\xi \not = \eta} R_N(\eta,\xi)$,
and by $\color{bblue} \mu_N$ the unique stationary state.

Denote by $\color{bblue} D(\bb{R}_{+},\,\mc{H}_{N})$ the space of
right-continuous functions ${\bf x}:\bb{R}_{+}\to\mc{H}_{N}$ with
left-limits, endowed with the Skorohod topology and its associated
Borel $\sigma$-field. For a probability measure $\nu$ on $\mc{H}_{N}$,
let ${\color{bblue}\mb{P}_{\nu}^{N}}$ be the measure on
$D(\bb{R}_{+},\,\mc{H}_{N})$ induced by the process $\xi_{N}(\cdot)$
starting from $\nu$. When  $\nu = \delta_\eta$ is the Dirac measure concentrated on a
configuration $\eta \in \mc H_N$, we denote
$\mb{P}_{\delta_\eta}^{N}$ by $\color{bblue} \mb{P}_{\eta}^{N}$.
Expectation with respect to $\mb{P}_{\nu}^{N}$, $\mb{P}_{\eta}^{N}$ is
represented by ${\color{bblue}\mb{E}_{\nu}^{N}}$,
${\color{bblue}\mb{E}_{\eta}^{N}}$, respectively.

Fix a finite set $S$, and denote by $\color{bblue} \mc{E}_{N}^{x}$,
$x\in S$, a family of disjoint subsets of $\mc{H}_{N}$. Let
\begin{equation}
\label{n9}
\mc{E}_{N}\,=\,\bigcup\limits _{x\in S}\mc{E}_{N}^{x}
\;\;\;\;\text{and\;\;\;\;}
\Delta_{N}\,=\,\mc{H}_{N}\,\setminus\,
\Big(\,\bigcup_{x\in S}\mc{E}_{N}^{x}\,\Big)\;.
\end{equation}

The sets $\mathcal{E}_N^x$, $x\in S$, represent the metastable sets of
the dynamics $\xi_N(\cdot)$, in the sense that, as soon as the process
$\xi_N(\cdot)$ enters one of these sets, say $\mathcal{E}_N^x$, it
equilibrates in $\mathcal{E}_N^x$ before hitting a new set
$\mathcal{E}_N^y$, $y\not = x$. These metastable sets are often called
wells.  The goal of the theory is to describe the evolution between
these wells. To this end, we introduce the so-called \textit{order
process.}

For $\mc{A}\subset\mc{H}_{N}$, denote by $T^{\mc{A}}(t)$ the total
time the process $\xi_{N}(\cdot)$ spends in $\mc{A}$ in the
time-interval $[0,t]$:
\begin{equation*}
T^{\mc{A}}(t)\;=\;\int_{0}^{t}\,\chi_{\mc{A}}(\xi_{N}(s))\,ds\;,
\end{equation*}
where \textcolor{bblue}{$\chi_{\mc{A}}$} represents the indicator
function of the set $\mc{A}$. Denote by $S^{\mc{A}}(t)$ the
generalized inverse of $T^{\mc{A}}(t)$:
\begin{equation}
\label{15}
S^{\mc{A}}(t)\;=\;\sup\{\,s\ge0\,:\,T^{\mc{A}}(s)\le t\,\}\;.
\end{equation}

The trace of $\xi_{N}(\cdot)$ on $\mc{A}$, denoted by
$\{\xi_{N}^{\mc{A}}(t) : t \ge 0\}$, is defined by
\begin{equation}
\label{104}
\xi_{N}^{\mc{A}}(t)\;=\;\xi_{N}(\,S^{\mc{A}}(t)\,)\;;\;\;\;t\ge0\;.
\end{equation}
It is an $\mc{A}$-valued, continuous-time Markov chain, obtained by
turning off the clock when the process $\xi_{N}(\cdot)$ visits the set
$\mc{A}^{c}$, that is, by deleting all excursions to $\mc{A}^{c}$. For
this reason, it is called the trace process of $\xi_{N}(\cdot)$ on
$\mc{A}$.

Let $\Psi_{N}:\mc{E}_{N}\to S$ be the projection given by
\begin{equation*}
\Psi_{N}(\eta)\;=\;\sum_{x\in S}x\cdot\chi_{\mc{E}_{N}^{x}}(\eta)\;.
\end{equation*}
The order process $(Y_{N}(t) : t\ge0)$ is defined as
\begin{equation}
Y_{N}(t)\;=\;\Psi_{N}(\xi_{N}^{\mc{E}_{N}}(t))\;,\;\;\;t\ge0\;.
\label{105}
\end{equation}
Denote by $\color{bblue} \bb Q^N_\nu$ the probability measure on
$D(\bb R_+, S)$ induced by the measure $\mb P_\nu^N$ and the order
process $Y_N$: $\bb Q^N_\nu = \mb P_\nu^N \circ Y_N^{-1}$.

Fix a probability measure $\color{bblue} \pi$ on $S$ and a
{\color{bblue} generator $L$} of a $S$-valued, continuous-time Markov
chain.  Denote by $\color{bblue} \bb Q^{ L}_\pi$ the measure on
$D(\bb R_+, S)$ induced by the Markov chain whose generator is $ L$
and which starts from $\pi$. Denote $\bb Q^{ L}_{\delta_x}$, $x\in S$,
simply by $\color{bblue} \bb Q^L_x$.

\begin{definition}
\label{def1}
Fix a sequence of probability measures $\{\nu_N : N \ge 1\}$ on
$\mc H_N$ such that $\nu_N (\mathcal{E}_{N}) =1$ for all $N\ge 1$.
The sequence of Markov chains $\{\xi_N(\cdot):N\ge 1\}$ is said to be
$(\nu_N, \{\mc E^x_N : x\in S\}, \pi,  L)$-metastable if

\smallskip
\noindent{\rm (A)} As $N\to\infty$, the sequence of laws
$(\bb Q_{\nu_N}^{N})_{N\in\bb{N}}$ converges weakly to
$\bb Q_{\pi}^{ L}$.
\smallskip

\noindent{\rm (B)} For all $t>0$,
\begin{equation*}
\lim_{N\to\infty}\, \mb{E}_{\nu_N}^{N}\,
\Big[\,\int_{0}^{t}\,\chi_{\Delta_{N}}(\xi_{N}(s))\,ds\,\Big]\;=\;0\;.
\end{equation*}
\end{definition}

Condition (A) asserts that the order process $Y_N(\cdot)$ converges
weakly and condition (B) that the process $\xi_{N}(\cdot)$ spends a
negligible amount of time on $\Delta_{N}$. It ensures, therefore, that
the trace process does not differ much from the original one when
starting from $\nu_N$. Note that, by condition (A),
$\nu_N(\mathcal{E}_N^x)\rightarrow \pi(x)$.

\subsection*{Comments}

The above definition of metastability differs from the one presented
in \cite{BL1, BL2} in that the initial state $\nu_N$ is a measure
spread over a well and not a Dirac measure concentrated on a
configuration. We introduce some notation to explain the reasons of
this modification.

Denote by $\tau_{\mc A}$, $\mc A \subset \mc H_N$ the hitting time of
the set $\mc A$:
\begin{equation*}
\tau_{\mc A} \;=\; \inf\{t\ge 0: \xi_N (t)\in \mc A\}\;,
\end{equation*}
and by
$\mathfrak{h}_{\mathcal{A},\,\mathcal{B}}:\mathcal{H}_{N}\rightarrow\mathbb{R}$
the equilibrium potential between two disjoint, non-empty subsets
$\mathcal{A}$ and $\mathcal{B}$ of $\mathcal{H}_{N}$:
\begin{equation*}
\mathfrak{h}_{\mathcal{A},\,\mathcal{B}}(\eta)
\;=\; \mathbf{P}_{\eta}^{N}\left[\tau_{\mathcal{A}}<\tau_{\mathcal{B}}\right]\;.
\end{equation*}
Recall that $\mu_N$ denotes the unique invariant measure for the
Markov chain $\xi_N(t)$, and denote by $\bb D_{N}$ the Dirichlet form
associated to the generator $\ms L_N$: for each
$F:\mathcal{H}_{N}\rightarrow\mathbb{R}$,
\begin{equation}
\label{e022}
\bb {D}_{N}(F) \;=\; \langle \,F\,,\, (- \ms L_N F)
\,\rangle _{\mu_{N}} \;,
\end{equation}
where $\color{bblue} \langle \,\cdot\,,\,\cdot\,\rangle _{\mu_{N}}$
stands for the scalar product in $L^2(\mu_N)$.  The capacity between
$\mathcal{A}$ and $\mathcal{B}$ is given by
\begin{equation}
\label{capac}
\textrm{cap}_{N}(\mathcal{A},\,\mathcal{B})
\;=\; \bb D_{N}(\mathfrak{h}_{\mathcal{A},\,\mathcal{B}})\;.
\end{equation}

The method proposed in \cite{BL1, BL2} to derive metastability relies
on the following condition.

\medskip\noindent{\bf (H1)} For each $x\in S$, there exists a sequence
of configurations {\color{bblue} $\{\xi_{N}^x : N\ge 1\}$} such that
$\xi_{N}^x \in \mc E^x_N$ for all $N\ge 1$ and
\begin{equation*}
\lim_{N\to \infty} \max_{\eta\in \mc E^x_N}
\frac{\Cap_N(\mc E^x_N,\breve{\mc E}^x_N)}
{\Cap_N(\xi_N^x, \eta)}\;=\;0\; ,
\end{equation*}
where
\begin{equation}
\label{Ehat}
\mathcal{\breve{E}}_{N}^{x} \;=\;
\mathcal{E}_{N}\setminus\mathcal{E}_{N}^{x}
\;=\; \bigcup_{y:y\neq x}\mathcal{E}_{N}^{y}\;.
\end{equation}

By \cite[Theorem 2.7]{BL2}, condition {\bf (H1)} implies that the
process $\xi_N(\cdot)$ visits all configurations of a well before it
hits a new well: for all $x\in S$,
\begin{equation}
\label{n16}
\lim_{N\to \infty} \max_{\eta, \xi \in \ms E^x_N}
\mb P^N_{\eta}\big[\, \tau_{\breve{\ms E}^x} \,<\,
\tau_{\xi}  \,\big]\, =\, 0\; .
\end{equation}

Condition {\bf (H1)} and its aftermath \eqref{n16} have been derived
for many dynamics (cf. \cite{l-review}), but they are clearly not
satisfied in many others. For instance, critical zero-range processes,
considered in this article, diffusions in potential fields
\cite{BEGK1, LMaS} or condensing zero-range dynamics in the
thermodynamical limit \cite{AGL2}, to mention a few.

In \cite{LMS2}, we present a robust method, based on properties of the
resolvent equation, to derive the metastable behavior of dynamics
satisfying \eqref{n16}. In this article, we present a tailor-made
approach, also based on properties of the resolvent equation, to
handle dynamics which do not.

\subsection*{Main result}

The main result of this section provides sufficient conditions for a
sequence of Markov chains to be metastable in the sense of Definition
\ref{def1}. Denote by $\mu_{N}^{x}$
the measure $\mu_N$ conditioned on $\mathcal{E}_{N}^{x}$:
\begin{equation}
\label{condmeas}
\mu_{N}^{x}(\eta) \,=\,
\frac{\mu_{N}(\eta)}{\mu_{N}(\mathcal{E}_{N}^{x})}\;,
\quad \eta\in\mathcal{E}_{N}^{x} \;.
\end{equation}

The first ingredient  is the following condition.

\smallskip\noindent{\bf (C1)} The set $\Delta_N$ is negligible in the
sense that for all $x\in S$, $\mu_N(\Delta_N)/\mu_N(\mc E^x_N) \to 0$.
Moreover, the initial state, represented by $\nu_N$, is concentrated on one
well: there exists $x_0\in S$ such that
$\nu_N (\mathcal{E}_{N}^{x_0}) =1$ for all $N\ge 1$. Finally, there
exists a finite constant $C_1$ such that
\begin{equation}
\label{nl2cond}
{E}_{\mu_{N}^{x_0}}\left[\left(
\frac{d\nu_{N}}{d\mu_{N}^{x_0}}\right)^{2}\right]
\;=\; \sum_{\eta\in\mathcal{E}_{N}^{x_0}}
\frac{\nu_{N}(\eta)^{2}}{\mu_{N}^{x_0}(\eta)} \;\le\; C_1
\;\; \text{for all $N\ge1$.}
\end{equation}

The second ingredient reads:

\smallskip\noindent{\bf (C2)} For all $x\in S$,
\begin{equation*}
\limsup_{a\rightarrow0} \limsup_{N\rightarrow\infty}
\sup_{\eta\in\mathcal{E}_{N}^{x}}
\mathbf{P}_{\eta}^{N}\big[\, \tau_{\breve{\mathcal{E}}_{N}^{x}}
\,<\, a\, \big] \;=\; 0\;.
\end{equation*}

The last condition {\bf(C3)} below requires the solutions of some
resolvent equations to be asymptotically constant on each well
$\mc E_N^x$, $x\in S$. Recall that $L$ represents the generator of a
$S$-valued continuous-time Markov chain.  Fix $\lambda>0$ and a
function $f:S \to \bb R$. Let
$G_{N}:\mathcal{H}_{N}\rightarrow\mathbb{R}$ be given by
\begin{equation}
\label{n03}
G_{N}(\eta) \;=\; \sum_{x\in S} [\, (\lambda \,-\, L) \, f\,] (x)\,
\chi_{\mathcal{E}_{N}^{x}}  (\eta) \;.
\end{equation}
That is, the function $G_{N}$ is equal to
$[\, (\lambda \,-\, L) \, f\,] (x)$ on $\mathcal{E}_{N}^{x}$,
$x\in S$, and it vanishes on $\Delta_{N}$.  Denote by
$F_{N}:\mathcal{H}_{N}\rightarrow\mathbb{R}$ the solution of the
resolvent equation
\begin{equation}
\label{nfg03}
(\lambda \,-\, \mathscr{L}_{N})\, F_{N} \;=\; G_{N}
\quad\text{on}\quad \mathcal{H}_{N}\;.
\end{equation}

\smallskip\noindent{\bf (C3)} For all $\lambda>0$,
function $f:S \to \bb R$ and $t\ge 0$,
\begin{equation}
\label{fg01}
\lim_{N\rightarrow\infty} \mathbf{E}_{\nu_{N}}^{N}
\left[\, \left|\, F_{N}(\xi_{N}^{\mathcal{E}_{N}}(t))
\,-\, f(Y_{N}(t))\, \right| \, \right]\;=\; 0\;.
\end{equation}

\begin{thm}
\label{mt0}
Assume that conditions {\rm (C1)} -- {\rm (C3)} are in force for some
$x_0 \in S$.  Then, the process $\xi_N(\cdot)$ is
$(\nu_N, \{\mc E^x_N : x\in S\}, \delta_{x_0}, L)$-metastable in the
sense of Definition \ref{def1}.
\end{thm}

\begin{rem}
In contrast with the others, condition {\rm (C2)} requires to
estimate an event with respect to a measure on the path space whose
initial distribution is a configuration. The initial state in the
other two conditions are not too far from the stationary measure
conditioned to a well. For this reason, in many dynamics, condition
{\rm (C2)} is the most difficult to prove.
\end{rem}

\begin{rem}
By \cite[Lemma 3.2 in Chapter 4]{ek}, it is enough to prove condition
{\rm (C3)} for just one $\lambda>0$. This observation, however, does
not simplify the proofs.
\end{rem}

In Section \ref{nsec1}, we prove Theorem \ref{mt0}, and in Section
\ref{ns2} we provide a set of conditions that yield (C2) and (C3). In
the second part of the article we show that these conditions are in
force for the critical condensing zero-range processes presented
below. In \cite{LLS22, RS22}, we apply the method presented here to
diffusions and to condensing zero-range processes in the
thermodynamical limit, respectively.

\section{Condensing zero-range processes}
\label{ns0}

Fix a finite set $S$ and let $\color{bblue} \kappa = |S|$, which is
assumed to be larger than or equal to $2$. Elements of $S$ are denoted
by the letters $x$, $y$, $z$. Let $\color{bblue} \{X(t)\}_{t\ge0}$ be
a continuous-time, irreducible Markov chain on the set $S$. The jump
rates are represented by $r$ and the generator by $L_{X}$ so that
\begin{equation*}
(L_{X}f)(x)=\sum_{y\in S}r(x,\,y)\, \big[\,f(y)-f(x)\,\big]
\end{equation*}
for all $f:S\rightarrow\mathbb{R}$. For convenience, we set
$r(x,\,x)=0$ for all $x\in S$.

Denote by $\color{bblue} \mathbb{P}_{x}$, $x\in S$, the law of the
random walk $X$ starting from $x$, and by {\color{bblue} $m$ its
unique invariant probability measure}.

Let $\color{bblue} \mathbb{N}=\{0,1,2,\dots\}$, fix
$\color{bblue} \alpha>0$, and define
$a = a_\alpha :\mathbb{N}\rightarrow\mathbb{R}^{+}$ as
\begin{equation*}
a(0)=1\;\;\;\text{and\;\;\;} a(n)=n^{\alpha}
\;\; \text{ for } \;\; n\ge 1\;.
\end{equation*}
Denote by $g = g_\alpha:\mathbb{N}\rightarrow\mathbb{R}^{+}$ the
function given by
\begin{equation*}
g (0)=0\;,\;\;\; g (1)=1
\;\;\;\text{and\;\;\;} g (n)=
\frac{{a}(n)}{{a}(n-1)}=
\left(\frac{n}{n-1}\right)^{\alpha}\;, \;\; n\geq 2\;.
\end{equation*}

Denote by $\mb a$, $\mb g: \bb N^{S} \to \bb R$ the functions given by
\begin{equation*}
\mathbf{g}(\eta)=\prod_{x\in S} g(\eta_{x})\;\;\;\text{and\;\;\;\,}
\mathbf{a}(\eta)=\prod_{x\in S} {a}(\eta_{x})\;.
\end{equation*}
For each $x\not = y\in S$ and $\eta\in \bb N^S$, denote by
$\sigma^{x,\,y}\eta$ the configuration obtained from $\eta$ by moving
a particle from $x$ to $y$:
\begin{equation*}
(\sigma^{x,\,y}\eta)_{z}=
\begin{cases}
\eta_{x}-1 & \text{ if }z=x\;,\\
\eta_{y}+1 & \text{ if }z=y\;,\\
\eta_{z} & \text{ otherwise\;,}
\end{cases}
\end{equation*}
if $\eta_{x}\ge1$, and $\sigma^{x,\,y}\eta=\eta$ if $\eta_{x}=0$.

Denote by $\mc H_N = \mc H_{S,\,N} $, $N\ge 1$, the set given by
\begin{equation*}
\mathcal{H}_{N} \;=\; \Big\{ \eta=(\eta_{x})_{x\in S}\in\mathbb{N}^{S}
:\sum_{x\in S}\eta_{x}=N \Big\} \;.
\end{equation*}
The zero-range process with parameters $\alpha$ and $r$ is the
$\mc H_N$-valued, continuous-time Markov chain
$\color{bblue} \{\eta_{N}(t)\}_{t\geq0}$ whose generator, denoted by
$\mathcal{\mathscr{A}}_N$, is given by
\begin{equation}
\label{generator}
(\mathcal{\mathscr{A}}_N F)(\eta)
\;=\; \sum_{x,\,y\in S} {g}(\eta_{x})\,
r(x,\,y)\, \big[\, F(\sigma^{x,\,y}\eta)-F(\eta)\,\big]\;,
\end{equation}
for all functions $F: \mc H_N \rightarrow\mathbb{R}$. Clearly, the
process $\eta_{N}(\cdot)$ is ergodic.

\subsection{Condensation and metastable behavior in the
  super-critical case}
\label{sec22}

In this subsection, we review the results for the super-critical case
$\alpha>1$ obtained in \cite{BL3,Lan2,Seo, OR}.

\subsubsection*{Condensation phenomenon}

To simplify the presentation, we assume that the invariant probability
measure $m$ of the underlying random walk $X(\cdot)$ is the uniform
measure:
\begin{equation}
\label{ua}
m(x) \;=\; \frac{1}{\kappa}\;, \quad x\in S\;.
\end{equation}
For the general result without this assumption, we refer to
\cite{BL3, Seo}.

The invariant probability measure for the zero-range process can be
written as
\begin{equation}
\label{inv}
\mu_{N}(\eta) \;=\; \frac{1}{\widehat{Z}_{N}}\,
\frac{1}{\mathbf{a}(\eta)}\;,
\end{equation}
where $\widehat{Z}_{N}$ is the normalizing constant given by
\begin{equation*}
\widehat{Z}_{N} \;=\;
\sum_{\eta\in\mathcal{H}_{N}}\frac{1}{\mathbf{a}(\eta)}\;.
\end{equation*}

Let $(\ell_{N})_{N\in\mathbb{N}}$ be a sequence of integer numbers
satisfying
\begin{equation*}
1\ll\ell_{N}\ll N\;.
\end{equation*}
Here, for two sequences $(a_{N})_{N\in\mathbb{N}}$ and
$(b_{N})_{N\in\mathbb{N}}$ of positive real numbers, $a_{N}\ll b_{N}$
stands for $\lim_{N\rightarrow\infty}a_{N}/b_{N}=0$.

Denote by $\mathcal{E}_{N}^{x}$, $x\in S$, the set of configurations
with at least $N-\ell_N$ particles at site $x$:
\begin{equation}
\label{enx}
\mathcal{E}_{N}^{x}=\left\{ \eta:\eta_{x}\ge N-\ell_{N}\right\} \;,
\end{equation}
and recall from \eqref{n9} the definition of the sets
$\mathcal{E}_{N}$ and $\Delta_{N}$.  The sets $\mathcal{E}^x_{N}$,
$x\in S$, are called the wells. To stress the dependence of $\Delta_N$
on the set $S$, we sometimes write $\Delta_N$ as
$\color{bblue} \Delta_{S,N}$.

The next result asserts that the dynamics tend to concentrate
particles on a single site. This is called the condensation
phenomenon.

\begin{thm}[Condensation in the super-critical case]
\label{t23}
For $\alpha>1$,
\begin{equation*}
\lim_{N\rightarrow\infty}
\mu_{N}(\mathcal{E}_{N}^{x}) \;=\;
\frac{1}{\kappa}\;\;\;\text{for all } x\in S\;.
\end{equation*}
\end{thm}

It follows from this result that $\lim_{N\rightarrow\infty}
\mu_{N}(\Delta_N) =0$. Versions of this result have been obtained in
\cite{ev1, jmp, gss, BL3, al, al2, agl}.  We refer to \cite[Section
3]{BL3} for a proof without the assumption \eqref{ua}.

\subsubsection*{Time-scale and speeded-up process}

Since the transition time between two wells is of order
$N^{1+\alpha}$, we speed-up the process by this amount: let
\begin{equation}
\label{xin}
\theta_N \;=\; N^{1+\alpha}\;, \quad
\xi_N(t) \;=\; \eta_N(t\, \theta_N)\;.
\end{equation}
The process $\xi_N(t)$ is the $\bb N^S$-valued, Markov chain whose
generator, denoted by $\mathscr{L}_N$, is given by
$\color{bblue}\mathscr{L}_N = \theta_N\, \mathscr{A}_N$, where the
generator $\ms {A}_N$ has been introduced in \eqref{generator}.

\smallskip\noindent{\bf Warning:} We borrow from the previous section
all notation introduced there. Besides the measure $\mu_N$, the
generator $\ms L_N$ and the process $\xi_N(\cdot)$, which already
appeared, this includes the probability measures $\mb P^N_\nu$,
$\mb P^N_\eta$ on $D(\bb R_+, \mc H_N)$, the Dirichlet form $\bb D_N$,
the capacity $\Cap_N$ and the measures $\bb Q^{N}_\nu$.

\subsubsection*{The limiting process $Y(\cdot)$ in the  super-critical case}

Denote by $D_{X}(\cdot)$ the Dirichlet form associated to the random
walk $X$: for $f:S\rightarrow\mathbb{R}$,
\begin{equation*}
D_{X}(f)=\frac{1}{2}\sum_{x,\,y\in S}
m(x)\, r(x,\,y)\, [f(y)-f(x)]^{2}\;.
\end{equation*}

Denote by $\tau_{C}$, $C\subset S$, the hitting
time of the set $C$:
\begin{equation*}
\tau_{C}=\inf\{t\ge 0: X(t)\in C\}\;.
\end{equation*}
Fix two non-empty, disjoint subsets $A$, $B$ of $S$.  The equilibrium
potential $h_{A,\,B}:S\rightarrow\mathbb{R}$ between $A$ and $B$ is
defined by
\begin{equation}
\label{PE}
h_{A,\,B}(x)=\mathbb{P}_{x}[\tau_{A}<\tau_{B}]\;, \quad x\in S\;.
\end{equation}
It is well-known that $h_{A,\,B}$ is the unique solution to the
Dirichlet problem:
\begin{equation*}
\begin{cases}
(L_{X}h)(x)=0 & x\in S\setminus\{A\cup B\}\;,\\
h(x)=1 & x\in A\;,\\
h (x)=0 & x\in B\;.
\end{cases}
\end{equation*}

The capacity $\text{cap}_{X}(A,B)$ between $A$ and $B$ is given by
\begin{equation}
\label{CAP}
\text{cap}_{X}(A,B)=D_{X}(h_{A,\,B})\;.
\end{equation}
If $A=\{x\}$ is a singleton we write $\text{cap}_{X}(x,B)$ instead of
$\text{cap}_{X}(\{x\},B)$.

The limiting process $Y(\cdot)$ is a continuous-time Markov chain
on $S$ whose generator, represented by $L_{Y}$, is given by
\begin{equation*}
(L_{Y}f)(x)=\frac{\kappa}{\Gamma_{\alpha}I_{\alpha}}
\sum_{y\in S}\text{cap}_{X}(x,y)\, [f(y)-f(x)] \;,
\quad f:S\rightarrow\mathbb{R}\;.
\end{equation*}
In this formula,
\begin{equation}
\label{ga}
\Gamma_{\alpha}  =\sum_{j=0}^{\infty}
\frac{1}{{a}(j)}=1+\sum_{n\ge1}
\frac{1}{n^{\alpha}} \quad \text{and} \quad
I_{\alpha} =\int_{0}^{1}u^{\alpha}(1-u)^{\alpha}du\;.
\end{equation}
Remark that $\Gamma_{\alpha}$ is finite because $\alpha>1$.

Denote by $\color{bblue} \bb Q^Y_x$ the probability measure on $D(\bb
R_+, S)$ induced by the Markov chain associated to the generator $L_Y$
starting from $x$.

\subsubsection*{Metastable behavior}

We may now describe the evolution of the condensate, characterizing
the metastable behavior of super-critical zero range processes.
Recall from \eqref{105} the definition of the process $Y_N(t)$ and of
the measure $\bb Q^N_\nu$.

\begin{thm}
\label{t24}
Suppose that $\alpha>1$ and that the invariant probability measure $m$
of $X(t)$ is the uniform measure \eqref{ua}. Fix $\delta>0$ small and
let the sequence $\ell_{N}$, introduced in \eqref{enx}, be given by
$\ell_N = [N^{\delta}]$, where {\color{bblue} $[r]$ stands for the
integer part of $r>0$}. Fix $x_0\in S$ and a sequence
$\{\eta_N : N\ge 1\}$ such that $\eta_N \in \mathcal{E}_{N}^{x_0}$ for
all $N\ge1$.  Then, the sequence of Markov chains $\xi_N(\cdot)$ is
$(\delta_{\eta_N}, \{\mc E^x_N : x\in S\}, \delta_{x_0},
L_Y)$-metastable in the sense of Definition \ref{def1}. Moreover,
\begin{equation*}
\lim_{N\rightarrow\infty}\, \sup_{\eta\in\mathcal{E}_{N}}
\mathbf{E}_{\eta}^{N}\Big[\, \int_{0}^{t}
\chi_{\Delta_{N}}(\xi_{N}(s)) \; ds \, \Big] \;=\; 0\;.
\end{equation*}
\end{thm}

This result has been proven, without the uniformity assumption
\eqref{ua}, in \cite{BL3} for reversible dynamics. It has been
extended in \cite{Lan2, Seo} to the general case.

\subsection{Critical zero-range processes}
\label{sec23}

We turn to the case $\alpha=1$, under the uniformity condition
\eqref{ua}. We adopt the same notation as in the previous subsection.
The only and important difference lies on the definition of $\ell_N$,
which defines the wells, and of $\theta_N$, which describes the
time-scale.

\subsubsection*{Condensation of particles}

We first describe the condensation. Let $\ell_N$ be the sequence given by
\begin{equation}
\label{defln}
\ell_{N} \;=\; \Big[ \, \frac{N}{\log N}\,\Big] \;.
\end{equation}

\begin{thm}
\label{t26}
The assertions of Theorem \ref{t23} hold for $\alpha=1$ provided
$(\ell_{N})_{N\in\mathbb{N}}$ is chosen as in \eqref{defln}.
\end{thm}

The proof of this result, given in Section \ref{sec3}, is similar to
the one of the super-critical case, presented in \cite[Section 3]{BL3}.
The assumption \eqref{ua} can be removed, at the cost of heavy
notation, which we preferred to avoid.

\begin{rem}
It follows from the proof of Theorem \ref{t26} that the sequence
$\ell_{N}$ needs only to satisfy the conditions
\begin{equation*}
\lim_{N\rightarrow\infty}\frac{\ell_{N}}{N}\;=\; 0\;\;\;\text{and}\;\;\;
\lim_{N\rightarrow\infty}\frac{\log\ell_{N}}{\log N}=1\;.
\end{equation*}
In particular, we can select $\ell_{N}=[N/(\log N)^h]$ for any $h>0$.

Theorem \ref{t26} fails, however, for $\ell_{N} = [N^{\delta}]$,
$\delta\in(0,\,1)$, see Lemma \ref{p32}. Therefore, in the critical
case there are typically $N/\log N$ particles not sitting at the
condensate, while in the super-critical case there are less than $k_N$
particles, where $k_N$ is any sequence increasing to $\infty$. In
particular, the wells in the critical case are much larger than in the
super-critical case. This is a source of problems and explains why the
critical case is much more demanding than the super-critical one.
\end{rem}

\subsubsection*{Time-scale}

The transition time between two wells can be easily guessed by
examining the case with two sites, where the zero-range becomes a
birth-and-death process on $\{0, \dots, N\}$. In this situation, one
can compute explicitly the capacities between two wells and deduce
from them the time-scale.  In the critical case, it is of order
$N^{2}\,\log N$. In particular, in the critical case, in the
definition of the process $\xi_N(t)$, introduced in \eqref{xin}, we
take $\color{bblue} \theta_N = N^{2}\,\log N$.

\subsubsection*{Metastable behavior}

The statement of the metastable behavior of the condensate requires
further notation and hypotheses.

We assume in this article that the underlying random walk is
reversible with respect to the invariant (uniform) measure $m$:
\begin{equation}
\label{sym}
r(x,\,y)=r(y,\,x) \quad \forall\,  x,\,y\in S\;.
\end{equation}

The evolution of the condensate is described by the $S$-valued Markov
chain, denoted by $\color{bblue} \{Z(t) : t\ge 0\}$, whose generator
$L_Z$ is given by
\begin{equation}
\label{genz}
(L_{Z}f)(x) \;=\;
\sum_{y\in S} r_{Z} (x,\,y)\, \{f(y)-f(x)\} \;,
\quad f:S\rightarrow\mathbb{R} \;,
\end{equation}
where
\begin{equation}
\label{bexy}
r_Z(x,y) \;\coloneqq\; 6\, \kappa\, \text{cap}_{X}(x,y)\;,
\quad x,\,y\in S\;.
\end{equation}
The factor $6$ in this formula represents $1/I_1$, where $I_\alpha$ is
defined in \eqref{ga}. Note that the invariant measure of $Z(\cdot)$
is the uniform distribution $m$ on $S$ since the capacity is
symmetric.  Denote by $\color{bblue} \bb Q^Z_x$ the probability
measure on $D(\bb R_+, S)$ induced by the Markov chain associated to
the generator $L_Z$ starting from $x$.

Recall the definition of the measure $\mu_N$ introduced in
\eqref{inv}, and the one of $\mu_{N}^{x}$ introduced in
\eqref{condmeas}.
The second main result of this article reads as follows.

\begin{thm}
\label{t27}
Assume that $\alpha=1$ and that conditions \eqref{sym} are in
force. Fix $x_0\in S$ and a sequence of probability measures
$\{\nu_N : N \ge 1\}$ on $\mc H_N$ satisfying condition {\rm (C1)}.
Then, the sequence of Markov chains $\xi_N(\cdot)$ is
$(\nu_N, \{\mc E^x_N : x\in S\}, \delta_{x_0}, L_Z)$-metastable in the
sense of Definition \ref{def1}.
\end{thm}

The proof of this result is outlined in Section \ref{sec24}.

\begin{rem}
\label{mrm0}
No assumption on the geometry of the lattice is needed. We only
require the stationary measure of the underlying random walk to be
uniform, which is not a necessary but simplifying assumption.
\end{rem}

\begin{rem}
\label{mrm1}
In the cases $|S|=2$ or $3$, one can prove that the process visits all
configurations in a well before hitting a new well in the sense of
condition {\rm ({\bf H1})} of \cite{BL1}. In particular, in these low
dimensions, one can repeat the approach presented in \cite{BL3} to
derive the metastable behavior of the critical zero-range process.
\end{rem}

\begin{rem}
\label{mrm2}
The result should hold without the assumptions that the stationary
measure of the random walk $X(t)$ is uniform and that the process is
reversible.  The first hypothesis should not be difficult to
remove. It is a minor technical point. The second one provides some
symmetry in the construction of a super-harmonic function in Section
\ref{sec9}. In the general case, another test function has to be
created.
\end{rem}

\begin{rem}
On the diffusive scale $N^2$, as in the super-critical case
\cite{bjl}, we expect the density of particles to converge weakly to a
diffusion which is absorbed at the boundary. This is an open problem
which deserves to be considered.
\end{rem}

\begin{rem}
Another interesting open problem is to derive the evolution of the
condensate in the case where $S$ is the one-dimensional torus with
$\kappa = \kappa_N \to \infty$ points and $X(t)$ a finite-range,
symmetric random walk on $S$.

It has been proven in \cite{AGL2} that the condensate evolves as a
L\'evy process when $\alpha>20$ and $N/\kappa_N \to \rho > \rho_c$,
where $\rho_c$ is the critical density above which condensation
occurs.

It is being examined in \cite{RS22}, with the method presented in the
previous section, in the cases where $N/\kappa_N \to \rho > \rho_c$,
$\alpha>2$ and $\kappa_N \to \infty$, $N/\kappa_N\to \infty$ and
$\alpha \ge 1$.
\end{rem}

\begin{rem}
In \cite{LMS2}, we examine the metastable behavior of critical
zero-range processes starting from a configuration, instead of
starting from a measure.
\end{rem}

\begin{rem}
With a little more effort, one can prove that the finite-dimensional
distributions of the process $\Psi_{N}(\xi_N(t)) $ converge to the
ones of the process $Y(t)$, applying Proposition 2.1 of \cite{LLM}.
\end{rem}

\section{Proof of Theorem \ref{mt0}}
\label{nsec1}

The assertions of Theorem \ref{mt0} follow from Lemma \ref{nl01} and
Propositions \ref{nt44} and \ref{np-uniq} below. We first derive
condition (B) of Definition \ref{def1}.

\begin{lem}
\label{nl01}
Condition {\rm (C1)} implies condition {\rm (B)} of Definition
\ref{def1}.
\end{lem}

\begin{proof}
Fix $t>0$ and define $e_N: \mc H_N \to \bb R$ by
\begin{equation*}
e_{N}(\eta) \;=\; \mathbf{E}^N_{\eta}
\Big[\, \int_{0}^{t} \chi_{\Delta_{N}}(\xi_{N}(s)) \, ds\, \Big]\;.
\end{equation*}
Clearly, $e_N$ is uniformly bounded by $t$.  By definition of
$\mu_{N}^{x_0}$,
\begin{equation*}
\sum_{\eta\in\mathcal{E}_{N}^{x_0}} \mu_{N}^{x_0}(\eta)\,
e_{N}(\eta)\;=\; \frac{1}{\mu_{N}(\mathcal{E}_{N}^{x_0})}
\sum_{\eta\in\mathcal{E}_{N}^{x_0}} \mu_{N}(\eta)\,
e_{N}(\eta) \;.
\end{equation*}
By Fubini's theorem, and since $\mu_N$ is invariant, this expression
is bounded by
\begin{align}\label{rev1}
\frac{1}{\mu_{N}(\mathcal{E}_{N}^{x_0})} \, \mathbf{E}_{\mu_{N}}^{N}
\Big[\, \int_{0}^{t} \chi_{\Delta_{N}}(\xi_{N}(s)) \, ds\, \Big]
\;=\; \frac{t\, \mu_{N}(\Delta_{N})}{\mu_{N}(\mathcal{E}_{N}^{x_0})}
\;\cdot
\end{align}

Let $f_N(\eta) = \nu_{N}(\eta)/\mu^{x_0}_N(\eta)$
By the Cauchy-Schwarz inequality and since $e_N$ is uniformly bounded
by $t$,
\begin{align*}
& \mathbf{E}^N_{\nu_{N}}
\Big[\, \int_{0}^{t} \chi_{\Delta_{N}}(\xi_{N}(s)) \, ds\, \Big]
\;=\; \sum_{\eta\in\mathcal{E}_{N}^{x}}\mu^{x_0}_{N}(\eta)\, f_N(\eta)\,
e_{N}(\eta) \\
&\quad
\le\;
\sqrt{t} \, \Big( \sum_{\eta\in\mathcal{E}_{N}^{x}} \mu^{x_0}_{N}(\eta)\, f_N(\eta)^2\,
\Big)^{1/2} \Big( \sum_{\eta\in\mathcal{E}_{N}^{x}} \mu^{x_0}_{N}(\eta)\, e_N(\eta)\,
\Big)^{1/2}  \;.
\end{align*}
By condition (C1) and \eqref{rev1}, this expression is bounded from
above by
\begin{equation*}
C_1^{1/2} \, t \,
\Big(\, \frac{\mu_{N}(\Delta_{N})}{\mu_{N}(\mathcal{E}_{N}^{x_0})}
\,\Big)^{1/2}\;,
\end{equation*}
where $C_1$ is the constant appearing in \eqref{nl2cond}.  By
condition (C1), this expression vanishes as $N\to\infty$, which proves
that (B) of Definition \ref{def1} holds.
\end{proof}

The proof that the sequence of measures $\bb Q^N_{\nu_N}$ converges
weakly to $\bb Q^L_x$ is divided in two parts: we show that the
sequence is tight and that the limit point is unique.

\subsection*{Tightness}

Denote by $\{\mathscr{F}_{t}^{0}\}_{t\ge0}$ the natural filtration of
$D(\bb R_+,\, \mc H_N)$,
$\color{bblue} \mathscr{F}_{t}^{0} = \sigma(\xi (s):s\in[0,\,t])$, and
by $\color{bblue} \{\mathscr{F}_{t}\}_{t\ge0}$ its usual augmentation.
Recall from \cite[Section 1.4, page 43]{RY} the definition of the
$\sigma$-algebra at a stopping time.  The next result is \cite[Lemma
7.2 and the paragraph below]{LS3}.

\begin{lem}
\label{lem42b}
We have that
\begin{enumerate}
\item For every $s\ge0$, the random time ${S}^{\mathcal{E}_{N}}(s)$,
introduced in \eqref{15}, is a stopping time with respect to the
filtration $\{\mathscr{F}_{t}\}_{t\ge0}$.

\item Let
$\color{bblue} \mathscr{G}^N_{t} =
\mathscr{F}_{{S}^{\mathcal{E}_{N}}(t)}$, $t\ge 0$, and let $\tau$ be a
stopping time with respect to the filtration
$\{\mathscr{G}_{t}^{N}\}_{t\ge0}$. Then, the random time
${S}^{\mathcal{E}_{N}}(\tau)$ is a stopping time with respect to the
filtration $\{\mathscr{F}_{t}\}_{t\ge0}$.

\item The trace process $\{\xi^{\mc E_N}_{N}(t)\}_{t\ge0}$ is a
$\mathcal{E}_{N}$-valued, continuous-time Markov chain with respect
to the filtration $\{\mathscr{G}^N_{t}\}_{t\ge0}$.
\end{enumerate}
\end{lem}

Denote by $\color{bblue} \mathscr{T}_{M}$, $M>0$, the collection of
stopping times, with respect to the filtration
$\{\mathscr{G}^N_{t}\}_{t\ge0}$, bounded by $M$. The proof of the next
result is similar to the one of \cite[Lemma 7.5]{LS3}. We present it
here for the sake of completeness.

\begin{lem}
\label{lem43b}
Suppose that the sequence of probability measures
$(\nu_{N})_{N\in\mathbb{N}}$ satisfies condition {\rm (C1)}. For all
$M>0$, we have
\begin{equation*}
\lim_{a_{0}\rightarrow0} \limsup_{N\rightarrow\infty}
\sup_{\tau\in\mathscr{T}_{M}} \sup_{a\in(0,\,a_{0})}
\mathbf{P}_{\nu_{N}}^{N} \big[\, {S}^{\mathcal{E}_{N}}(\tau+a)
\,-\, {S}^{\mathcal{E}_{N}}(\tau) \,\ge\, 2a_{0}\,\big]
\;=\; 0\;.
\end{equation*}
\end{lem}

\begin{proof}
We note first that
\begin{equation*}
\{{S}^{\mathcal{E}_{N}}(\tau+a)-{S}^{\mathcal{E}_{N}}(\tau)
\, \ge\, 2a_{0}\} \; \subset\;
\Big\{\,
\int_{{S}^{\mathcal{E}_{N}}(\tau)}^{{S}^{\mathcal{E}_{N}}(\tau)+2a_{0}}
\chi_{\mathcal{E}_{N}}(\xi_{N}(t)) \; dt \,<\, a \,\Big\} \;.
\end{equation*}
Therefore, the probability appearing in the statement of the lemma
is bounded by
\begin{equation*}
\mathbf{P}_{\nu_{N}}^{N}
\Big[\,
\int_{{S}^{\mathcal{E}_{N}}(\tau)}^{{S}^{\mathcal{E}_{N}}(\tau)+2a_{0}}
\chi_{\Delta_{N}}(\xi_{N}(t))\; dt \,>\, 2a_{0}-a \, \Big]\;.
\end{equation*}
This expression is less than or equal to
\begin{equation}
\label{dec11b}
\mathbf{P}_{\nu_{N}}^{N} \Big[\,
\int_{0}^{2M+2a_{0}} \chi_{\Delta_{N}}(\xi_{N}(t))\; dt
\,>\, 2a_{0}-a \, \Big]
\;+\; \mathbf{P}_{\nu_{N}}^{N}
\left[\, {S}^{\mathcal{E}_{N}}(\tau) \,>\, 2M\, \right]\;.
\end{equation}
By the Markov inequality, the first probability is bounded from above
by
\begin{equation*}
\frac{1}{2a_{0}-a}\mathbf{E}_{\nu_{N}}^{N} \Big[\,
\int_{0}^{2M+2a_{0}}  \chi_{\Delta_{N}}(\xi_{N}(t))\; dt
\, \Big]\;.
\end{equation*}
Thus, by Lemma \ref{nl01},
\begin{equation*}
\lim_{a_{0}\rightarrow0} \limsup_{N\rightarrow\infty}
\sup_{\tau\in\mathscr{T}_{M}} \sup_{a\in(0,\,a_{0})}
\mathbf{P}_{\nu_{N}}^{N}\Big[\,
\int_{0}^{2M+2a_{0}} \chi_{\Delta_{N}}(\xi_{N}(t))\; dt
\,>\, 2a_{0}-a \, \Big] \;=\; 0\;.
\end{equation*}
We turn to the second term in \eqref{dec11b}. Note that
${S}^{\mathcal{E}_{N}}(\tau)>2M$ and $\tau\in\mathscr{T}_{M}$ implies
that
\begin{equation*}
\int_{0}^{2M}  \chi_{\Delta_{N}}(\xi_{N}(t))\; dt
\;>\; M\;.
\end{equation*}
Thus, the second term in \eqref{dec11b}  can be handled
as the previous one, which completes the proof of the lemma.
\end{proof}

Now we prove the main result regarding the tightness.

\begin{prop}
\label{nt44}
Under conditions {\rm (C1)} and {\rm (C2)}, the sequence of
probability measures $\{\mathbb{Q}_{\nu_{N}}^{N}\}_{N\in\mathbb{N}}$
is tight on $D(\bb R_+,\,S)$. Moreover, any limit point
$\mathbb{Q}^{*}$ satisfies $\mathbb{Q}^{*} [\, Y(0)=x_0 \,] \,=\, 1$
and
\begin{equation}
\label{n02}
\mathbb{Q}^{*}[\, Y(t) \not =Y(t-)\,]\, =\, 0\text{ for all } t>0\;.
\end{equation}
\end{prop}

\begin{proof}
By Aldous' criterion, it suffices to verify that for all $M>0$,
\begin{equation*}
\lim_{a_{0}\rightarrow0} \limsup_{N\rightarrow\infty}
\sup_{\tau\in\mathscr{T}_{M}} \sup_{a\in(0,\,a_{0})}
\mathbf{P}_{\nu_{N}}^{N}\left[\, Y(\tau+a) \neq Y(\tau) \, \right] \,=\, 0\;.
\end{equation*}
By Lemma \ref{lem43b}, it is enough to show that
\begin{equation*}
\lim_{a_{0}\rightarrow0} \limsup_{N\rightarrow\infty}
\sup_{\tau\in\mathscr{T}_{M}} \sup_{a\in(0,\,a_{0})}
\mathbf{P}_{\nu_{N}}^{N}\left[\, Y(\tau+a) \neq Y(\tau) \,,\,
{S}^{\mathcal{E}_{N}}(\tau+a)-{S}^{\mathcal{E}_{N}}(\tau)<2a_{0}\,
\right] \,=\, 0\;.
\end{equation*}
The last probability is bounded from above by
\begin{equation*}
\mathbf{P}_{\nu_{N}}^{N} \big [\,
\Psi(\xi ({S}^{\mathcal{E}_{N}}(\tau)+t)) \, \neq\,
\Psi(\xi ({S}^{\mathcal{E}_{N}}(\tau)))
\text{ for some } t\in(0,\,2a_{0})\,\big]\;.
\end{equation*}
By part (2) of Lemma \ref{lem42b} and the strong Markov property, this
expression is less than or equal to
\begin{equation*}
\sup_{\eta\in\mathcal{E}_{N}}\mathbf{P}_{\eta}^{N}
\big[\, \Psi(\xi (t)) \,\neq\, \Psi(\eta)
\text{ for some }t\in(0,\,2a_{0})\,\big]\;.
\end{equation*}
This expression is bounded from above by
\begin{equation*}
\max_{y\in S} \sup_{\eta\in\mathcal{E}_{N}^{y}}
\mathbf{P}_{\eta}^{N}\big[\, \tau_{\breve{\mathcal{E}}_{N}^{y}}<2a_{0}\,\big]\;.
\end{equation*}
To complete the proof of the first assertion of the proposition, it
remains to recall the content of condition (C2).

For the second assertion, note that $\mathbb{Q}^{*}[Y(0)=x_0]=1$
follows from the fact that $\nu_{N}$ is concentrated on
$\mathcal{E}^{x_0}_{N}$. For the last claim of the proposition, it
suffices to prove that
\begin{equation*}
\lim_{a_{0}\rightarrow0} \limsup_{N\rightarrow\infty}
\mathbf{P}_{\nu_{N}}^{N}\left[Y(t-a) \neq Y(t)
\text{ for some }a\in(0,\,a_{0})\right] \,=\, 0\;.
\end{equation*}
The proof of this estimate is identical to the one of the first claim.
\end{proof}

\smallskip\noindent{\bf Uniqueness.}  The uniqueness part relies on
the uniqueness of solutions of martingale problems. We start with two
estimates.  The proof of \cite[Lemma 4.3]{LMS2} and condition (B) of
Definition \ref{def1} yield that for all $t>0$, the random time
$S^{\mathcal{E}_N} (t)$ is close to $t$ in the sense that, for all
$\lambda>0$,

\begin{equation}
\label{n01}
\begin{gathered}
\lim_{N\to\infty} \mb{E}_{\nu_N}^{N}\,
\Big[\,e^{-\lambda t}\,-\,e^{-\lambda S^{\mc{E}_{N}}(t)}\,\Big]
\;=\;0\;,
\\
\text{and}\;\; \lim_{N\to\infty} \mb{E}_{\nu_N}^{N}\,
\Big[\,\int_{0}^{t}\,\Big\{ e^{-\lambda r}\,
-\,e^{-\lambda S^{\mc{E}_{N}}(r)}\,\Big\}\,dr\,\Big]\;=\;0\;.
\end{gathered}
\end{equation}

\begin{prop}
\label{np-uniq}
Assume that conditions {\rm (C1) -- (C3)} are in force. Let $\bb Q^{*}$
be a limit point of the sequence $\bb Q_{\nu_N}^{N}$ which
satisfies \eqref{n02}. Then, $\bb Q^{*}=\bb Q_{x}^{L}$.
\end{prop}

\begin{proof}
Fix $\lambda>0$, a function $f:S\rightarrow\bb{R}$, and let
$\color{bblue} g = (\lambda - L) f$. Recall from \eqref{n03} the
definition of $G_N$, and let ${F}_{N}$ be the solution of
\eqref{nfg03}.

Under the measure $\mb{P}_{\eta^{N}}^{N}$, the process
$M_{N}(t)$ given by
\begin{equation*}
M_{N}(t)\;=\;e^{-\lambda t}\,{F}_{N}(\xi_{N}(t))\,
-\,{F}_{N}(\xi_{N}(0))\,+\,\int_{0}^{t}e^{-\lambda r}
\big[\,(\,\lambda\,-\,\,{\ms{L}}_{N}\,)
\,{F}_{N}\,\big](\xi_{N}(r))\,dr
\end{equation*}
is a martingale with respect to the filtration
$\{\ms{F}_{t}\}_{t\ge0}$ defined in Lemma \ref{lem42b} above. By
\eqref{nfg03}, we may replace $(\,\lambda\,-\,{\ms{L}}_{N}\,)\,{F}_{N}$
by $G_{N}$. Thus, since $G_{N}$ vanishes on $\Delta_{N}$, we can rewrite $M_N(t)$ as
\begin{align*}
M_{N}(t)\;=\,\; & e^{-\lambda t}\,{F}_{N}(\xi_{N}(t))\,
-\,{F}_{N}(\xi_{N}(0))
\,+\,\int_{0}^{t}\,e^{-\lambda r}\,G_{N}(\xi_{N}(r))\,
\chi_{\mc{E}_{N}}(\xi_{N}(r))\,dr\;.
\end{align*}

Recall from Lemma \ref{lem42b} the definition of the filtration
$\{\ms{G}_{t}^{N}\}_{t\ge0}$.  Since $S^{\mc{E}_{N}}(t)$ is a stopping
time with respect to $\ms{F}_{t}$, the process
$\widehat{M}_{N}(t)=M(S^{\mc{E}_{N}}(t))$ is a martingale with respect
to the filtration $\{\ms{G}_{t}^{N}\}_{t\ge0}$:
\begin{equation*}
\begin{aligned}\widehat{M}_{N}(t)\;=\;\, &
e^{-\lambda S^{\mc{E}_{N}}(t)}\,{F}_{N}(\xi_{N}^{\mc{E}_{N}}(t))\,
-\,{F}_{N}(\xi_{N}^{\mc{E}_{N}}(0))\\
& \qquad +\,\int_{0}^{S^{\mc{E}_{N}}(t)}\,e^{-\lambda r}\,
G_{N}(\xi_{N}(r))\,\chi_{\mc{E}_{N}}(\xi_{N}(r))\,dr\;.
\end{aligned}
\end{equation*}

The presence of the indicator of the set $\mc{E}_{N}$ in the
integral permits to perform the change of variables $r'=T^{\mc{E}_{N}}(r)$.
Hence, as $T^{\mc{E}_{N}} (\, S^{\mc{E}_{N}}(t)\, ) \,=\, t$,
\begin{equation*}
\widehat{M}_{N}(t)\;=\;e^{-\lambda S^{\mc{E}_{N}}(t)}\,
{F}_{N}(\xi_{N}^{\mc{E}_{N}}(t))\,
-\,{F}_{N}(\,\xi_{N}^{\mc{E}_{N}}(0)\,)\,
+\,\,\int_{0}^{t}\,e^{-\lambda S^{\mc{E}_{N}}(r)}\,\,
G_{N}(\xi^{\mc{E}_{N}}(r'))\,dr'\;.
\end{equation*}

Recall the definition of $g:S\to \bb R$ introduced at the beginning of
this proof.  By \eqref{n01}, condition (C3), the definitions of
$G_{N}$, $Y_{N}(\cdot)$, and since $F_N$, $G_N$ are bounded,
\begin{equation*}
\widehat{M}_{N}(t)\;=\;e^{-\lambda t}\,f(Y_{N}(t))\,
-\,f(Y_{N}(0))\,-\,\int_{0}^{t}\,e^{-\lambda r'}\,g(Y_{N}(r'))\,dr'\,
+\,R_{N}(t)\;,
\end{equation*}
where, for all $t\ge 0$,
\begin{equation}
\lim_{N\to\infty} \mb{E}_{\nu_N}^{N}\,
\big[\,R_{N}(t)\,\big]\;=\;0\;.
\label{1e_82}
\end{equation}

Fix $0\le s<t$, $p\ge1$, $0\le s_{1}<s_{2}<\cdots<s_{p}\le s$ and a
bounded measurable functions $h:S^{p}\to\bb{R}$. Let
\begin{gather*}
\mf{M}_{f}^{s,\,t}(Y(\cdot))\;:=\;e^{-\lambda t}f(Y(t))\,-\,e^{-\lambda s}f(Y(s))\,+\,\int_{s}^{t}e^{-\lambda r}\,[\,(\lambda\,-\,L\,)f\,](Y(r))\,dr\;,\\
\mf{H}(Y(\cdot))\;:=\;h(\,Y(s_{1}),\,\dots,\,Y(s_{p})\,)\;,
\end{gather*}
and let $\bb Q^{*}$ be a limit point of the sequence
$\bb Q_{\nu_{N}}^{N}$ satisfying the hypothesis of the proposition. As
$\widehat{M}_{N}(t)$ is a martingale, by \eqref{1e_82},
\begin{equation*}
E_{\bb Q^{*}}\,\big[\,\mf{M}_{f}^{s,\,t}(Y(\cdot))\;
\mf{H}(Y(\cdot))\,\big]\;=\;\lim_{N\rightarrow\infty}
\mb{E}_{\nu_{N}}^{N}\Big[\,\mf{M}_{f}^{s,\,t}(Y_{N}(\cdot))\;
\mf{H}(Y_{N}(\cdot))\,\Big]\;=\;0\;.
\end{equation*}
To complete the proof, it remains to appeal to the uniqueness of
solutions of martingale problems in finite state spaces.
\end{proof}

\section{The conditions (C2) and (C3)}
\label{ns2}

In this section, we present mixing properties of the Markov chain
$\xi^N(\cdot)$ which entail conditions (C2) and (C3). We start with the
former.

Condition (C2) might be easier to prove if the process starts from the
bottom of the well. With this idea in mind, we first prove in Lemma
\ref{nl1} that condition (C2) is fulfilled if it holds for a subset
$\color{bblue} \mc D^x_N$ of $\mathcal{E}_{N}^{x}$, to be interpreted
as the bottom of the well, and if the set $\mc D^x_N$ is attained
before the process hits a new well (the set
$\breve{\mathcal{E}}_{N}^{x}$). Then, in Lemma \ref{nl2}, we propose a
general strategy, based on the construction of a super-harmonic
function, to show that the bottom of the well is attained before the
process hits a new well.

\begin{lem}
\label{nl1}
Suppose that for all $x\in S$, there exists a set $\mathcal{D}_{N}^{x}
\subset \mathcal{E}_{N}^{x}$ such that
\begin{align}
\label{n10}
&\lim\limits _{N\rightarrow\infty}\sup
\limits _{\eta\in\mathcal{E}_{N}^{x}}
\mathbf{P}_{\eta}^{N} [\,
\tau_{\breve{\mathcal{E}}_{N}^{x}}
\,<\, \tau_{\mathcal{D}_{N}^{x}} \, ] \;=\; 0\;,\\
and\;\;
 &\limsup_{a\rightarrow0} \limsup_{N\rightarrow\infty}
\sup_{\eta\in\mathcal{D}_{N}^{x}}
\mathbf{P}_{\eta}^{N}\big[\, \tau_{\breve{\mathcal{E}}_{N}^{x}}
\,<\, a\, \big] \;=\; 0\;.\label{n12}
\end{align}
Then, condition {\rm (C2)} holds.
\end{lem}

\begin{proof}
Fix $x\in S$, $a>0$ and $\eta\in \mathcal{E}_{N}^{x}$. Clearly,
\begin{equation*}
\mathbb{\mathbf{P}}_{\eta}^{N} [\,
\tau_{\breve{\mathcal{E}}_{N}^{x}} \,<\, a\,]
\;\le\;
\mathbb{\mathbf{P}}_{\eta}^{N} [\,
\tau_{\breve{\mathcal{E}}_{N}^{x}} \,<\, a  \,,\,
\tau_{\mathcal{D}_{N}^{x}} \,<\, \tau_{\breve{\mathcal{E}}_{N}^{x}}\,]
\:+\;
\mathbb{\mathbf{P}}_{\eta}^{N} [\,
\tau_{\mathcal{D}_{N}^{x}} \,>\, \tau_{\breve{\mathcal{E}}_{N}^{x}}\,]\;.
\end{equation*}
By the strong Markov property, this expression is bounded by
\begin{equation*}
\sup_{\zeta\in\mathcal{D}_{N}^{x}}\,
\mathbb{\mathbf{P}}_{\zeta}^{N} [\,
\tau_{\breve{\mathcal{E}}_{N}^{x}} \,<\, a\,  ]
\;+\;
\sup_{\eta'\in\mathcal{E}_{N}^{x}}
\mathbb{\mathbf{P}}_{\eta'}^{N} [\,
\tau_{\zeta} \,>\, \tau_{\breve{\mathcal{E}}_{N}^{x}}\,]\;.
\end{equation*}
The assertion of the lemma follows from this bound and the
hypotheses.
\end{proof}

\begin{rem}
In many models, including super-critical zero-range processes, the
condition \eqref{n10} is proved by verifying condition {\rm
\textbf{(H1)}} of \cite{BL1}, reducing the argument to an estimate of
capacities. However, as we have seen in \eqref{n16}, condition {\rm
\textbf{(H1)}} implies that the process visits all points of a well
before it hits a new one, a property which is not observed in many
models, including critical zero-range processes, because the wells are
too large.

Lemma \ref{nl2} below presents a new method to derive \eqref{n10},
based on the construction of a super-harmonic
function.
\end{rem}

Let $\color{bblue} \mc W^x_N$, $x\in S$, be a subset of $\mc H_N$ such
that
$\mc E^x_N \subset \mc W^x_N \subset (\breve{\mathcal{E}}_{N}^{x})^c$,
and recall that $\mc D^x_N \subset \mc E^x_N$. Let
$\partial^- \mc D^x_N$, $\partial^+ \mc W^x_N$ be the inner and outer
boundaries of $\mc D^x_N$, $\mc W^x_N$, respectively:
\begin{equation}
\label{n13}
\begin{gathered}
\partial^- \mc D^x_N \;:=\; \{\, \xi \in \mc D^x_N : \exists \, \eta
\not \in \mc D^x_N \;;\, R_N(\eta, \xi)>0\,\} \;,
\\
\partial^+ \mc W^x_N \;:=\; \{\, \xi \not\in \mc W^x_N : \exists \, \eta
\in \mc W^x_N \;;\, R_N(\eta, \xi)>0\,\} \;.
\end{gathered}
\end{equation}

Recall that a function $F:\mathcal{H}_N \rightarrow \bb{R}$ is said to
be super-harmonic on $\mc{A}\subset\mc{H}_N$ if
$\ms{L}_N F(\eta)\le 0$ for all $\eta\in \mc{A}$.

\begin{lem}
\label{nl2}
Suppose that for each $x\in S$ there exists a positive function
$G^x_N: \mc H_N \to \bb R_+$ which is super-harmonic on
$\mc W^x_N \setminus \mc D^x_N$ and satisfies
\begin{equation}
\label{n14}
\lim_{N\to \infty} \sup_{\eta\in \mc E^x_N\setminus\mathcal{D}_{N}^{x}}
\frac{G^x_N(\eta) \,-\, m_N(x) }{M_N(x) - m_N(x)} \;=\; 0\;,
\end{equation}
where
\begin{equation*}
 m_N(x) = \min_{\eta\in \partial^- \mc
D^x_N} G_N(\eta) \;\;\textup{and}\;\;
  M_N(x) = \min_{\eta \in \partial^+ \mc
W^x_N} G_N(\eta)\;.
\end{equation*}
Then, \eqref{n10} is in force for all $x\in S$.
\end{lem}

\begin{proof}
Fix $x\in S$ and
$\eta\in\mathcal{E}_{N}^{x}\setminus\mathcal{D}_{N}^{x}$.  Let
$\color{bblue}
\tau=\tau_{(\mathcal{W}_{N}^{x}\setminus\mathcal{D}_{N}^{x})^{c}}$.
For every $t>0$,
\begin{equation*}
\mathbf{E}_{\eta}^{N} \Big[\, G_{N}^{x}(\xi_N(\tau\wedge t))
\,-\, G_{N}^{x}(\eta) \,-\,
\int_{0}^{\tau\wedge t}(\mathscr{L}_{N} G_{N}^{x})(\xi_N(s))
\; ds\, \Big] \;=\; 0\;.
\end{equation*}
Hence, since $G^x_N$ is super-harmonic on
$\mc W^x_N \setminus \mc D^x_N$,
\begin{equation*}
\mathbf{E}_{\eta}^{N}\left[\,
G_{N}^{x}(\xi_N(\tau\wedge t))\,\right]
\; \le\;  G_{N}^{x}(\eta) \;,
\end{equation*}
Letting $t\rightarrow\infty$ and since the hitting time $\tau$ is
finite almost surely, by Fatou's lemma,
\begin{equation}
\label{negn1}
\mathbf{E}_{\eta}^{N} [\, G_{N}^{x}(\xi_N(\tau)) \,]
\;\le\;  G_{N}^{x}(\eta)  \;.
\end{equation}

To obtain a lower bound for the last expectation, let us write
\begin{equation*}
p_{N}(\eta) \;=\;
\mathbf{P}_{\eta}^{N} [\, \tau_{(\mathcal{W}_{N}^{x})^{c}}
\,<\, \tau_{\mathcal{D}_{N}^{x}}\,] \;, \quad
\eta\in\mathcal{E}_{N}^{x}\setminus\mathcal{D}_{N}^{x}\;.
\end{equation*}
Then, by definitions of $m_N(x)$ and $M_N(x)$, we have
\begin{equation*}
\mathbf{E}_{\eta}^{N} [\, G_{N}^{x}(\xi(\tau))\,]
\;\ge\; p_{N}(\eta)\, M_N(x)
\;+\; [\, 1-p_{N}(\eta) \,] \, m_N(x) \;.
\end{equation*}
Inserting this bound to  \eqref{negn1} yields that
\begin{align*}
p_{N}(\eta) \; \leq\;
\frac{ G_{N}^{x}(\eta) \,-\, m_N(x) }
{M_N(x) \,-\, m_N(x) }\;, \quad
\eta\in\mathcal{E}_{N}^{x}\setminus\mathcal{D}_{N}^{x}\;.
\end{align*}
Hence, by the hypothesis of lemma,
\begin{equation*}
\lim_{N\rightarrow\infty}\,
\sup_{\eta\in\mathcal{E}_{N}^{x}\setminus\mathcal{D}_{N}^{x}}
\, p_{N}(\eta) \;=\; 0\;.
\end{equation*}

Since \eqref{n10} holds trivially for $\eta\in \mathcal{D}_{N}^{x}$,
and since $\breve{\mathcal{E}}_{N}^{x}  \subset (\mc W^x_N)^c$, the lemma is
proved.
\end{proof}

We turn to condition (C3).  Denote by $\mu_{N}^{\mc E_N}$ the measure
$\mu_N$ conditioned on $\mathcal{E}_{N}$:
\begin{equation}
\label{nff06}
\mu_{N}^{\mc E_N}(\eta) \,=\,
\frac{\mu_{N}(\eta)}{\mu_{N}(\mathcal{E}_{N})}\;\;,
\quad \eta\in\mathcal{E}_{N} \;.
\end{equation}
Let $\{\nu_N : N \ge 1\}$ be a sequence of probability measures on
$\mc H_N$ satisfying condition (C1). Then,
\begin{equation}
\label{nff07}
\mb E_{\mu_{N}^{\mathcal{E}_{N}}}^N
\Big[\, \Big(
\frac{d\nu_{N}}{d\mu_{N}^{\mathcal{E}_{N}}}\Big)^{2}\, \Big]
\;=\;\frac{\mu_N(\mc E_N)}{\mu_N(\mc E^{x_0}_N)}\,
\mb E_{\mu_{N}^{x_0}}^N\Big[\, \Big(
\frac{d\nu_{N}}{d\mu_{N}^{x_0}}\Big)^{2}\, \Big]
\;\le\; C_1\, \frac{\mu_N(\mc E_N)}{\mu_N(\mc E^{x_0}_N)}
\;,
\end{equation}
where $C_1$ is the constant appearing in condition (C1).

Fix $\lambda > 0$ and $f\colon S \to \bb R$. Denote by $F_N$ the
solution of the resolvent equation \eqref{nfg03}, and recall the
definition of the Dirichlet form introduced in \eqref{e022}. The first
result provides elementary estimates of $F_N$.

\begin{lem}
\label{np3}
There exists a finite constant $C_{0} = C_0(f)$ such that
\begin{equation*}
\max_{\eta\in\mc{H}_{N}}|\,{F}_{N}(\eta)\,|\;\le\;C_{0} \,
\lambda^{-1} \;, \quad
\bb D_{N}(F_{N})\;\leq\;C_{0}\, (1+\lambda^{-1})
\end{equation*}
for all $N\ge 1$.
\end{lem}

\begin{proof}
The first bound is obtained by recalling the stochastic representation
of the solution of the resolvent equation (see equation (4.1) in
\cite{LMS2}). We turn to the second. Multiply both sides of
\eqref{nfg03} by $\mu_{N}(\eta) F_{N}(\eta)$ and then sum over
$\eta\in\mc{H}_{N}$ to get
\begin{equation*}
\lambda\,\mb {E}_{\mu_{N}}^N [F_{N}^{2}]\,+\,
\bb D_{N}(F_{N})\;=\;
\left\langle F_{N},\,G_{N}\right\rangle _{\mu_{N}}\;.
\end{equation*}
Since $|G_N|$ is uniformly bounded by $C_0(1+\lambda)$ for some constant $C_0=C_0(f)$, by the first
estimate of the lemma, the right-hand side is less than or equal to
$C_0(1+\lambda^{-1})$, as claimed.
\end{proof}

Let $f_{N}:S\rightarrow\bb{R}$ be
the average of $F_{N}$ on the well
$\mc{E}_{N}^{x}$ with respect to the invariant measure $\mu_{N}$
conditioned to $\mc E^x_N$:
\begin{equation}
f_{N}(x)\;=\;\mb {E}_{\mu_{N}^{x}}^N[F_{N}]
\;=\;\frac{1}{\mu_{N}(\mc{E}_{N}^{x})}\,
\sum_{\eta\in\mc{E}_{N}^{x}}F_{N}(\eta)\,\mu_{N}(\eta)\;.
\label{nfbn}
\end{equation}

Condition (C3) with $f$ replaced by $f_N$ asserts that the solution of
the resolvent equation is close to its average in the $L^1$-sense. The
next result states that this weaker form of condition (C3) follows
from a bound on the variance of $F_N$ in each well.

Define the conditional variance on $\mathcal{E}_{N}^{x}$ of a function
$F:\mathcal{H}_{N}\rightarrow\mathbb{R}$ as
\begin{equation}
\label{ff04}
\text{Var}_{\mu_N^x}(F) \;=\; \mb{E}_{\mu_{N}^x }^N
\Big[\, \big(\, F \,-\, \mb{E}^N_{\mu_{N}^x } (F)
\,\big)^{2}\,\Big]\;.
\end{equation}

\begin{prop}
\label{np2}
Assume that condition {\rm (C1)} holds, and that
\begin{equation}
\label{n6}
\lim_{N\to\infty}
\sum_{x\in S} \frac{\mu_{N}(\mc{E}_{N}^{x})}
{\mu_{N}(\mc{E}^{x_0}_{N})}\,
\text{\rm Var}_{\mu_{N}^{x}}(F_{N}) \;=\; 0 \;.
\end{equation}
Then, for all $t\ge 0$,
\begin{equation}
\label{n7}
\lim_{N\to\infty} \mb{E}_{\nu_{N}}^{N}\left[\,\left|\,
F_{N}(\xi_{N}^{\mc{E}_{N}}(t))
\,-\,f_{N}(Y_{N}(t))\,\right|\,\right]
\;=\; 0 \;.
\end{equation}
\end{prop}

\begin{proof}
Define $\overline{F}_{N}:\mc{E}_{N}\rightarrow\bb{R}$ as
\begin{equation*}
\overline{F}_{N}(\eta)
\;=\;\sum_{x\in S}f_{N}(x)\,\chi_{\mc{E}_{N}^{x}}(\eta)\;.
\end{equation*}
Note that $f_{N}(Y_{N}(t)) = \overline{F}_{N}(\xi_{N}^{\mc{E}_{N}}(t))$.

Recall the definition of the measure $\mu_{N}^{\mc{E}_{N}}$ introduced
in \eqref{nff06}.  Let ${\color{bblue}\nu_{N}(t)}$, $t\ge0$, be the
distribution of $\xi_{N}^{\mc{E}_{N}}(t)$ on $\mc{E}_{N}$ when the
process $\xi_N(\cdot)$ starts from $\nu_{N}$. With these notations, we can write
\begin{equation*}
\mb{E}_{\nu_{N}}^{N}\left[\,\left|\,
F_{N}(\xi_{N}^{\mc{E}_{N}}(t))
\,-\,f_{N}(Y_{N}(t))\,\right|\,\right]
\;=\;\mb{E}_{\ensuremath{\nu_{N}(t)}}^N
\left[\,\left|\,F_{N}(\eta)-\overline{F}_{N}(\eta)\,\right|\,\right]\;.
\end{equation*}
By the Cauchy-Schwarz inequality, the square of the right-hand side
is bounded above by
\begin{equation}
\mb{E}^N_{\mu_{N}^{\mc{E}_{N}}}
\left[\,\left|\,F_{N}-\overline{F}_{N}\right|^{2}\,\right]
\;\mb{E}^N_{\mu_{N}^{\mc{E}_{N}}}
\Big[\,\Big(\frac{d\nu_{N}(t)}{d\mu_{N}^{\mc{E}_{N}}}\Big)^{2}\,\Big]\;.
\label{n4}
\end{equation}
By \cite[Proposition 6.3]{BL2},
$\mu_{N}^{\mc{E}_{N}}(\cdot)=\mu_{N}(\,\cdot\,|\mc{E}_{N})$ is the
invariant measure for the trace process
$\eta_{N}^{\mc{E}_{N}}(\cdot)$.  Let
$h^N_t = d\nu_{N}(t)/d\mu_{N}^{\mc{E}_{N}}$. Since the process is
reversible, by the first displayed equation in \cite[Section 5.2]{kl},
we have $\partial_t h^N_t = \ms L^{\mc E_N}_N h^N_t$, where
$\color{bblue} \ms L^{\mc E_N}_N$ stands for the generator of the
trace process on the set $\mc{E}_N$. In particular,
\begin{equation*}
\frac{d}{dt}\, \mb{E}^N_{\mu_{N}^{\mc{E}_{N}}} \,\big[\,
(h^N_t)^2\,\big] \;=\; 2\, \mb{E}^N_{\mu_{N}^{\mc{E}_{N}}} \,\big[\,
h^N_t \, \ms L^{\mc E_N}_N\, h^N_t \,\big] \;\le\; 0\;.
\end{equation*}
Therefore, by \eqref{nff07},
\begin{equation*}
\mb{E}^N_{\mu_{N}^{\mc{E}_{N}}}
\Big[\,\Big(\frac{d\nu_{N}(t)}{d\mu_{N}^{\mc{E}_{N}}}\Big)^{2}\,\Big]
\;\le\;\mb{E}^N_{\mu_{N}^{\mc{E}_{N}}}
\Big[\,\Big(\frac{d\nu_{N}}{d\mu_{N}^{\mc{E}_{N}}}\Big)^{2}\,\Big]
\;\le\; C_1\, \frac{\mu_N(\mc E_N)}{\mu_N(\mc E^{x_0}_N)} \;\cdot
\end{equation*}

On the other hand, by definition of $\overline{F}_{N}$, and since
$f_{N}(x)$, introduced in \eqref{nfbn}, is the mean of $F_{N}$ with
respect to the measure $\mu_{N}^{x}$, the first expectation in
\eqref{n4} can be written as
\[
\sum_{x\in S}\frac{\mu_{N}(\mc{E}_{N}^{x})}
{\mu_{N}(\mc{E}_{N})}\,\mb{E}^N_{\ensuremath{\mu_{N}^{x}}}
\left[\,\left|\,F_{N}-\overline{F}_{N}\right|^{2}\,\right]
\;=\;
\sum_{x\in S} \frac{\mu_{N}(\mc{E}_{N}^{x})}
{\mu_{N}(\mc{E}_{N})}\,
\text{Var}_{\mu_{N}^{x}}(F_{N})\;.
\]
In the last term, $F_{N}$ is regarded as a function defined only
on $\mc{E}_{N}^{x}$.  The assertion of the proposition follows from
the previous estimates.
\end{proof}

The next result is an elementary consequence of the previous
proposition. It provides a sufficient condition for \eqref{n6} to
hold, stated in terms of a local ergodic property of the dynamics.
More precisely, it asserts that \eqref{n6} is fulfilled provided the
dynamics restricted to each well has a spectral gap of an order
$\epsilon^{-1}_N$, where $\epsilon_N/\mu_{N}(\mc{E}^{x_0}_{N}) \to 0$.

\begin{cor}
\label{nc1}
Assume that there exist a finite constant $C_0$ and a sequence
$\{\epsilon_N : N\ge 1\}$ such that
$\epsilon_N/\mu_{N}(\mc{E}^{x_0}_{N}) \to 0$ and
\begin{equation}
\label{n8}
\mu_{N}(\mc{E}^{x}_{N})\,
\text{\rm Var}_{\mu_N^x}(F) \;\le\; C_0\, \epsilon_N
\, \bb D_{N}(F)
\end{equation}
for all $x\in S$ and $F:\mathcal{H}_{N}\rightarrow\mathbb{R}$.  Then,  \eqref{n6} holds.  In
particular, if condition  {\rm (C1)} and \eqref{n8} hold, then
\eqref{n7} is in force for all $t\ge 0$.
\end{cor}

\begin{proof}
The assertions are simple consequences of Proposition \ref{np2} and
Lemma \ref{np3}.
\end{proof}

To derive condition (C3), it remains to replace $f_N$ by $f$ in
\eqref{n7}. We state this observation in the next result.

\begin{cor}
\label{nc2}
Assume that the hypotheses of Proposition \ref{np2} are fulfilled and
that
\begin{equation}
\label{n15}
\lim_{N\rightarrow\infty}\max_{x\in S}\big|\,f_{N}(x)-f(x)\,\big|\;=\;0\;.
\end{equation}
Then, condition {\rm (C3)} is in force.
\end{cor}

Here is an approach to derive \eqref{n15} in the reversible
case. Consider the bilinear form in $L^{2}(\mu_{N})$ given by
\begin{equation}
\bb
D_{N}(F,\,G)\;=\; \< \, -\ms L_N  F \,,\, G\, \>_{\mu_N} \;.
\label{ndnfg}
\end{equation}

Fix a function $g:S\rightarrow\bb{R}$. We construct a test function,
denoted by $V^{g}:\mc{H}_{N}\rightarrow\bb{R}$, which takes the value
$g(x)$ in $\mc E^x_N$ for each $x\in S$ and such that
\begin{equation}
\label{compu1}
\bb D_{N}(V^{g},\,F_{N})\;=\;D (g,f_{N})\;+\; o_N(1) \;,
\end{equation}
where $D$ is the bilinear Dirichlet form associated to the generator
$L$. We expect this identity to hold because $V^g$, $F_N$ are
close to $g$, $f_N$ on each well and the measure $\mu_N$ is
concentrated on the union of these wells.

On the other hand, since $F_N$ is the solution of the resolvent
equation, $-\ms L_N  F_N = (\lambda - \ms L_N) F_N \,-\, \lambda  F_N
= G_N \,-\, \lambda  F_N$, we get
\begin{equation}
\label{compu2}
\begin{aligned}
\bb D_{N}(V^{g},\,F_{N})\; &=\;
\sum_{x\in S} g(x) \, [(\lambda -  L ) f] (x) \,
\mu_N(\mc E^x_N) \,- \, \lambda \sum_{x\in S} g(x) \,f_N (x) \,
\mu_N(\mc E^x_N) \,+\, o_N(1)
\\
&=\; \lambda \sum_{x\in S} g(x) \,(f-f_{N})(x) \, \mu_N(\mc E^x_N)
\, +\,D(g,\,f)\,+\, o_N(1)  \;.
\end{aligned}
\end{equation}

By combining the two previous estimates, we obtain that
\[
D(g,f-f_{N})\;+\;
\lambda \sum_{x\in S} g(x) \,(f-f_{N})(x) \, \mu_N(\mc E^x_N) \;=\;  o_N(1) \;.
\]
Choosing $g=f-f_{N}$ yields that $D(f-f_{N},f-f_{N})$ is
asymptotically small, from what we conclude that
$\Vert f-f_{N}\Vert_{\infty}=o_{N}(1)$.

In Section \ref{sec6}, we apply this strategy to derive \eqref{n15}
for critical zero-range processes. This approach first appeared in
\cite{RS} in the context of metastable diffusions. It has been
extended to non-reversible models in \cite{LeeSeo}.

\section{Outline of the proof of Theorem \ref{t27}}
\label{sec24}

From this point up to the end of the article, we consider the critical
zero-range process $\xi_N(\cdot)$ whose generator, denoted by
$\ms L_N$, is defined right after \eqref{xin} with
$\theta_N = N^2 \log N$ and $\alpha=1$.

In this section, we highlight the main steps of the proof of the
metastable behavior of $\xi_N(\cdot)$.  Recall the notation introduced
in Section \ref{ns0}.  In view of Theorem \ref{mt0}, to prove Theorem
\ref{t27} we have to show that conditions (C1) -- (C3) are
fulfilled. In Section \ref{sec3}, we prove Theorem \ref{t23}, which
together with the hypothesis of Theorem \ref{t27} on the initial
state, entails Condition (C1).

We turn to condition (C2).  Denote by $\mathcal{D}_{N}^{x}$, $x\in S$,
the deep wells given by
\begin{equation}
\label{dn}
\mathcal{D}_{N}^{x} \;=\;
\left\{\, \eta\in\mathcal{H}_{N}:\eta_{x}\geq N-N^{\gamma}
\, \right\} \;,
\end{equation}
where $0< \gamma <2/\kappa$, and by $\mathcal{W}_{N}^{x}$ the shallow
wells given by
\begin{equation}
\label{wn}
\mathcal{W}_{N}^{x} \;=\;
\Big\{\, \eta\in\mathcal{H}_{N}:
\eta_{x}\geq N-\frac{N}{(\log N)^\beta} \, \Big\}  \;,
\end{equation}
where $0<\beta<1$. Clearly,
\begin{equation*}
\mathcal{D}_{N}^{x}\;\subset\;\mathcal{E}_{N}^{x}
\;\subset\mathcal{W}_{N}^{x}
\;\subset\; \breve{\mathcal{E}}_{N}^{x}\;.
\end{equation*}

The proof of condition (C2), presented in Section \ref{sec7}, is based
on Lemmata \ref{nl1}, \ref{nl2} and carried out in two steps.  First,
in Proposition \ref{p71}, we prove a weaker version of condition (C2),
replacing the initial distribution, concentrated on a configuration,
by the stationary state conditioned on the deeper well
$\mathcal{D}_{N}^{x}$.

Then, in Proposition \ref{p74}, we show that, starting from a
configuration in $\mathcal{E}_{N}^{x}$, the process hits any
configuration $\zeta$ in $\mathcal{D}_{N}^{x}$ before reaching another
well (that is, the set $\breve{\mathcal{E}}_{N}^{x}$). This is a
stronger version of condition \eqref{n10} and is based on the
construction of a super-harmonic function, as indicated in Lemma
\ref{nl2}.  At the end of Section \ref{sec7}, we show that condition
(C2) follows from Propositions \ref{p71} and \ref{p74}. The
super-harmonic function is constructed in Section \ref{sec9}.

We turn to condition (C3). By Theorem \ref{t26} and Corollary
\ref{nc1}, it is enough to prove that \eqref{n8} and \eqref{n15} are
fulfilled for some sequence $\epsilon_N$ which vanishes in the limit.

In Section \ref{sec5}, we prove Theorem \ref{t29} below, which
provides the estimate \eqref{n8}, while \eqref{n15} is the content of
Proposition \ref{p4}.

\begin{thm}
\label{t29}
There exists a finite constant $C_0$ such that
\begin{equation*}
\text{\rm Var}_{\mu_N^x}(F) \;\le\; \frac{C_0}{(\log N)^3}
\, \bb D_{N}(F)
\end{equation*}
for all $x\in S$ and $F:\mathcal{H}_{N}\rightarrow\mathbb{R}$.
\end{thm}

\section{Condensation of Critical Zero-range Process}
\label{sec3}

We prove in this section Theorem \ref{t26}.  From now on, $\alpha=1$
and $\ell_{N}$ is the sequence introduced in \eqref{defln}.

Rewrite the invariant measure $\mu_N$, introduced in \eqref{inv}, as
\begin{equation*}
\mu_{N}(\eta)\;\coloneqq\;
\frac{1}{Z_{N,\,\kappa}} \, \frac{N}{\,(\log N)^{\kappa-1}} \,
\frac{1}{\mathbf{a}(\eta)}\;, \quad \eta\in\mathcal{H}_{N}\;,
\end{equation*}
where the normalizing constant $Z_{N,\,\kappa}$ is given by
\begin{equation*}
Z_{N,\,\kappa} \;\coloneqq\; \frac{N}{(\log N)^{\kappa -1}}
\sum_{\eta\in\mathcal{H}_{N}}\frac{1}{\mathbf{a}(\eta)}\;.
\end{equation*}
Sometimes we denote $Z_{N,|S_{0}|}$ as $Z_{N,S_{0}}$ to stress on
which set the sum is carried out.

\begin{prop}
\label{p31}
We have that
\begin{equation*}
\lim\limits _{N\rightarrow\infty}Z_{N,\,\kappa}=\kappa\;.
\end{equation*}
\end{prop}

\begin{proof}
The proof is carried out by induction in $\kappa$. For $\kappa=2$,
since $N/[j(N-j)] = j^{-1} + (N-j)^{-1}$,
\begin{equation*}
Z_{N,\,2}\;=\frac{N}{\log N}
\bigg( \frac{2}{N} + \sum_{j=1}^{N-1}\frac{1}{j\,(N-j)}\bigg)
\;=\; \frac{2}{\log N} \;+\;
\frac{2}{\log N}\sum_{j=1}^{N-1}\frac{1}{j}\;\cdot
\end{equation*}
The assertion of the proposition follows.

Assume that $\lim\limits _{N\rightarrow\infty}Z_{N,\,\kappa}=\kappa$,
and write $Z_{N,\,\kappa+1}$ as
\begin{equation*}
Z_{N,\,\kappa+1}\;=\;\frac{N}{(\log N)^{\kappa}}\,
\bigg\{\,\frac{1}{N}\;+\;
\sum_{j=0}^{N-1} \frac{[\log (N-j)]^{\kappa-1}}
{\,(N-j)\, {a}(j)}\,Z_{N-j,\,\kappa} \,\bigg\}\;.
\end{equation*}
The first term inside braces is negligible, as well as the term
$j=0$. We divide the remaining ones in four pieces.

Recall from \eqref{defln} the definition of the sequence $\ell_N$. By
the induction hypothesis and the fact that $N-j\simeq N$ for
$j\le\ell_{N}$,
\begin{equation*}
\label{ts1}
\lim_{N\rightarrow\infty}\frac{N}{(\log N)^{\kappa}}\,
\sum_{j=1}^{\ell_{N}} \frac{[\log (N-j)]^{\kappa-1}}{j\,(N-j)}
\,Z_{N-j,\,\kappa} \;=\; \kappa\,
\lim_{N\rightarrow\infty}
\frac{1}{\log N}\sum_{j=1}^{\ell_{N}}\frac{1}{j}\;=\;
\kappa\;.
\end{equation*}

As $(Z_{N,\,\kappa})_{N\ge1}$ is a bounded sequence (because, by the
induction hypothesis, it converges), there exists a finite constant
$C_0$, independent of $N$ and whose value may change from line to
line, such that
\begin{equation*}
\frac{N}{(\log N)^{\kappa}}\,\sum_{j=\ell_{N}+1}^{N/2}
\frac{[\log (N-j)]^{\kappa-1}}{j\,(N-j)}Z_{N-j,\,\kappa}\;\le\;
\frac{C_0}{\log N}\sum_{j=\ell_{N}+1}^{N/2}\frac{1}{j}\;.
\end{equation*}
Since $\log \ell_N/\log N \to 1$, and since
\begin{equation*}
\sum_{j=\ell_{N}+1}^{N/2}\frac{1}{j}
\;=\; [1+o_{N}(1)]\,
\Big(\log\frac{N}{2} \,-\, \log\ell_{N}\Big)\;,
\end{equation*}
we have that
\begin{equation*}
\lim_{N\rightarrow\infty}
\frac{N}{(\log N)^{\kappa}}\,
\sum_{j=\ell_{N}+1}^{N/2}\frac{[\log (N-j)]^{\kappa-1}}
{j\,(N-j)}\,Z_{N-j,\,\kappa}=0\;.
\end{equation*}

We turn to the third term. By a change of variables,
\begin{equation*}
\frac{N}{(\log N)^{\kappa}}\,
\sum_{j=N/2+1}^{N-\ell_{N}-1
}\frac{[\log (N-j)]^{\kappa-1}}{j\,(N-j)}\,Z_{N-j,\,\kappa}
\;=\; \frac{N}{(\log N)^{\kappa}}\,
\sum_{j=\ell_{N}+1}^{N/2-1}\frac{(\log j)^{\kappa-1}}{j\,(N-j)}
\,Z_{j,\,\kappa}\;.
\end{equation*}
This expression is bounded by
\begin{equation*}
\frac{C_0}{\log N}\,\sum_{j=\ell_{N}+1}^{N/2-1}
\frac{1}{j} \;.
\end{equation*}
At this point, we may proceed as for the second term to show that this
expression vanishes as $N\to\infty$.

It remains to consider the sum
\begin{equation*}
\frac{N}{(\log N)^{\kappa}}\,\sum_{j=N-\ell_{N}}^{N-1}
\frac{[\log (N-j)]^{\kappa-1}}{j\,(N-j)}\,Z_{N-j,\,\kappa}
\;=\;
\frac{N}{(\log N)^{\kappa}}\,\sum_{j=1}^{\ell_{N}}
\frac{(\log j)^{\kappa-1}}{j\,(N-j)}\,Z_{j,\,\kappa}\;,
\end{equation*}
where we performed a change of variables.

Let $(m_{N})_{N\ge1}$ be a sequence such that
\begin{equation*}
\lim_{N\rightarrow\infty}m_{N}=\infty\;\;\;
\text{and\;\;\;\ensuremath{\lim_{N\rightarrow\infty}
\frac{\log m_{N}}{\log N}=0\;.}}
\end{equation*}
Since the sequence $(Z_{N,\,\kappa})_{N\ge1}$ is bounded,
\begin{equation*}
\frac{N}{(\log N)^{\kappa}}\,\sum_{j=1}^{m_{N}}
\frac{(\log j)^{\kappa-1}}{j\,(N-j)}\,Z_{j,\,\kappa}
\;\le\; \frac{C_0}{\log N}\sum_{j=1}^{m_{N}}\frac{1}{j}\;.
\end{equation*}
Thus, by the second property of the sequence $m_N$, the left-hand side
of the previous inequality converges to $0$ as $N\rightarrow\infty$.

We turn to the remaining sum. Since $m_{N}\rightarrow\infty$ and
$Z_{N,\,\kappa} \to \kappa$, for $m_{N}+1\le j\le\ell_{N}$,
$Z_{j,\,\kappa}=\kappa\, [1+o_{N}(1)]$. Hence, as $\ell_N/N\to 0$,
\begin{equation*}
\frac{N}{(\log N)^{\kappa}}\,\sum_{j=m_{N}+1}^{\ell_{N}}
\frac{(\log j)^{\kappa-1}}{j\,(N-j)}\,Z_{j,\,\kappa} \;=\;
[1+o_{N}(1)]\, \frac{\kappa}{(\log  N)^{\kappa}}
\sum_{j=m_{N}+1}^{\ell_{N}}\frac{(\log j)^{\kappa-1}}{j}\;\cdot
\end{equation*}
Estimating sums by integrals yields that
\begin{equation*}
\sum_{j=m_{N}+1}^{\ell_{N}}\frac{(\log j)^{\kappa-1}}{j} \;=\;
[1+o_{N}(1)]\,
\frac{(\log \ell_{N})^{\kappa}-(\log m_{N})^{\kappa}}{\kappa}\;.
\end{equation*}
Thus, since $\log \ell_N/\log N \to 1$ and $\log m_N/\log N \to 0$
\begin{equation*}
\lim_{N\to\infty} \frac{N}{(\log N)^{\kappa}}\,\sum_{j=N-\ell_{N}}^{N-1}
\frac{[\log (N-j)]^{\kappa-1}}{j\,(N-j)}\,Z_{N-j,\,\kappa}
\;=\; 1 \;.
\end{equation*}
The assertion of the proposition follows.
\end{proof}

We turn to the

\begin{proof}[Proof of Theorem \ref{t26}]
Fix $x\in S$ and write
\begin{equation*}
\mu_{N}(\mathcal{E}_{N}^{x}) \;=\;
\frac{N}{Z_{N,\,\kappa}\, (\log N)^{\kappa -1}}
\sum_{\eta_{x}\ge N-\ell_{N}}\frac{1}{\mathbf{a}(\eta)}\;,
\end{equation*}
where the sum is performed over all configurations $\eta\in\mc H_N$
such that $\eta_x\ge N-\ell_N$. We may rewrite this sum as
\begin{align*}
\frac{N}{Z_{N,\,\kappa}\, (\log N)^{\kappa -1}}
\sum_{j=N-\ell_{N}}^{N}\frac{1}{j}\,
\frac{[\log (N-j)]^{\kappa-2}}{N-j}Z_{N-j,\,\kappa-1}\;.
\end{align*}

By the last part of the proof of the previous proposition,
\begin{equation*}
\lim_{N\to\infty} \frac{N}{(\log N)^{\kappa -1}}\sum_{j=N-\ell_{N}}^{N}
\frac{1}{j} \, \frac{[\log (N-j)]^{\kappa-2}}{N-j}
\, Z_{N-j,\,\kappa-1} \;=\; 1\;.
\end{equation*}
Therefore, by Proposition \ref{p31},
$\lim_{N\rightarrow\infty}\mu_{N}(\mathcal{E}_{N}^{x})=1/\kappa$.
\end{proof}

The condition $\log \ell_N/\log N \to 1$ is crucial in the previous
proofs. The next result shows that if it does not hold, the measure of
the set $\mc E^x_N$ is no longer close to $1/\kappa$.

Note that the sequence $p_N = N^\delta$ fulfills the conditions of the
next lemma. In particular, in critical zero-range processes the wells
are very large. This is in sharp contrast with super-critical dynamics
in which the wells are formed by configurations in which one
site contains at least $N-m_N$ particles, where $m_N$ is any sequence
such that $m_N\to \infty$, $m_N/N\to 0$.

\begin{lem}
\label{p32}
Let $(p_{N})_{N\ge1}$ be a sequence of positive integers
such that
\begin{equation*}
\lim_{N\rightarrow\infty}p_{N}=\infty\;\;\;
\text{and\;\;\;}\lim_{N\to\infty}\frac{\log p_{N}}{\log N}
\;=\; \delta\,\in\, (0,1]\;.
\end{equation*}
Then, for all $x\in S$,
\begin{equation*}
\lim_{N\to\infty}\mu_{N}\big\{\,\eta_{x}\ge N-p_{N}\,\big\}
\;=\; \frac{1}{\kappa}\, \delta^{\kappa-1}\;.
\end{equation*}
In particular,
\begin{equation*}
\lim_{N\rightarrow\infty}\mu_{N}(\mathcal{D}_{N}^{x})
\;=\; \frac{1}{\kappa}\, \gamma^{\kappa-1}\;,
\end{equation*}
where $\mathcal{D}_{N}^{x}$ is the deep well introduced in \eqref{dn}.
\end{lem}

\begin{proof}
The probability $\mu_{N}\{\eta_{x}\ge N-p_{N}\}$ can be written
as
\begin{equation*}
\frac{1}{Z_{N,\,\kappa}}\,\frac{N}{(\log N)^{\kappa -1}}\,
\bigg(\, \frac{1}{N} \,+\,
\sum_{j=N-p_{N}}^{N-1}\frac{[\log (N-j)]^{\kappa-2}}{j\,(N-j)}\;
Z_{N-j,\,\kappa-1}\, \bigg)\;.
\end{equation*}

At this point, we repeat the steps presented at the end of the proof
of Proposition \ref{p31}. Let $m_N$ be the sequence introduced there
and note that $m_N\ll p_N$ because $\log m_N /\log p_N \to 0$.

According to the proof of Proposition \ref{p31}, in the previous
displayed equation, the sum of the terms $N-m_N\le j\le N-1$ is
negligible, while the sum between $N-p_N$ and $N-m_N$ is equal to
\begin{equation*}
[1+o_{N}(1)] \, \frac{1}{Z_{N,\,\kappa}}\,
\frac{\kappa-1}{(\log N)^{\kappa -1}}\,
\frac{(\log p_{N})^{\kappa-1} \,-\,
(\log m_{N})^{\kappa-1}}{\kappa-1}\;\cdot
\end{equation*}
The result now follows from the properties of the sequences $m_N$ and
$p_N$.
\end{proof}

\section{Proof of Condition (C2)}
\label{sec7}

The proof relies on two results. The first one, Proposition \ref{p71},
provides a weaker version of Condition (C2), in which the initial
condition, a configuration, is replaced by the invariant measure
conditioned to the set $\mc D^x_N$. The second one, Proposition
\ref{p74}, asserts that starting from the well $\mathcal{E}_{N}^{x}$,
the process visits every configuration of the deep well
$\mathcal{D}_{N}^{x}$ before it hits a new well $\mathcal{E}_{N}^{y}$.

Recall from \eqref{Ehat} the definition of
$\breve{\mathcal{E}}_{N}^{x}$. The proof of Proposition \ref{p71} is
based on the enlargement of the zero-range process and requires an
estimate of the capacity
$\Cap_{N}(\mathcal{E}_{N}^{x},\,\breve{\mathcal{E}}_{N}^{x})$.

This estimate is provided in Section \ref{sec71}. In Section
\ref{sec72}, we introduce the enlargement process and present a bound,
in terms of capacities, for the probability that the hitting time of a
set is small. This general result, stated as Proposition \ref{p73},
can be useful in other contexts.

\subsection{Upper bound of the capacity}
\label{sec71}

Recall from \eqref{capac} the definition of the capacity.  The main
result of this subsection reads as follows. Its proof is presented at
the end of Section \ref{sec65}.

\begin{prop}
\label{pp71}
There exists a finite constant $C_0$ such that
for all $x\in S$
\begin{equation*}
\Cap_{N}(\mathcal{E}_{N}^{x},\,\breve{\mathcal{E}}_{N}^{x})
\;\le\; C_0\;.
\end{equation*}
\end{prop}

\begin{rem}
Although we do not provide the detailed proof here, we can compute the
sharp asymptotics for the capacity and show that
\begin{equation*}
  \Cap_{N}(\mathcal{E}_{N}^{x},\,\breve{\mathcal{E}}_{N}^{x})
\;=\; \big[\, 1+o_{N}(1) \,\big] \, \frac{1}{\kappa}
\, \sum_{y\in S\setminus\{x\}} r_{Z}(x,\,y)\;.
\end{equation*}
\end{rem}

\subsection{The enlarged process}
\label{sec72}

Recall the definition of the sets $\mc D^x_N$, introduced in
\eqref{dn}. Denote by $\pi_{N}^{x}$ the measure $\mu_{N}$ conditioned
on $\mathcal{D}_{N}^{x}$:
\begin{equation}
\label{pi}
\pi_{N}^{x}(\eta)=\frac{\mu_{N}(\eta)}{\mu_{N}(\mathcal{D}_{N}^{x})}
\;, \quad \eta\in\mathcal{D}_{N}^{x}\;.
\end{equation}
The main result of this subsection reads as follows.

\begin{prop}
\label{p71}
For all $x\in S$,
\begin{equation*}
\limsup_{a\rightarrow0} \limsup_{N\rightarrow\infty}
\mathbf{P}_{\pi_{N}^{x}}^{N} \big[\,
\tau_{\breve{\mathcal{E}}_{N}^{x}}
<a  \, \big] \;=\; 0\;.
\end{equation*}
\end{prop}

The proof of this proposition is based on the next result which
provides a bound for the transition time in terms of the initial
distribution and the capacity. This result is a modification of
\cite[Corollary 4.2]{BL4}.

\begin{prop}
\label{p73}
For every $x\in S$, probability measure $\upsilon_{N}$ concentrated on
the set $\mathcal{E}_{N}^{x}$, $\gamma_N >0$ and $N\ge 1$,
\begin{equation*}
\Big( \mathbb{\mathbf{P}}_{\upsilon_{N}}^{N}
\Big[\, \tau_{\breve{\mathcal{E}}_{N}^{x}} \;\le\;
\frac{1}{\gamma_{N}} \, \Big]\, \Big)^{2} \;\le\;
\frac{e^{2}}{\gamma_{N}}\, {E}_{\mu_{N}^{x}}\Big[\,
\Big(\frac{\upsilon_{N}}{\mu_{N}^{x}}\Big)^{2}\, \Big]\,
\frac{1}{\mu_{N}(\mathcal{E}_{N}^{x})} \,
\Cap_{N} (\mathcal{E}_{N}^{x},\,\breve{\mathcal{E}}_{N}^{x})\;.
\end{equation*}
\end{prop}

\begin{proof}[Proof of Proposition \ref{p71}]
By the definition \eqref{pi} of $\pi_{N}^{x}$,
\begin{equation*}
\mathbb{E}_{\mu_{N}^{x}}\Big[\,
\Big(\frac{\pi_{N}^{x}}{\mu_{N}^{x}}\Big)^{2} \, \Big]
\;=\; \sum_{\eta\in\mathcal{E}_{N}^{x}}
\frac{\pi_{N}^{x}(\eta)^{2}}{\mu_{N}^{x}(\eta)}
\;=\; \frac{\mu_{N}(\mathcal{E}_{N}^{x})}{\mu_{N}(\mathcal{D}_{N}^{x})}\;\cdot
\end{equation*}
Therefore, by Proposition \ref{p73} with $\upsilon_{N}=\pi_{N}^{x}$
and $\gamma_{N}^{-1} = a $,
\begin{equation*}
\Big( \mathbb{\mathbf{P}}_{\pi_{N}^{x}}^{N}
\big[\tau_{\breve{\mathcal{E}}_{N}^{x}}\le a
 \,\big]\Big)^{2}
\;\le\; \frac{e^{2}\, a}{\mu_{N}(\mathcal{D}_{N}^{x})}\,
\Cap_N(\mathcal{E}_{N}^{x},\,\breve{\mathcal{E}}_{N}^{x})\;.
\end{equation*}
By Propositions \ref{p32} and \ref{pp71}, there exists a finite
constant $C(\gamma)$, where $\gamma$ is the parameter appearing in the
definition of the set $\mc D^x_N$, such that
\begin{equation*}
\Big( \mathbb{\mathbf{P}}_{\pi_{N}^{x}}^{N}
\big[\tau_{\breve{\mathcal{E}}_{N}^{x}}\le a
 \,\big]\Big)^{2} \;\le\;  C(\gamma)\, a\;.
\end{equation*}
This completes the proof.
\end{proof}

Besides Proposition \ref{p73}, the main ingredients of the proof were
the strictly positive lower bound for $\mu_{N}(\mathcal{D}_{N}^{x})$
and the upper bound for the capacity.

We turn to the proof of Proposition \ref{p73} which relies on an
enlargement of the state space, introduced in \cite[Section
2]{BL4}. Denote by
$\color{bblue} \mathcal{R}_{N}^{\mathcal{E}_{N}}:
\mathcal{E}_{N}\times\mathcal{E}_{N}\rightarrow[0,\,\infty)$ the jump
rates of the trace process
$\{\eta_{N}^{\mathcal{E}_{N}}(t)\}_{t\ge0}$ .

Let $\color{bblue} \mathcal{E}_{N}^{\star}$ be a copy of
$\mathcal{E}_{N}$, and denote by
$\color{bblue} \eta^{\star}\in\mathcal{E}_{N}^{\star}$ the copy of
$\eta\in\mathcal{E}_{N}$.

\begin{defn}[Enlarged process]
\label{def72}
Fix $N\ge 1$ and $\gamma_{N}>0$.  The \textit{$\gamma_{N}$-enlarged
  process} $\{\eta_{N}^{\star}(t)\}_{t\ge0}$ is the continuous-time
Markov process on $\mathcal{E}_{N}\cup\mathcal{E}_{N}^{\star}$ whose
jump rates
$\mathcal{R}_{N}^{\star}:\mathcal{E}_{N}\cup\mathcal{E}_{N}^{\star}
\times\mathcal{E}_{N}\cup\mathcal{E}_{N}^{\star}\rightarrow[0,\,\infty)$
are given by
\begin{equation*}
\mathcal{R}_{N}^{\star}(\eta,\,\zeta) \;=\;
\begin{cases}
\mathcal{R}_{N}^{\mathcal{E}_N}(\eta,\,\zeta) &
\text{if }\text{\ensuremath{\eta,\,\zeta\in\mathcal{E}_{N}}}\;,\\
\gamma_{N} &
\text{if }\zeta=\eta^{\star}\text{ or }\eta=\zeta^{\star}\;,\\
0 & \text{otherwise}\;.
\end{cases}
\end{equation*}
\end{defn}

Namely, the process $\eta_{N}^{\star}(t)$ at
$\eta^{\star}\in\mathcal{E}_{N}^{\star}$ only jumps to $\eta$ at rate
$\gamma_{N}$, while at $\eta\in\mathcal{E}_{N}$ it jumps to other
points of $\mathcal{E}_{N}$ as in the original dynamics of the trace
process, and it jumps to $\eta^{\star}$ at rate $\gamma_{N}$.

The invariant measure for the $\gamma_{N}$-enlarged process
$\eta_{N}^{\star}(t)$ is given by
\begin{equation*}
\mu_{N}^{\star}(\eta) \;=\; \mu_{N}^{\star}(\eta^{\star})
\;=\; \frac{1}{2}\, \mu_{N}(\eta)\text{\;\;\;for all }\eta\in\mathcal{E}_{N}\;.
\end{equation*}
Actually, the process $\eta_{N}^{\star}(\cdot)$ is reversible with
respect to this measure.

Denote by $\color{bblue} \text{cap}_{N}^{\star}(\mc A \,,\, \mc B)$
the capacity between two disjoint, nonempty subsets $\mc A$, $\mc B$
of $\mathcal{E}_{N}\cup\mathcal{E}_{N}^{\star}$, defined in a same
manner as \eqref{capac}.

\begin{proof}[Proof of Proposition \ref{p73}]
Denote by $\mathbf{P}_{\upsilon_{N}}^{N,\,\mathcal{E}}$ the law of the
trace process $\xi^{\mc E_N}(t)$ on $\mathcal{E}_{N}$ starting from
the measure $\upsilon_{N}$. In view of \cite[Corollary
4.2]{BL4}, to prove the proposition, it is enough to show that
\begin{gather*}
\mathbf{P}_{\upsilon_{N}}^{N} \Big[\,
\tau_{\breve{\mathcal{E}}_{N}^{x}} \,\le\,
\frac{1}{\gamma_{N}}\, \Big]
\;\le\; \mathbf{P}_{\upsilon_{N}}^{N,\,\mathcal{E}}
\Big[\, \tau_{\breve{\mathcal{E}}_{N}^{x}}
\,\le\, \frac{1}{\gamma_{N}}\,\Big]\;, \\
\Cap_{N}^{\star} (\, \mathcal{E}_{N}^{\star,\,x}
\,,\,\breve{\mathcal{E}}_{N}^{x}\,)
\;\le\; \frac{1}{2\mu_{N}(\mathcal{E}_{N})}\,
\Cap_{N} (\, \mathcal{E}_{N}^{x}\,,\,\breve{\mathcal{E}}_{N}^{x}\,)\;,
\end{gather*}
where $\color{bblue} \mathcal{E}_{N}^{\star,\,x}$,
$\color{bblue} \mathcal{\breve{E}}_{N}^{\star,\,x}$ represent the
copies of $\mathcal{E}_{N}^{x}$, $\mathcal{\breve{E}}_{N}^{x}$,
respectively.

The first estimate holds because the trace process hits the set
$\breve{\mathcal{E}}_{N}^{x}$ before the original process, as the
later one may spend some time on $\Delta_{N}$.

We turn to the second estimate. By \cite[Lemma 2.2]{GL}, the
capacity is monotone, so that
\begin{equation*}
\Cap^{\star}_{N} (\,\mathcal{E}_{N}^{\star,\,x}
\,,\,\breve{\mathcal{E}}_{N}^{x}\,) \;\le\;
\Cap^{\star}_{N} (\, \mathcal{E}_{N}^{\star,\,x}\cup\mathcal{E}_{N}^{x}
\,,\,\breve{\mathcal{E}}_{N}^{\star,\,x}
\cup\breve{\mathcal{E}}_{N}^{x}\,)\;.
\end{equation*}
Denote by
$\chi^\star_x =
\chi_{\mathcal{\breve{E}}_{N}^{\star,\,x}\cup\mathcal{\breve{E}}_{N}^{x}}
:\mathcal{E}_{N}\cup\mathcal{E}_{N}^{\star}\rightarrow\mathbb{R}$ the
indicator function of the set
$\mathcal{\breve{E}}_{N}^{\star,\,x}\cup\mathcal{\breve{E}}_{N}^{x}$.
Since $\chi^\star_x$ is the equilibrium potential between the sets
$\mathcal{E}_{N}^{\star,\,x}\cup\mathcal{E}_{N}^{x}$ and
$\mathcal{\breve{E}}_{N}^{\star,\,x}\cup\mathcal{\breve{E}}_{N}^{x}$
for the $\gamma_{N}$-enlarged process, the right-hand side of the
previous displayed equation is equal to
$\bb {D}_{N}^{\star}(\chi^\star_x)$, where
$\bb {D}_{N}^{\star}$ represents the Dirichlet form associated to
the $\gamma_{N}$-enlarged process.

By definition of the enlarged process, in the computation of the
Dirichlet form of the indicator function $\chi^\star_x$ the only terms
which do not vanish are those which correspond to jumps between
$\mathcal{E}_{N}^{x}$ and $\breve{\mathcal{E}}_{N}^{x}$. Hence,
\begin{equation*}
\bb {D}_{N}^{\star}(\chi^\star_x) \;=\;
\sum_{\eta\in\mathcal{E}_{N}^{x},\,\zeta\in\breve{\mathcal{E}}_{N}^{x}}
\mu^\star_{N}(\eta) \, \mathcal{R}_{N}^{\star}(\eta,\,\zeta)\,
\big[\, \chi^\star_x (\zeta) \,-\, \chi^\star_x (\eta)\, \big]^{2}\;.
\end{equation*}
By definition of $\mu^\star_N$, $\mathcal{R}_{N}^{\star}$ and
$\chi^\star_x$, this sum is equal to
\begin{equation*}
\bb {D}_{N}^{\star}(\chi^\star_x) \;=\;
\frac{1}{2} \sum_{\eta\in\mathcal{E}_{N}^{x},\,\zeta\in\breve{\mathcal{E}}_{N}^{x}}
\mu_{N}(\eta) \, \mathcal{R}_{N}^{\mathcal{E}} (\eta,\,\zeta)\,
\big[\, \chi_{\mathcal{\breve{E}}_{N}^{x}}(\zeta)
\,-\, \chi_{\mathcal{\breve{E}}_{N}^{x}}(\eta) \, \big]^{2}\;.
\end{equation*}

Denote by $\bb {D}_{N}^{\mathcal{E}_{N}}$ the Dirichlet form
associated to the trace process. The previous sum is equal to
$(1/2)\, \bb {D}_{N}^{\mathcal{E}_N }(\chi_{\breve{\mathcal{E}}_{N}^{x}})$.
Since $\chi_{\breve{\mathcal{E}}_{N}^{x}}$
is the equilibrium potential between $\mathcal{E}_{N}^{x}$ and
$\breve{\mathcal{E}}_{N}^{x}$ for the trace process,
\begin{equation*}
\frac{1}{2} \, \bb D_{N}^{\mathcal{E}_{N}}(\chi_{\breve{\mathcal{E}}_{N}^{x}})
\;=\; \frac{1}{2} \, \Cap_{N}^{\mathcal{E}_{N}}(\mathcal{E}_{N}^{x}
\,,\,\mathcal{\breve{E}}_{N}^{x})\;,
\end{equation*}
where $\Cap_{N}^{\mathcal{E}_{N}}$ stands for the capacity for the
trace process. By \cite[Lemma 6.9]{BL1},
\begin{equation*}
\frac{1}{2} \, \Cap_{N}^{\mathcal{E}_{N}}(\mathcal{E}_{N}^{x}
\,,\,\mathcal{\breve{E}}_{N}^{x}) \;=\;
\frac{1}{2\,\mu(\mathcal{E}_{N})}\,
\Cap_{N}(\mathcal{E}_{N}^{x}\,,\,\breve{\mathcal{E}}_{N}^{x})\;,
\end{equation*}
which completes the proof of the proposition.
\end{proof}

\subsection{Visiting points}
\label{sec73}

The next result asserts that, starting from the well
$\mathcal{E}_{N}^{x}$, the process visits every configuration of the
deep well $\mathcal{D}_{N}^{x}$ before it hits a new well
$\mathcal{E}_{N}^{y}$. Its proof, presented in Section \ref{sec8},
relies on the construction of a super-harmonic function on
$\mathcal{W}_{N}^{x}\setminus\mathcal{D}_{N}^{x}$, carried out in
Section \ref{sec9}.

\begin{prop}
\label{p74}
For each $x\in S$,
\begin{equation*}
\lim\limits _{N\rightarrow\infty}
\inf_{\zeta\in\mathcal{D}_{N}^{x}}\,
\inf\limits _{\eta\in\mathcal{E}_{N}^{x}}
\, \mathbf{P}_{\eta}^{N} [\,
\tau_{\zeta} \,<\, \tau_{\check{\mathcal{E}}_{N}^{x}}\,]
\;=\; 1\;.
\end{equation*}
\end{prop}

\begin{proof}[Proof of Condition {\rm (C2)}]
In view of Lemma \ref{nl1} and Proposition \ref{p74}, it is enough to
show that condition \eqref{n12} is fulfilled.  Fix $x\in S$, $a>0$ and
$\eta$, $\zeta\in \mathcal{D}_{N}^{x}$. Clearly,
\begin{equation*}
\mathbb{\mathbf{P}}_{\eta}^{N} [\,
\tau_{\breve{\mathcal{E}}_{N}^{x}} \,<\, a\,]
\;\le\;
\mathbb{\mathbf{P}}_{\eta}^{N} [\,
\tau_{\breve{\mathcal{E}}_{N}^{x}} \,<\, a  \,,\,
\tau_{\zeta} \,<\, \tau_{\breve{\mathcal{E}}_{N}^{x}}\,]
\:+\;
\mathbb{\mathbf{P}}_{\eta}^{N} [\,
\tau_{\zeta} \,>\, \tau_{\breve{\mathcal{E}}_{N}^{x}}\,]\;.
\end{equation*}
By the strong Markov property, this expression is bounded by
\begin{equation*}
\mathbb{\mathbf{P}}_{\zeta}^{N} [\,
\tau_{\breve{\mathcal{E}}_{N}^{x}} \,<\, a\,  ]
\;+\;
\sup_{\eta' , \zeta' \in\mathcal{D}_{N}^{x}}
\mathbb{\mathbf{P}}_{\eta'}^{N} [\,
\tau_{\zeta'} \,>\, \tau_{\breve{\mathcal{E}}_{N}^{x}}\,]\;.
\end{equation*}
Multiplying both sides by $\pi_{N}^{x}(\zeta)$ and summing over
$\zeta\in\mathcal{D}_{N}^{x}$ yields that
\begin{equation*}
\mathbb{\mathbf{P}}_{\eta}^{N} [\,
\tau_{\breve{\mathcal{E}}_{N}^{x}} \,<\, a\, ]
\;\le\;
\mathbb{\mathbf{P}}_{\pi_{N}^{x}}^{N} [\,
\tau_{\breve{\mathcal{E}}_{N}^{x}} \,<\, a\, ]
\;+\;
\sup_{\eta' , \zeta' \in\mathcal{D}_{N}^{x}}
\mathbb{\mathbf{P}}_{\eta'}^{N} [\,
\tau_{\zeta'} \,>\, \tau_{\breve{\mathcal{E}}_{N}^{x}}\,] \;.
\end{equation*}
Hence, condition \eqref{n12} follows from Propositions \ref{p71} and
\ref{p74}.
\end{proof}

\section{Local Spectral Gap}
\label{sec5}

In this section, we prove Theorem \ref{t29}. Fix $x_{0}\in S$, and let
$\color{bblue} S_{0}=S\setminus\{x_{0}\}$.

\subsection{Restricted process}
\label{sec51}

For $x\in S$, let
\begin{equation}
\label{wenx}
\widehat{\mathcal{E}}_{N}^{x} \;=\;
\big\{\, \eta\in\mathcal{H}_{N}:
\eta_{y} \, \le\, \ell_{N}\text{ for all }y\in S\setminus\{x\}
\, \big\}\;.
\end{equation}
Thus, $\mathcal{E}_{N}^{x}\subset\widehat{\mathcal{E}}_{N}^{x}$.

The zero-range process (without acceleration) \textit{restricted} to
$\widehat{\mathcal{E}}_{N}^{x_{0}}$ is the
$\widehat{\mathcal{E}}_{N}^{x_{0}}$-valued dynamics obtained by
removing all jumps from $\widehat{\mathcal{E}}_{N}^{x_{0}}$ to its
complement.

The generator of this process, denoted by $\mathcal{L}_{N}^{x_{0}}$,
is given by
\begin{equation*}
(\mathcal{\mathscr{L}}_{N}^{x_{0}}F) (\eta)
\;=\;   \sum_{z,\,w\in S} g(\eta_{z}) \, r(z,\,w) \,
\big[\, F(\sigma^{z,\,w}\eta) - F(\eta) \,\big]
\mathbf{1}\{\sigma^{z,\,w}\eta\in\widehat{\mathcal{E}}_{N}^{x_{0}}\}\;,
\end{equation*}
for $F:\widehat{\mathcal{E}}_{N}^{x_{0}}\rightarrow\mathbb{R}$. Denote
by $\color{bblue} \eta^{x_0}_N (t)$ the Markov chain associated to the
generator $\mathcal{L}_{N}^{x_{0}}$.

Let
\begin{equation*}
\widehat\mu_{N}^{x_{0}}(\eta) \;=\; \frac{\mu_{N}(\eta)}
{\mu_{N}(\widehat{\mathcal{E}}_{N}^{x_{0}})} \;,
\quad \eta\in \widehat{\mathcal{E}}_{N}^{x_{0}}
\end{equation*}
be the probability measure obtained by conditioning the invariant
measure $\mu_N$ to the set
$\widehat{\mathcal{E}}_{N}^{x_{0}}$. As $\mu_N$, this measure fulfills
the detailed balance conditions. In particular, it is invariant.

The Dirichlet form associated to the restricted process
$\eta_{N}^{x_{0}}(t)$, denoted by
$\mathcal{\mathcal{\bb D}}_{N}^{x_{0}}$, is given by,
\begin{equation*}
\mathcal{\mathcal{\bb D}}_{N}^{x_{0}}(F) \;=\;
\frac{1}{2}\sum_{z,\,w\in S}\;
\sum_{\eta,\,\sigma^{z,\,w}\eta\in\widehat{\mathcal{E}}_{N}^{x_{0}}}
\widehat\mu_{N}^{x_{0}}(\eta)\, g(\eta_{z})\, r(z,\,w)\,
\big[\, F(\sigma^{z,\,w}\eta)-F(\eta) \,\big]^{2}\;,
\end{equation*}
for $F:\widehat{\mathcal{E}}_{N}^{x_{0}}\rightarrow\mathbb{R}$,

Denote by $\textrm{Var}_{\widehat\mu_N^{x_{0}}}(F)$ the variance of a function
$F:\widehat{\mathcal{E}}_{N}^{x_{0}}\rightarrow\mathbb{R}$ with
respect to the measure $\widehat\mu_{N}^{x_{0}}(\cdot)$:
\begin{equation*}
\textrm{Var}_{\widehat\mu_N^{x_{0}}}(F) \;=\;
\textrm{E}_{\widehat\mu_{N}^{x_{0}}} \Big[\, \big( F \,-\,
\textrm{E}_{\widehat\mu_{N}^{x_{0}}}[\, F\,]\, \big)^{2}\,\Big]\;.
\end{equation*}
The next result establishes a lower bound for the spectral gap of the
generator $\mathcal{\mathscr{L}}_{N}^{x_{0}}$.

\begin{thm}
\label{ts52}
There exists a finite constant $C_0>0$ such that, for all
$F:\widehat{\mathcal{E}}_{N}^{x_{0}}\rightarrow\mathbb{R}$,
\begin{equation*}
\textup{Var}_{\widehat\mu_N^{x_{0}}}(F) \;\le\; C_0\, \ell_N^2
\,\mathcal{\mathcal{\bb D}}_{N}^{x_{0}}(F)\;.
\end{equation*}
\end{thm}

The proof of the local spectral gap is based on an idea presented in
\cite[Section 4]{AGL2}. It consists in comparing the restricted process
with a collection of independent birth-and-death dynamics whose
spectral gap is of order $\ell_{N}^{-2}.$

\subsection{Proof of Theorem \ref{t29}}
\label{sec52}

The argument relies on the next result.

\begin{lem}
\label{lem53}
We have that $\mu_{N}(\widehat{\mathcal{E}}_{N}^{x_{0}}) \,=\,
[\,1+o_{N}(1)\,] \, (1/\kappa)$.
\end{lem}

\begin{proof}
Let $\widehat{\mathcal{F}}_{N}^{x_{0}} \,=\, \{\eta\in\mathcal{H}_{N}:
\eta_{x_0}\ge N-(\kappa-1)\ell_{N}\}$
so that
\begin{equation*}
\mathcal{E}_{N}^{x_{0}} \;\subset\; \widehat{\mathcal{E}}_{N}^{x_{0}}
\;\subset\; \widehat{\mathcal{F}}_{N}^{x_{0}}\;.
\end{equation*}
By the proof of Theorem \ref{t26},
$\mu_{N}(\widehat{\mathcal{F}}_{N}^{x_{0}}) \,=\, [\, 1+o_{N}(1)\, ]
\, (1/\kappa)$.  The assertion of the lemma follows from this
observation and  Theorem \ref{t26}.
\end{proof}

\begin{proof}[Proof of Theorem \ref{t29}]
Fix $F:\mathcal{H}_{N}\rightarrow\mathbb{R}$. Since
${\mathcal{\bb D}}_{N}^{x_{0}}(F) \, \le\,
\mathcal{\bb D}_{N}(F) / \theta_N\,
\mu_{N}(\widehat{\mathcal{E}}_{N}^{x_{0}})$ and since
$\theta_N^{-1}\ell_N^2=(\log N)^{-3}$, it suffices to show that there
exists a finite constant $C_0$ such that
\begin{equation*}
\textrm{Var}_{\mu_N^{x_0}}(G) \;\le\; C_0\, \textup{Var}_{\widehat\mu_N^{x_{0}}}(G)\;
\end{equation*}
for all functions $G:\mathcal{H}_{N}\rightarrow\mathbb{R}$

Write
$\overline{G} = \sum_{\zeta\in\widehat{\mathcal{E}}_{N}^{x_{0}}}
\widehat\mu_{N}^{x_{0}} (\zeta) \, G(\zeta)$.  Then,
\begin{align*}
\textrm{Var}_{\mu_N^{x_0}}(G)  \;& =\;
\min_{c\in\mathbb{R}}\, \frac{1}{\mu_{N}(\mathcal{E}_{N}^{x_{0}})} \,
\sum_{\eta\in\mathcal{E}_{N}^{x_{0}}} \,
\mu_{N}(\eta)\, \, [G(\eta)-c \,]^{2}\\
\; & \le\; \frac{1}{\mu_{N}(\mathcal{E}_{N}^{x_{0}})}\,
\sum_{\eta\in\mathcal{E}_{N}^{x_{0}}}\,
\mu_{N}(\eta)\, \, [G(\eta)-\overline{G}\,]^{2}\;.
\end{align*}
Since
$\mathcal{E}_{N}^{x_{0}} \subset \mathcal{\widehat{E}}_{N}^{x_{0}}$,
this expression is bounded by
\begin{equation*}
\frac{\mu_{N}(\mathcal{\widehat{E}}_{N}^{x_{0}})}
{\mu_{N}(\mathcal{E}_{N}^{x_{0}})} \,
\sum_{\eta\in\mathcal{\widehat{E}}_{N}^{x_{0}}}
\widehat\mu_{N}^{x_{0}}(\eta)\left[G(\eta)-\overline{G}\right]^{2}
\;=\; [\, 1+o_{N}(1)\,]\, \textup{Var}_{\widehat\mu_N^{x_{0}}}(G)\;,
\end{equation*}
where the last identity follows from Theorem \ref{t26} and Lemma
\ref{lem53}.
\end{proof}

\subsection{A birth-and-death process}
\label{sec53}

Consider a birth-and-death process $\{w(t)\}_{t\ge0}$ on
$\color{bblue}
\mathbb{X}=\mathbb{X}_{N}=\{0,\,1,\,\cdots,\,\ell_{N}\}$ with jump
rates given by
\begin{equation*}
R(i,\,j)=\begin{cases}
1 & \text{if }j=i+1 \;\text{ and }\; j\le \ell_{N}\;,\\
{g}(i) & \text{if }j=i-1\, \;\text{ and }\; j\ge 0 \;,\\
0 & \text{otherwise .}
\end{cases}
\end{equation*}

The invariant probability measure, denoted by
$\varphi(\cdot)=\varphi_N(\cdot)$, is given by
\begin{equation}
\label{ms-1}
\varphi(k)=\frac{1}{z_{N}} \, \frac{1}{{a}(k)} \;, \quad k\in \mathbb{X}\,,
\end{equation}
where $z_{N}$ is the normalizing constant satisfying
\begin{equation}
\label{ms-2}
z_{N} \;=\; \sum_{k=0}^{\ell_{N}} \frac{1}{{a}(k)}
\;=\; [\, 1+o_{N}(1) \,] \, \log N\;.
\end{equation}
The process is actually reversible with respect to $\varphi(\cdot)$.

Consider independent, birth-and-death processes $\zeta_{x}(t)$,
$x\in S_{0}$, each one having the same law as $w(\cdot)$.  Denote by
$\zeta(t)$ the continuous-time Markov chain on $\mathbb{X}^{S_{0}}$ given by
$\color{bblue} \zeta(t)=(\zeta_{x}(t))_{x\in S\setminus\{x_{0}\}}$.

Here and below, elements of $\mathbb{X}^{S_0}$ are represented by
$\omega=(\omega_{x})_{x\in S\setminus\{x_{0}\}}$. For each $x\in S_{0}$,
let $\mathfrak{d}^{x}\in \mathbb{X}^{S_0}$ be the configuration consisting of
only one particle at site $x$:
\begin{equation*}
(\mathfrak{d}^{x})_{y}=\mathbf{1}\{x=y\}\;, \quad y\in S\setminus\{x_{0}\}\;.
\end{equation*}

The next assertions about the process $\zeta(t)$ are elementary.  The
invariant measure is the product measure $\varphi^{S_{0}}(\cdot)$, defined
by
\begin{equation*}
\varphi^{S_{0}}(\omega)=\prod_{x\in S_{0}}\varphi(\omega_{x})\;, \quad \omega\in \mathbb{X}^{S_0}\;.
\end{equation*}
Actually, $\zeta(\cdot)$ is reversible with respect to $\varphi^{S_{0}}$.

The generator of the process $\zeta(\cdot)$, denoted by
$\mathcal{\mathscr{L}}_{N}^{\textrm{BDP}}$, is given by
\begin{equation*}
\begin{aligned}
(\mathcal{\mathscr{L}}_{N}^{\textrm{BDP}}G)(\omega) \;= & \;
\sum_{x\in S_{0}} \left[\, G(\omega+\mathfrak{d}^{x})-G(\omega) \,\right]
\, \mathbf{1}\{\omega_{x}+1\in \mathbb{X}\}  \\
\; + &\; \sum_{x\in S_{0}} {g}(\omega_{x}) \,
\left[\, G(\omega-\mathfrak{d}^{x})-G(\omega)\, \right]\,
\mathbf{1}\{\omega_{x}-1\in \mathbb{X}\}\;,
\end{aligned}
\end{equation*}
for $G:\mathbb{X}^{S_0}\rightarrow\mathbb{R}$, and the
Dirichlet form by
\begin{equation*}
\mathcal{\mathcal{\bb D}}_{N}^{\textrm{BDP}}(G)
\;=\; \frac{1}{2} \, \sum_{x\in S_{0}} \sum_{\omega\in \mathbb{X}^{S_0}}
\varphi^{S_{0}}(\omega) \, \left[\, G(\omega+\mathfrak{d}^{x})-G(\omega)\,
\right]^{2} \, \mathbf{1}\{\omega_{x}+1\in \mathbb{X}\}\;.
\end{equation*}

Denote by $\textrm{Var}_{N}^{\textrm{BDP}}(G)$ the variance of
$G:\mathbb{X}^{S_0}\rightarrow\mathbb{R}$:
\begin{equation*}
\textrm{Var}_{N}^{\textrm{BDP}}(G) \;=\;
\textup{E}_{\varphi^{S_{0}}}\left[\, \left(\, G-
\textup{E}_{\varphi^{S_{0}}}\left[G\right]\, \right)^{2}\,
\right]\;.
\end{equation*}

The next result is \cite[Theorem 1.2]{CSC}.  The lower bound is sharp.
It can be shown that there exists constants $0<C_1<C_2<\infty$ such
that
$C_{1}\ell_{N}^{-2}\le\lambda^{\textrm{BDP}}_N\le C_{2}\ell_{N}^{-2}$,
where $\lambda^{\textrm{BDP}}_N$ represents the spectral gap of the
generator $\mathcal{\mathscr{L}}_{N}^{\textrm{BDP}}$.  We provide a
simple proof of Proposition \ref{p55} based on the Efron-Stein
inequality.

\begin{prop}
\label{p55}
There exists a  finite constant $C_0$ such that
\begin{equation*}
\textrm{ \rm Var}_{N}^{\textup{BDP}}(F) \;\le\;
C_0\, \ell_{N}^{2} \, \bb D_{N}^{\textup{BDP}}(F)
\end{equation*}
for all $N\ge 1$, $F:\mathbb{X}^{S_0}\rightarrow\mathbb{R}$.
\end{prop}

The next result is \cite[Theorem 6, page 214]{blb04} and follows from
the Efron-Stein inequality \cite{ef81}.

\begin{lem}
\label{l24}
Let $X_{1},\,X_{2},\,\cdots,\,X_{n}$ be independent random variables,
and let $f: \bb R^{n}\rightarrow\mathbb{R}$,
$f_{1},\,f_{2},\,\cdots,\,f_{n}: \bb R^{n-1}\rightarrow\mathbb{R}$ be
measurable, bounded functions. Define the random variables
\begin{gather*}
Z \;=\; f(X_{1},\,X_{2},\,\cdots,\,X_{n}) \;, \\
Z_{i} \;=\;
f_{i}(X_{1},\,\cdots,\,X_{i-1},\,X_{i+1},\,\cdots,\,X_{n})\;,
\quad 1\le i\le n \;.
\end{gather*}
Then,
\begin{equation*}
\textup{Var}\,(Z) \;\le\;\sum_{i=1}^{n}
\,\textup{E}\left[(Z-Z_{i})^{2}\right]\;.
\end{equation*}
\end{lem}

The proof below is similar to the one of \cite[Lemma 4.4]{AGL2}.

\begin{proof}[Proof of Proposition \ref{p55}]
For $\omega\in \mathbb{X}^{S_0}$ and $x\in S_{0}$, denote by $\omega^{x,\,k}$ the
configuration obtained from $\omega$ by replacing $\omega_{x}$ with $k$:
\begin{equation*}
(\omega^{x,\,k})_{y}=\begin{cases}
\omega_{y} & \text{if }y\neq x\\
k & \text{if }y=x\;.
\end{cases}
\end{equation*}
Observe that $G(\omega^{x,\,0})$, $x\in S_{0}$, is a function of
$\omega_{y}$, $y\neq x$. Hence, by Lemma \ref{l24},
\begin{equation*}
\textrm{Var}_{N}^{\textrm{BDP}}(G) \;\le\;
\sum_{x\in S_{0}}\sum_{\omega\in \mathbb{X}^{S_0}}
\varphi^{S_{0}}(\omega)\left[G(\omega)-G(\omega^{x,\,0})\right]^{2}\;.
\end{equation*}
By the Cauchy-Schwarz inequality,
\begin{align*}
\varphi^{S_{0}}(\omega)\, \left[G(\omega)-G(\omega^{x,\,0})\right]^{2} \;\le\;
\varphi^{S_{0}}(\omega)\,
\omega_{x}\sum_{k=0}^{\omega_{x}-1}\left[G(\omega^{x,\,k+1})-G(\omega^{x,\,k})\right]^{2}\;.
\end{align*}
Since $k\le {a}(k)$ and $\varphi^{S_{0}}(\omega)\, {a}(\omega_{x}) \;=\;
\varphi^{S_{0}}(\omega^{x,\,k})\, {a}(k)$, the previous expression is less than
or equal to
\begin{equation*}
\sum_{k=0}^{\omega_{x}-1} \varphi^{S_{0}}(\omega^{x,\,k})\, {a}(k)
\, \left[G(\omega^{x,\,k}+\mathfrak{d}^{x})-G(\omega^{x,\,k})\right]^{2}\;.
\end{equation*}

Up to this point we proved that
\begin{equation*}
\textrm{Var}_{N}^{\textrm{BDP}}(G) \; \le\;
\sum_{x\in S_{0}}\sum_{\omega\in \mathbb{X}^{S_0}}\sum_{k=0}^{\omega_{x}-1}
\varphi^{S_{0}}(\omega^{x,\,k})\, {a}(k) \,
\left[G(\omega^{x,\,k}+\mathfrak{d}^{x})-G(\omega^{x,\,k})\right]^{2}\;.
\end{equation*}
Changing variables $\zeta= \omega^{x,\,k}$, yields that this sum is equal to
\begin{equation*}
\sum_{x\in S_{0}}\sum_{\zeta \in \mathbb{X}^{S_0}} \varphi^{S_{0}}(\zeta)\,
{a}(\zeta_{x})\, (\ell_{N}-\zeta_{x}) \,
\left[G(\zeta+\mathfrak{d}^{x})-G(\zeta)\right]^{2}
\, \mathbf{1}\{\zeta_{x}+1\in \mathbb{X}\}\;.
\end{equation*}
To complete the proof, it remains to observe that
${a}(\zeta_{x})\, (\ell_{N}-\zeta_{x})\le\ell_{N}^{2}$.
\end{proof}

\subsection{Proof of Theorem \ref{ts52}}
\label{secs13}

For $N$ sufficiently large, $N\ge (\kappa -1) \ell_N$, there exists a
natural bijection between $\mathbb{X}^{S_0}$ and
$\widehat{\mathcal{E}}_{N}^{x_{0}}$ given by
\begin{equation}
\label{bijxi}
\omega\in \mathbb{X}^{S_0} \; \longleftrightarrow \; \widetilde{\omega}
\;=\; (N-|\omega|,\,\omega)\in\widehat{\mathcal{E}}_{N}^{x_{0}}\;,
\end{equation}
where $(N-|\omega|,\,\omega)\in\mathcal{H}_{N}$ represents the configuration
with $N-|\omega|$ particles at the site $x_{0}$, and $\omega_{x}$ particles
at the site $x\in S_{0}$. Therefore, we can identify a function
$G:\mathbb{X}^{S_0}\rightarrow\mathbb{R}$ with
$\widetilde{G}:\widehat{\mathcal{E}}_{N}^{x_{0}}\rightarrow\mathbb{R}$
by
\begin{equation}
\label{bijG}
\widetilde{G}(\widetilde{\omega}) \;=\; G(\omega)\;.
\end{equation}
The map $G\leftrightarrow\widetilde{G}$ is a bijection between the
space of real-valued functions on $\mathbb{X}^{S_0}$ and on
$\widehat{\mathcal{E}}_{N}^{x_{0}}$.

\begin{prop}
\label{sp27}
There exists a finite constant $C_0$ such that,
\begin{equation*}
\textrm{\rm Var}_{\widehat\mu_N^{x_{0}}}(\widetilde{G}) \;\le\;
C_0 \, \textrm{\rm Var}_{N}^{\textup{BDP}}(G)
\end{equation*}
for all $N$ such that $N \ge (\kappa -1) \ell_N$ and
$G:\mathbb{X}^{S_0}\rightarrow\mathbb{R}$.
\end{prop}

\begin{proof}
We first claim that there exists a finite constant $C_0$ such that
\begin{equation}
\label{bmmu}
\widehat\mu_{N}^{x_{0}}(\widetilde{\omega}) \;\le\; C_0\, \varphi^{S_{0}}(\omega)
\quad\text{for all } N\in\mathbb{N}
\text{ and } \widetilde{\omega}\in\widehat{\mathcal{E}}_{N}^{x_{0}}\;.
\end{equation}

Indeed, since $|\omega|\le\ell_{N}$, by Proposition \ref{p31} and Lemma \ref{lem53},
\begin{align*}
\widehat \mu_{N}^{x_{0}}(\widetilde{\omega}) \; & =\;
\frac{1}{\mu_N(\widehat{\mathcal{E}}_N^{x_0}) } \,  \frac{N}{Z_{N,\,S}(\log N)^{\kappa -1}}\, \frac{1}{(N-|\omega|)}
\, \frac{1}{\mathbf{a}(\omega)}\\
\:& = \; [\, 1+o_{N}(1))\,]\,
\prod_{x\in S_{0}}\frac{1}{\log N}\, \frac{1}{{a}(\omega_{x})} \;\cdot
\end{align*}
At this point, \eqref{bmmu} follows from  \eqref{ms-1} and
\eqref{ms-2}.

Fix $G:\mathbb{X}^{S_0}\rightarrow\mathbb{R}$. Since the expectation
minimizes the square distance,
\begin{equation*}
\textup{Var}_{\widehat\mu_N^{x_{0}}}(\widetilde{G})
\;\le\; \sum_{\widetilde{\omega}\in\widehat{\mathcal{E}}_{N}^{x_{0}}}
\big(\, \widetilde{G}(\widetilde{\omega})
\,-\, \textup{E}_{\varphi^{S_{0}}}\left[G\right] \big)^{2}\,
\,\widehat\mu_{N}^{x_{0}}(\widetilde{\omega})\;.
\end{equation*}
By \eqref{bijxi}, \eqref{bijG}, and \eqref{bmmu}, this expression is bounded from above by
\begin{equation*}
C_0 \sum_{ {\xi}\in  \mathbb{X}^{S_0}}
\big(\, G(\xi)- \textup{E}_{\varphi^{S_{0}}}\left[G\right] \big)^{2}\,
\varphi^{S_{0}}(\xi) \;=\;C_0 \, \textrm{Var}_{\varphi^{S_{0}}}(G)
\;=\;C_0 \, \textrm{Var}_{N}^{\textup{BDP}}(G)\;,
\end{equation*}
which completes the proof of the proposition.
\end{proof}

\begin{prop}
\label{sp28}
There exists a finite constant $C_0$ such that
\begin{equation*}
\mathcal{\mathcal{\bb D}}_{N}^{\textup{BDP}}(G)
\;\le\; C_0\, \mathcal{\mathcal{\bb D}}_{N}^{x_{0}}(\widetilde{G})
\end{equation*}
for all $N$ such that $N \ge (\kappa -1) \ell_N$ and
$G:\mathbb{X}^{S_0}\rightarrow\mathbb{R}$.
\end{prop}

The proof of this result relies on a technical lemma. We say that two
configurations $\eta$, $\eta'$ are neighbors if
$\eta'=\sigma^{x,\,y}\eta$ for some $x,\,y\in S$ with $r(x,\,y)>0$.

\begin{lem}
\label{sl19}
For all $\eta\in\widehat{\mathcal{E}}_{N}^{x_{0}}$ and $x\in S_0$ such
that $\sigma^{x_{0},\,x}\eta \in \widehat{\mathcal{E}}_{N}^{x_{0}}$,
there is a path
$\mathfrak{s}(\eta,\,x) = (\eta^{(0)} =\eta \,,\, \eta^{(1)} \,,\,
\dots \,,\,\eta^{(m)} = \sigma^{x_{0},\,x}\eta)$ in
$\widehat{\mathcal{E}}_{N}^{x_{0}}$ from $\eta$ to
$\sigma^{x_{0},\,x}\eta$ such that
\begin{enumerate}
\item $m \le\kappa$
\item $\eta^{(i)}$ and $\eta^{(i+1)}$ are neighbors for all $0\le i <m$,
\item $\mu_{N} (\eta)\le 4\, \mu_{N} (\eta^{(i)})$ for
all $0\le i \le m$,
\item Each pair $(\eta',\,\eta'')$ of neighboring configurations
appears as a consecutive pair in no more than $2\kappa^{4}$
paths $\mathfrak{s}(\eta,\,x)$.
\end{enumerate}
\end{lem}

\begin{proof}
Fix $x\in S_0$. As the random walk is irreducible, there exists
$m<\kappa$ and a sequence
\begin{equation*}
x_{0} \,=\, v_{0} \,,\,v_{1} \,,\,\cdots \,,\,v_{m} \,=\, x
\end{equation*}
such that $r(v_{k},\,v_{k+1})>0$ for all $0\le k < m$. This sequence
depends only on $x$. It is fixed and will be the same for all
configurations $\eta\in\widehat{\mathcal{E}}_{N}^{x_{0}}$.

Fix $\eta\in\widehat{\mathcal{E}}_{N}^{x_{0}}$ such that
$\sigma^{x_{0},\,x}\eta \in \widehat{\mathcal{E}}_{N}^{x_{0}}$.  The
natural definition of the path $\mathfrak{s}(\eta,\,x)$ is to set
${\eta}^{(k)}=\sigma^{v_{0},\,v_{k}}\eta$. However, if
$\eta_{v_{k}} = \ell_{N}$ for some $k$, this path leaves the set
$\widehat{\mathcal{E}}_{N}^{x_{0}}$, which is not permitted.  We
modify the natural path to keep it in the set
$\widehat{\mathcal{E}}_{N}^{x_{0}}$.

Note that $\eta_{v_m}<\ell_N$ because $v_m =x$ and
$\sigma^{x_{0},\,x}\eta \in \widehat{\mathcal{E}}_{N}^{x_{0}}$.  If
$\eta_{v_k} <\ell_N$ for $1\le k < m$, the path $\mf s(\eta,x)$ is the
one above.

If this is not the case, let $p$ be the first integer such that
$\eta_{v_k}=\ell_N$:
\begin{equation*}
p \;=\; \min \big\{\, 1\le k\le m : \eta_{v_k} = \ell_N\,\big\}\;.
\end{equation*}
Let $q\ge p$ be the last one with the property that all sites in
between are occupied by $\ell_N$ particles:
\begin{equation*}
q \;=\; \max \big\{\, p\le k\le m : \eta_{v_j} = \ell_N\,,\, p\le j\le
k \,\big\}\;.
\end{equation*}
Note that $q<m$ because $\eta_{v_m}<\ell_N$ and that
$\eta_{v_{q+1}}<\ell_N$.

The path is constructed as follows. We first move a particle from
$x_0=v_0$ to $v_1$, then we move it from $v_1$ to $v_2$, until we reach
$v_{p-1}$. At this point, we may not move it to $v_p$. To remove a
particle from $v_p$, we move a particle from $v_{q}$ to $v_{q+1}$,
then from $v_{q-1}$ to $v_{q}$, until we move one from $v_{p}$ to
$v_{p+1}$. At this point we move a particle from $v_{p-1}$ to $v_p$.

Up to this point, a particle has been displaced from $x_0=v_0$ to
$v_{q+1}$. If all sites between $v_{q+2}$ and $v_m$ have less than
$\ell_N$ particles, we continue to move the particle up to the
end. Otherwise, we repeat the surgery to avoid leaving the set
$\widehat{\mathcal{E}}_{N}^{x_{0}}$. This defines the path $\mf
s(\eta,x)$.

Note that the path $\mf s(\eta,x)$ does not visit the same
configuration twice: $\eta^{(i)} \not = \eta^{(j)}$ for $i\not = j$.

It is clear that the conditions (1) and (2) are fulfilled. By
definition of the path, for each $1\le k\le m$, there exists
$w_1, \dots, w_4$ [which depend on $k$, naturally], such that
$\eta^{(k)} = \sigma^{x_0,w_1} \eta$ or
$\eta^{(k)} = \sigma^{x_0,w_2} \sigma^{w_3,w_4} \eta$.
Since, for every $x\not = y$,
\begin{equation*}
\frac{\mu_{N}(\eta)}{\mu_{N}(\sigma^{y,\,x}\eta)}
\;=\; \frac{{a}(\eta_{x}+1)}{{a}(\eta_{x})} \,
\frac{{a}(\eta_{y}-1)}{{a}(\eta_{y})}\le 2\;,
\end{equation*}
condition (3) is proved.

We turn to (4). Suppose that a pair
$(\eta',\,\eta'' = \sigma^{u,\, v} \eta')$ appears in the path
$\mathfrak{s}(\eta,\,x)$ for some $\eta$ and $x$. Then, as we have
seen above, either $\eta' = \sigma^{x_0,w_1} \eta$ or
$\eta' = \sigma^{x_0,w_2} \sigma^{w_3,w_4}\eta$ for some
$w_1, \dots, w_4$. Hence, either $\eta = \sigma^{w_1,x_0} \eta'$ or
$\eta = \sigma^{w_4,w_3}\sigma^{w_2,x_0}\eta'$. Therefore, there are
at most $\kappa + \kappa^3 \le 2\kappa^3$ possible configurations
$\eta$ and $\kappa$ possible choices for $x$, making the total number
of possible pairs $(\eta, x)$ in which neighbors $(\eta,\eta')$
appear to be bounded by $2\kappa^4$.

Since a pair $(\eta',\,\eta'')$ of neighbor configurations appears
only once in a path $\mathfrak{s}(\eta,\,x)$, there are at most
$2\kappa^4$ different paths in which a fixed pair $(\eta',\,\eta'')$
may appear. This completes the proof of the lemma.
\end{proof}

\begin{proof}[Proof of Proposition \ref{sp28}]
Note that the bijection $\omega\leftrightarrow\widetilde{\omega}$ given in
\eqref{bijxi} satisfies
$\omega+\mathfrak{d}^{x}\leftrightarrow\sigma^{x_{0},\,x}\widetilde{\omega}$.
Thus, we can write
$\mathcal{\mathcal{\bb D}}_{N}^{\textrm{BDP}}(G)$ as
\begin{align*}
\mathcal{\mathcal{\bb D}}_{N}^{\textrm{BDP}}(G)
\; & =\; \frac{1}{2} \sum_{x\in S_{0}}\sum_{\omega\in \mathbb{X}^{S_0}}
\varphi^{S_{0}}(\omega) \, \left[G(\omega+\mathfrak{d}^{x})-G(\omega)\right]^{2}
\, \mathbf{1}\{\omega+\mathfrak{d}^{x}\in \mathbb{X}^{S_0}\} \\
\; & =\; \frac{1}{2} \sum_{x\in S_{0}}\sum_{\widetilde{\omega}\in\widehat{\mathcal{E}}_{N}^{x_{0}}}
\varphi^{S_{0}}(\omega) \,
\big[\, \widetilde{G}(\sigma^{x_{0},\,x}\widetilde{\omega}\,)
-\widetilde{G}(\widetilde{\omega}\,)\,\big]^{2}
\mathbf{1}\{\sigma^{x_{0},\,x}\widetilde{\omega}\in\widehat{\mathcal{E}}_{N}^{x_{0}}\}\;.
\end{align*}
By \eqref{bmmu} and since the map $\omega\leftrightarrow\widetilde{\omega}$
is bijection, it follows from the previous equation that there exists
a finite constant $C_0$, independent of $N$, such that
\begin{equation}
\label{penul}
\mathcal{\mathcal{\bb D}}_{N}^{\textrm{BDP}}(G)
\;\le\; C_0 \sum_{x\in S_{0}}
\sum_{\eta\in\widehat{\mathcal{E}}_{N}^{x_0}}
\widehat\mu_{N}^{x_{0}}(\eta) \, \big[\, \widetilde{G}(\sigma^{x_{0},\,x}\eta)
-\widetilde{G}(\eta)\,\big]^{2}
\, \mathbf{1}\{\sigma^{x_{0},\,x}\eta\in\widehat{\mathcal{E}}_{N}^{x_0}\}\;.
\end{equation}

Recall from Lemma \ref{sl19} the definition of the path
$\mathfrak{s}(\eta,\,x)=(\eta^{(0)},\,\dots,\,\eta^{(m)})$ for
$\eta\in\widehat{\mathcal{E}}_{N}^{x_0}$ and $x\in S_{0}$ such that
$\sigma^{x_{0},\,x}\eta\in\widehat{\mathcal{E}}_{N}^{x_0}$.  By the
Cauchy-Schwarz inequality and conditions (1) and (3) of that lemma,
\begin{align*}
\widehat\mu_{N}^{x_{0}}(\eta)\, \big[\, \widetilde{G}(\sigma^{x_{0},\,x}\eta)
-\widetilde{G}(\eta) \,\big]^{2}
\;& \le\; m\sum_{k=0}^{m-1}\widehat\mu_{N}^{x_{0}}(\eta)\,
\big[\, \widetilde{G}(\eta^{(k+1)})-\widetilde{G}(\eta^{(k)})\,\big]^{2}\\
\; & \le \; 4\, \kappa \sum_{k=0}^{m-1} \widehat\mu_{N}^{x_{0}}(\eta^{(k)})
\, \big[\, \widetilde{G}(\eta^{(k+1)})-\widetilde{G}(\eta^{(k)})\,\big]^{2}\;.
\end{align*}

Inserting this bound in \eqref{penul}, changing the order of
summations and applying part (4) of Lemma \ref{sl19} yield that
$\mathcal{\mathcal{\bb D}}_{N}^{\textrm{BDP}}(G)$ is bounded
above by
\begin{align*}
C_0(\kappa) \, \sum_{(x,y)}\,
\sum_{\eta\in\widehat{\mathcal{E}}_{N}^{x_{0}}}
\widehat\mu_{N}^{x_{0}}(\eta) \, \big[\, \widetilde{G}(\sigma^{x,\,y}\eta)
-\widetilde{G}(\eta)\,\big]^{2}\,
\mathbf{1} \{ \sigma^{x,\,y}\eta\in\widehat{\mathcal{E}}_{N}^{x_{0}} \} \;,
\end{align*}
where the first sum is carried over all pairs $(x,\,y)$ such that
$r(x,\,y)>0$.  The last summation is bounded above by
$C_0(\kappa) \,
\mathcal{\mathcal{\bb D}}_{N}^{x_{0}}(\widetilde{G})$ because
${g}(k)\ge 1$ for all $k\ge1$.
\end{proof}

\begin{proof}[Proof of Theorem \ref{ts52}]
Fix $F:\mathcal{\widehat{E}}_{N}^{x_0}\rightarrow\mathbb{R}$. By the
bijection introduced in Subsection \ref{secs13}, there exists
$G:\mathbb{X}^{S_0}\rightarrow\mathbb{R}$ such that $F=\widetilde{G}$ in the
sense of \eqref{bijG}. By Propositions \ref{p55}, \ref{sp27}, and
\ref{sp28}, there exists a finite constant $C_0 = C_0(\kappa)$,
independent of $N$, such that
\begin{equation*}
\textup{Var}_{\widehat\mu_N^{x_{0}}}(\widetilde{G}) \;\le\;  C_0\,
\textrm{Var}_{N}^{\textrm{BDP}}(G) \;\le\;
C_0\, \ell_{N}^{2}\, \mathcal{\mathcal{\bb D}}_{N}^{\textrm{BDP}}(G)
\;\le\; C_0\, \ell_{N}^{2}\,
\mathcal{\mathcal{\bb D}}_{N}^{x_{0}}(\widetilde{G})\;.
\end{equation*}
This proves Theorem \ref{ts52}.
\end{proof}

\section{The Resolvent Equation}
\label{sec6}

In this section, we prove \eqref{n15} for critical zero-range
processes. Fix $\lambda>0$ and a function $f:S \to \bb R$. Let
$F_{N}:\mathcal{H}_{N}\rightarrow\mathbb{R}$ be the solution of the
resolvent equation \eqref{nfg03}, and let $f_N \colon S \to\bb R$ be
given by \eqref{nfbn}.

\begin{prop}
\label{p4}
For all $\lambda > 0$, and $f:S\to \bb R$,
\[
\lim_{N\rightarrow\infty}\max_{x\in S}\big|\,f_{N}(x)-f(x)\,\big|\;=\;0\;.
\]
\end{prop}

\subsection{Sketch of the proof}

The proof follows the strategy presented below Corollary \ref{nc2}.
In the context of zero-range processes, the bilinear
$\bb D_{N}(F,\,G)$ introduced in \eqref{ndnfg} takes the form in
$L^{2}(\mu_{N})$ given by
\begin{equation}
\label{dnfg}
\bb D_{N}(F,\,G)\;=\;\frac{\theta_N}{2}\,
\sum_{\eta\in\mc{H}_{N}}\sum_{x,\,y\in S}
\mu_{N}(\eta)\,g(\eta_{x})\,r(x,\,y)\,(T_{x,y}F)(\eta)\,
(T_{x,y}G)(\eta)\;,
\end{equation}
where ${\color{bblue}(T_{x,y}H)(\eta)=H(\sigma^{x,\,y}\eta)-H(\eta)}$.

In Section \ref{sec64}, we present a partition of the set
$\mc{H}_{N}$.  The idea behind this construction is that in the
computation of the form $\bb D_{N}(F,\,G)$, for Lipschitz functions
$F$, $G$, in the sense of Lemma \ref{lem010}, only a tiny subspace of
$\mc{H}_{N}$, formed by the wells $\mc{E}_{N}^{x}$ and tubes
connecting them, matters.  In Section \ref{sec65}, we construct the
test function $V^{g}$ and show that it is Lipschitz. In the last
sections, we prove the two estimates \eqref{compu1} and \eqref{compu2}
on the bilinear form $\bb D_{N}(V^{g},\,F_{N})$ and Proposition
\ref{p4}.

\subsection{\label{sec62}Energy estimate}

We prove in this section a simple bound needed in the proof of
Proposition \ref{p4}.  Recall from \eqref{e022} that we denote by
$\langle\,\cdot\,,\,\cdot\,\rangle_{\mu_{N}}$ the scalar product in
$L^{2}(\mu_{N})$.  With this notation, we can write the Dirichlet form
as
\[
\bb D_{N}(F)\;=\;
\big\langle F,\,-\mathscr{L}_{N} F\big\rangle _{\mu_{N}}\;.
\]

\begin{lem}
\label{lemb} There exists a finite constant $C_{0}=C_{0}(\kappa)$
such that
\[
\sum_{\eta\in\mc{H}_{N}}\mu_{N}(\eta)\,
\left[F(\sigma^{x,\,y}\eta)-F(\eta)\right]^{2}
\;\le\;C_{0}\, \theta_N^{-1} \,\bb D_{N}(F)
\]
for all $x,\,y\in S$ and $F:\mc{H}_{N}\rightarrow\bb{R}$.
\end{lem}

\begin{proof}
Suppose first that $r(x,\,y)>0$. Then,
\[
\mu_{N}(\eta)\left[F(\sigma^{x,\,y}\eta)-F(\eta)\right]^{2}\;\le\;\frac{1}{r(x,\,y)}\,\mu_{N}(\eta)\,g(\eta_{x})\,r(x,\,y)\,\left[F(\sigma^{x,\,y}\eta)-F(\eta)\right]^{2}
\]
because ${g}(\eta_{x})\ge1$ if $\eta_{x}\ge1$, and both sides are
$0$ when $\eta_{x}=0$. Summing this over $\eta\in\mc{H}_{N}$
yields the assertion of the lemma.

If $r(x,\,y)=0$, by the irreducibility of the Markov chain $X$,
there exists a sequence $x=z_{0},\,z_{1},\,\cdots,\,z_{m}=y$ such
that $r(z_{i},\,z_{i+1})>0$ for $0\le i<m$. Hence, by the Cauchy-Schwarz
inequality
\[
\left[F_{N}(\sigma^{x,\,y}\eta)-F_{N}(\eta)\right]^{2}\le m\sum_{i=0}^{m-1}\left[F_{N}(\sigma^{z_{0},\,z_{i+1}}\eta)-F_{N}(\sigma^{z_{0},\,z_{i}}\eta)\right]^{2}\;.
\]
Applying the previous argument to each term at the right-hand side
completes the proof since there exists a finite constant $C_{0}$
such that
\begin{equation*}
\mu_{N}(\sigma^{z,\,w}\eta)\le C_{0}\,\mu_{N}(\eta)
\end{equation*}
for all $z,\,w\in S$, $\eta\in\mc{H}_{N}$, $N\in\bb{N}$.
\end{proof}

\subsection{Tubes and wells\label{sec64}}

Fix $\epsilon>0$ small. The subsets of $\mc{H}_{N}$ constructed
in this section may depend on $N$ and $\epsilon$, even if these
parameters do not appear in the notation. We also refer to Figure
\ref{fig1} for the illustration of the sets described in this subsection.

\begin{figure}
\protect \includegraphics[scale=0.21]{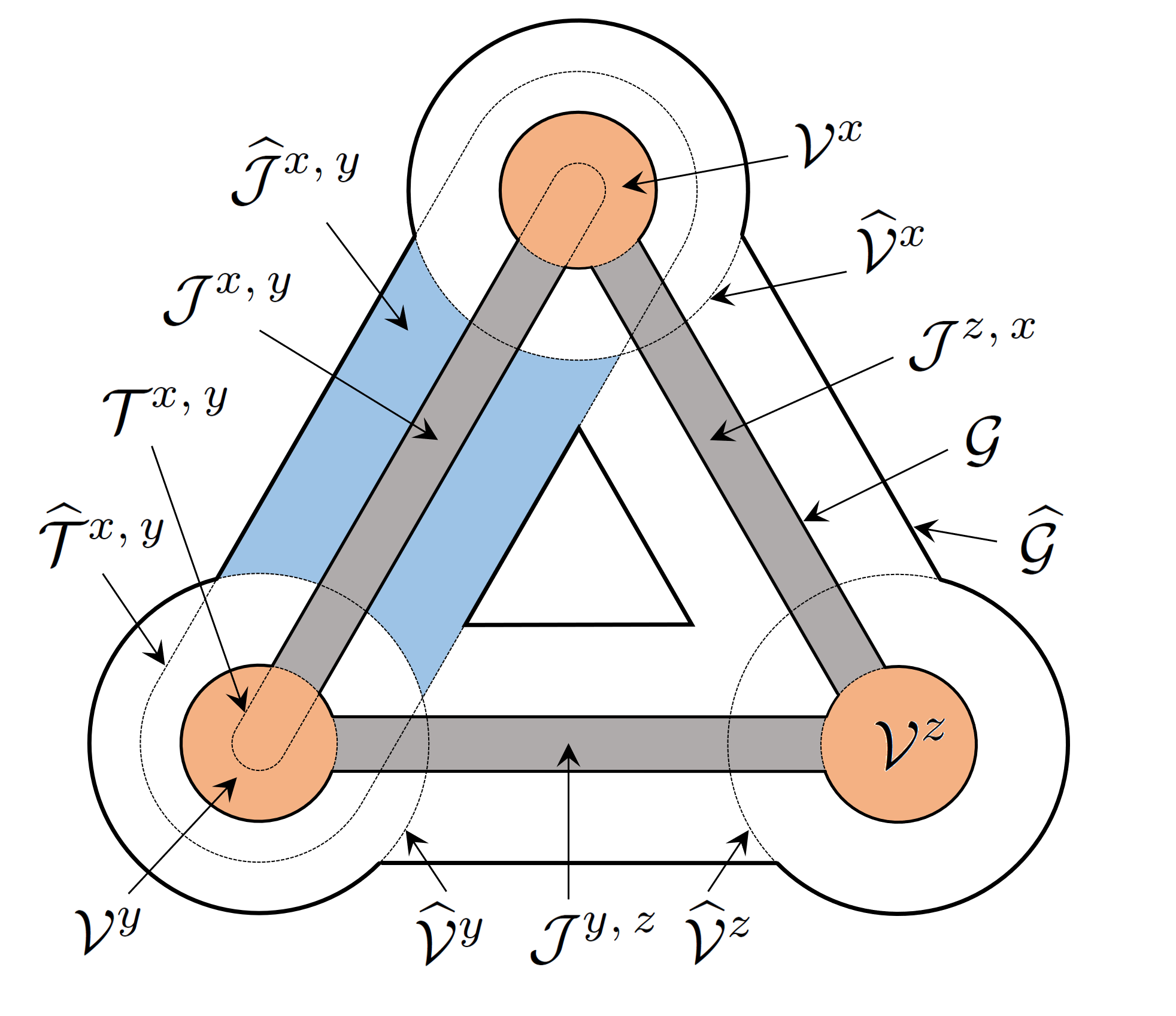}1\protect
\caption{ An illustration of sets introduced in Section \ref{sec64} when $S=\{x,\,y,\,z\}$.
We can notice from this figure that the sets $\mc{J}^{x,\,y}$,
$\mc{J}^{y,\,z}$, and $\mc{J}^{z,\,x}$ are disjoint. \label{fig1}}
\end{figure}

Define the enlarged wells $\mc{V}^{x}$, $\widehat{\mc{V}}^{x}$,
$x\in S$, by
\begin{align*}
\mc{V}^{x}=\{\eta\in\mc{H}_{N}:\eta_{x}\ge N(1-2\epsilon)\}\;,\quad\widehat{\mc{V}}^{x}=\{\eta\in\mc{H}_{N}:\eta_{x}\ge N(1-4\epsilon)\}\;.
\end{align*}
From now on, all the statements may hold only for large enough $N$.
More precisely, there exists a constant $N_{0}=N_{0}(\epsilon)$ which
is independent of $\eta$ such that the statement holds only for $N>N_{0}$.
With this convention, $\mc{E}{}_{N}^{x}\subset\mc{V}^{x}\subset\widehat{\mc{V}}^{x}$.

For $x,\,y\in S$, the tubes $\mc{T}^{x,\,y}$ and $\widehat{\mc{T}}^{x,\,y}$
connecting the wells $\mc{E}_{N}^{x}$ and $\mc{E}_{N}^{y}$
are defined by
\begin{gather*}
\mc{T}^{x,\,y}\;=\;\{\eta\in\mc{H}_{N}:\eta_{x}+\eta_{y}\ge N-\ell_{N}\}\;,\\
\widehat{\mc{T}}^{x,\,y}\;=\;\{\eta\in\mc{H}_{N}:\eta_{x}+\eta_{y}\ge N(1-3\epsilon)\}\;.
\end{gather*}
Let
\[
\mc{J}^{x,\,y}\;=\;\mc{T}^{x,\,y}\setminus[\,\mc{V}^{x}\cup\mc{V}^{y}\,]\;,\quad\widehat{\mc{J}}^{x,\,y}\;=\;\widehat{\mc{T}}^{x,\,y}\setminus[\,\widehat{\mc{V}}^{x}\cup\widehat{\mc{V}}^{y}\,]\;,
\]
and note that the definitions are symmetric: $\mc{T}^{x,\,y}=\mc{T}^{y,\,x}$,
$\widehat{\mc{T}}^{x,\,y}=\widehat{\mc{T}}^{y,\,x}$, $\mc{J}^{x,\,y}=\mc{J}^{y,\,x}$
and $\widehat{\mc{J}}^{x,\,y}=\widehat{\mc{J}}^{y,\,x}$.

Denote by $\mc{G}$ and $\widehat{\mc{G}}$ the union of
wells and tubes:
\begin{equation}
\mc{G}\;=\;\bigcup_{x,y\in S}\big(\,\mc{V}^{x}\,\cup\,\mc{J}^{x,\,y}\,\big)\;,\quad\widehat{\mc{G}}\;=\;\bigcup_{x,y\in S}\big(\,\widehat{\mc{V}}^{x}\,\cup\,\widehat{\mc{J}}^{x,\,y}\,\big)\;.\label{1g}
\end{equation}

We present below some properties of these sets.
\begin{lem}
\label{lem07} The following holds.
\begin{enumerate}
\item Suppose that $\eta\in\mc{\mc{J}}^{x,\,y}$ for some $x,\,y\in S$.
Then, $\eta_{x},\,\eta_{y}\in(N\epsilon,\,N(1-2\epsilon))$.
\item Suppose that $\eta\in\widehat{\mc{J}}^{x,\,y}$ for some $x,\,y\in S$.
Then, $\eta_{x},\,\eta_{y}\in(N\epsilon,\,N(1-4\epsilon))$.
\item For $\{x,\,y\}\neq\{x',\,y'\}$, $\mc{J}^{x,\,y}\cap\mc{J}^{x',\,y'}=\varnothing$.
In particular, the expression \eqref{1g} represents a partition of
$\mc{G}$.
\end{enumerate}
\end{lem}

\begin{proof}
For part (1), if $\eta\in\mc{J}^{x,\,y}=\mc{T}^{x,\,y}\setminus[\mc{V}^{x}\cup\mc{V}^{y}]$,
the bound $\eta_{x}<N(1-2\epsilon)$ is trivial because $\eta\notin\mc{V}^{x}$.
By symmetry this extends to $\eta_{y}$. By this bound and since $\eta\in\mc{T}^{x,\,y}$,
\[
\eta_{x}+N(1-2\epsilon)>\eta_{x}+\eta_{y}>N-\ell_{N}>N-\epsilon N\;.
\]
This proves the lower bound. Proof of part (2) is similar.

For part (3), it suffices to show that
\begin{equation}
\mc{T}^{x,\,y}\cap\mc{T}^{x,\,z}\subset\mc{V}^{x}\;\;\;\text{for all }x,\,y,\,z\in S\;.\label{e0071}
\end{equation}
To prove this, fix $\eta\in\mc{T}^{x,\,y}\cap\mc{T}^{x,\,z}$.
Since $\eta_{x}+\eta_{y}+\eta_{z}\le N$,
\[
2N-2\ell_{N}\le(\eta_{x}+\eta_{y})+(\eta_{x}+\eta_{z})\le\eta_{x}+N\;.
\]
Thus, $\eta_{x}\ge N-2\ell_{N}$, which implies that $\eta\in\mc{V}^{x}$,
proving \eqref{e0071}.
\end{proof}
In the remaining part of this subsection, we provide an estimate of
the measures $\mu_{N}(\widehat{\mc{G}}\setminus\mc{G})$
and $\mu_{N}(\mc{J}^{x,\,y})$. For $S_{0}\subset S$ and $k\in\bb{N}$,
let
\[
\mc{H}_{k,\,S_{0}}=\Big\{\xi=(\xi_{x})_{x\in S_{0}}\in\bb{N}^{S_{0}}:|\xi|:=\sum_{x\in S_{0}}\xi_{x}=k\Big\}\;.
\]

We adopt the following convention. Fix $c:\bb{N}\times(0,1]\to\bb{R}$.
We write ${\color{bblue}c(N,\,\epsilon)=o_{N}(1)}$ if $\lim_{N\rightarrow\infty}c({N,\,\epsilon})=0$
for all $\epsilon>0$, and ${\color{bblue}c(N,\,\epsilon)=o_{\epsilon}(1)}$
if
\[
\lim_{\epsilon\rightarrow0}\sup_{N\in\bb{N}}|c({N,\,\epsilon})|=0\;.
\]
Mind that we always send $N\rightarrow\infty$ before $\epsilon\rightarrow0$.

Hereafter, $C_{0}$ represents a finite constant independent of $N$,
$\epsilon$ and $\eta$, and $C_{\epsilon}$ a finite one, independent
of $N$ and $\eta$, but which may depend on $\epsilon$. The values
of $C_{0}$ and $C_{\epsilon}$ may change from line to line.
\begin{lem}
\label{lem012} We have that
\[
\mu_{N}(\widehat{\mc{G}}\setminus\mc{G})\le\frac{1}{\log N}\,[\,o_{N}(1)+o_{\epsilon}(1)\,]\;.
\]
\end{lem}

\begin{proof}
For $x\in S$, define
\[
\mc{A}^{x}=\widehat{\mc{V}}^{x}\setminus\mc{G}\;=\;\widehat{\mc{V}}^{x}\,\setminus\,\Big(\mc{V}^{x}\cup\bigcup_{y\in S\setminus\{x\}}\mc{J}^{x,\,y}\Big)\;.
\]
With this notation,
\[
\widehat{\mc{G}}\setminus\mc{G}\,\subset\,\,\bigcup_{x\in S}\mc{A}^{x}\,\cup\,\bigcup_{x,\,y\in S}(\widehat{\mc{J}}^{x,\,y}\setminus\mc{J}^{x,\,y})\;.
\]
Therefore, it is enough to show that
\begin{equation}
\mu_{N}(\mc{A}^{x})=\frac{o_{N}(1)}{\log N}\;\;\;\text{and\;\;\;}\mu_{N}(\widehat{\mc{J}}^{x,\,y}\setminus\mc{J}^{x,\,y})=\frac{o_{\epsilon}(1)}{\log N}\label{f002}
\end{equation}
for all $x,\,y\in S$.

We first consider $\mc{A}^{x}$. Since $\widehat{\mc{V}}^{x}\cap\mc{V}^{y}=\varnothing$,
in the definition of $\mc{A}^{x}$, we may add $\mc{V}^{y}$
to the expression inside parenthesis. At this point, we may replace
$\mc{J}^{x,\,y}$ by $\mc{T}^{x,\,y}$, and then remove
$\mc{V}^{y}$ to get that
\begin{equation}
\mc{A}^{x}=\widehat{\mc{V}}^{x}\,\setminus\,\Big(\mc{V}^{x}\cup\bigcup_{y\in S\setminus\{x\}}\mc{T}^{x,\,y}\Big)\;.\label{f001}
\end{equation}

Write
\[
\mc{A}^{x}=\bigcup_{m=N(1-4\epsilon)}^{N(1-2\epsilon)-1}\mc{A}_{m}^{x}\;,
\]
where $\mc{A}_{m}^{x}=\{\eta\in\mc{A}^{x}:\eta_{x}=m\}$.
Represent a configuration $\eta$ as $(\eta_{x},\,\xi)$, where $\xi\in\bb{N}^{S\setminus\{x\}}$
stands for the configuration $\eta$ on $S\setminus\{x\}$ {[}$\xi_{y}=\eta_{y}$
for all $y\in S\setminus\{x\}${]}. Note that $\xi\in\mc{H}_{N-m,\,S\setminus\{x\}}$
if $\eta\in\mc{A}_{m}^{x}$. Let ${\color{bblue}\mc{B}_{m}^{x}}$
the subset of configurations $\xi\in\mc{H}_{N-m,\,S\setminus\{x\}}$
such that $(m,\xi)\in\mc{A}_{m}^{x}$.

Recall the definition of the set $\Delta_{S,N}$ introduced below
\eqref{enx}. We claim that $\mc{B}_{m}^{x}\subset\Delta_{S\setminus\{x\},N-m}$.
Indeed, fix $\xi\in\mc{B}_{m}^{x}$ and $y\not=x$. Let $\eta$
be the configuration $(m,\xi)$, so that $\eta\in\mc{A}_{m}^{x}$.
By \eqref{f001}, $\eta\notin\mc{T}^{x,\,y}$. Hence, as $\eta_{x}=m$,
$\eta_{y}+m=\eta_{y}+\eta_{x}<N-\ell_{N}$ so that $\eta_{y}<N-m-\ell_{N}\le(N-m)-\ell_{N-m}$
because $\ell_{N-m}\le\ell_{N}$.

Therefore, configurations $\xi$ in $\mc{B}_{m}^{x}$ have a
total of $N-m$ particles and each site has strictly less than $(N-m)-\ell_{N-m}$
particles. Thus, by definition of $\Delta_{S,N}$, $\xi$ belongs
to $\Delta_{S\setminus\{x\},N-m}$, which proves the claim.

By definition of $\mu_{N}$ and $\mc{B}_{m}^{x}$,
\[
\mu_{N}(\mc{A}^{x})\;=\;\frac{N}{Z_{N,\,S}\,(\log N)^{\kappa-1}}\sum_{m=N(1-4\epsilon)}^{N(1-2\epsilon)-1}\frac{1}{{a}(m)}\sum_{\xi\in\mc{B}_{m}^{x}}\frac{1}{\mb{a}(\xi)}\;\cdot
\]
As $\mc{B}_{m}^{x}$ is contained in $\Delta_{S\setminus\{x\},N-m}$,
this expression is less than or equal to
\[
\frac{N}{(\log N)^{\kappa-1}}\sum_{m=N(1-4\epsilon)}^{N(1-2\epsilon)-1}\frac{1}{m}\,\frac{Z_{N-m,\,\kappa-1}}{Z_{N,\,S}}
\,\frac{[\log(N-m)]^{\kappa-2}}{N-m}\,\mu_{\kappa-1,\,N-m}(\Delta_{S\setminus\{x\},\,N-m})\;.
\]

By Proposition \ref{p31}, for each $p\ge1$, $(Z_{N,p})_{N\ge1}$
is a bounded sequence. Hence, $Z_{N-m,\kappa-1}/Z_{N,S}\le C_{0}$.
As $2\epsilon N\le N-m\le4\epsilon N$, $N/(N-m)\le C_{0}/\epsilon$,
and $\log(N-m)/\log N\le1$. The previous expression is thus bounded
above by
\[
\frac{C_{0}}{\log N}\,\frac{1}{\epsilon}\,\frac{1}{(1-4\epsilon)N}\,\sum_{M=2\epsilon N}^{4\epsilon N}\,\mu_{\kappa-1,M}(\Delta_{S\setminus\{x\},M})\;.
\]
By Theorem \ref{t26}, the sequence $\mu_{\kappa-1,M}(\Delta_{S\setminus\{x\},M})$
vanishes as $M\to\infty$. This shows that the previous sum is bounded
by $2\epsilon No_{N}(1)$. This proves the first estimate in \eqref{f002}.

We turn to the second bound of \eqref{f002}. Write $\widehat{\mc{J}}^{x,\,y}\setminus\mc{J}^{x,\,y}$
as
\begin{equation}
\widehat{\mc{J}}^{x,\,y}\setminus\mc{J}^{x,\,y}\;=\;\bigcup_{m=N(1-4\epsilon)}^{N-\ell_{N}-1}\mc{I}_{m}^{x,\,y}\;,\label{decomj}
\end{equation}
where $\mc{I}_{m}^{x,\,y}=\{\eta\in\widehat{\mc{J}}^{x,\,y}\setminus\mc{J}^{x,\,y}:\eta_{x}+\eta_{y}=m\}$.
Write $\eta\in\mc{I}_{m}^{x,\,y}$ as $\eta=(\eta_{x},\,\eta_{y},\,\zeta)$,
where $\zeta\in\mc{H}_{N-m,\,S\setminus\{x,\,y\}}$ represents
the configuration of $\eta$ on $S\setminus\{x,\,y\}$.

By part (2) of Lemma \ref{lem07}, $\eta_{x},\,\eta_{y}>N\epsilon$
for configurations $\eta$ in $\widehat{\mc{J}}^{x,\,y}$. Therefore,
by Proposition \ref{p31}, there exists a finite constant $C_{0}$
such that
\[
\mu_{N}(\mc{I}_{m}^{x,\,y})\;\le\;\frac{C_{0}\,N}{(\log N)^{\kappa-1}}\,\sum_{i=N\epsilon}^{m-N\epsilon}\frac{1}{{a}(i)\,{a}(m-i)}\sum_{\zeta\in\mc{H}_{N-m,\,S\setminus\{x,\,y\}}}\frac{1}{\mb{a}(\zeta)}\;.
\]
An elementary computation yields that there exists a finite constant
$C_{0}$ such that
\[
\sum_{i=N\epsilon}^{m-N\epsilon}\frac{1}{{a}(i)\,{a}(m-i)}\;\le\;\frac{C_{0}}{N}\,\log\frac{1}{\epsilon}
\]
for all $(1-4\epsilon)N\le m\le N$.

On the other hand, by Proposition \ref{p31} and since $m\le N-\ell_{N}$,
there exists a constant $C_{0}$ such that
\[
\sum_{\zeta\in\mc{H}_{N-m,\,S\setminus\{x,\,y\}}}\frac{1}{\mb{a}(\zeta)}\;\le\;C_{0}\,\frac{[\log(N-m)]^{\kappa-3}}{N-m}\;\le\;C_{0}\,\frac{(\log N)^{\kappa-3}}{\ell_{N}}\;.
\]

Combining the previous estimates yields that
\[
\mu_{N}(\mc{I}_{m}^{x,\,y})\;\le\;\frac{C_{0}}{N\log N}\log\frac{1}{\epsilon}\;,
\]
and hence by \eqref{decomj},
\[
\mu_{N}(\widehat{\mc{J}}^{x,\,y}\setminus\mc{J}^{x,\,y})\;\le\;4N\epsilon\,\frac{C_{0}}{N\log N}\log\frac{1}{\epsilon}\;=\;\frac{o_{\epsilon}(1)}{\log N}\;,
\]
as claimed.
\end{proof}
\begin{lem}
\label{lem013} There exists a finite constant $C_{0}$ such that,
for all $x,\,y\in S$,
\[
\mu_{N}(\mc{J}^{x,\,y})\le\frac{C_{0}}{\log N}\log\frac{1}{\epsilon}\;\cdot
\]
\end{lem}

\begin{proof}
The proof is similar to the one of the last part of the previous lemma.
Fix $x,\,y\in S$ and write
\[
\mc{J}^{x,\,y}=\bigcup_{m=N-\ell_{N}}^{N}\mc{J}_{m}^{x,\,y}\;,
\]
where ${\color{bblue}\mc{J}_{m}^{x,\,y}=\{\eta\in\mc{J}^{x,\,y}:\eta_{x}+\eta_{y}=m\}}$.

Represent a configuration $\eta$ in $\mc{J}_{m}^{x,\,y}$ as
$\eta=(\eta_{x},\,\eta_{y},\,\zeta)$ for $\zeta\in\mc{H}_{N-m,\,S\setminus\{x,\,y\}}$.
By part (1) of Lemma \ref{lem07}, $\eta_{x},\,\eta_{y}>N\epsilon$
for configurations $\eta$ in $\mc{J}^{x,\,y}$. Thus,
\[
\mu_{N}(\mc{J}_{m}^{x,\,y})\;\le\;\frac{C_{0}\,N}{(\log N)^{\kappa-1}}\,\sum_{i=N\epsilon}^{m-N\epsilon}\frac{1}{{a}(i)}\,\frac{1}{{a}(m-i)}\,\sum_{\xi\in\mc{H}_{N-m,\,S\setminus\{x,\,y\}}}\frac{1}{\mb{a}(\xi)}\;\cdot
\]
Clearly, there exists a finite $C_{0}$ such that
\begin{equation}
\sum_{i=N\epsilon}^{m-N\epsilon}\frac{1}{{a}(i)}\,\frac{1}{{a}(m-i)}\;\le\;\frac{C_{0}}{N}\,\log\frac{1}{\epsilon}\;,\label{e0133}
\end{equation}
for all $N-\ell_{N}\le m\le N$.

By Proposition \ref{p31},
\[
\sum_{\xi\in\mc{H}_{N-m,\,S\setminus\{x,\,y\}}}\frac{1}{\mb{a}(\xi)}\;\le\;\frac{C_{0}\,[\log(N-m)]^{\kappa-3}}{N-m}\;\le\;C_{0}\frac{(\log N)^{\kappa-3}}{N-m}\
\]
for $N-\ell_{N}\le m<N$. For $m=N$, the sum is bounded by $1$.

Putting together the previous estimates yields that
\[
\mu_{N}(\mc{J}_{m}^{x,\,y})\;\le\;\frac{C_{0}}{(\log N)^{2}}\,\frac{1}{N-m}\,\log\frac{1}{\epsilon}
\]
for $N-\ell_{N}\le m<N$ and $\mu_{N}(\mc{J}_{N}^{x,\,y})\;\le\;[C_{0}/(\log N)^{\kappa-1}]\,\log(1/\epsilon)$.

Summing over $N-\ell_{N}\le m\le N$ gives that
\[
\mu_{N}(\mc{J}^{x,\,y})\;\le\;\frac{C_{0}}{(\log N)^{2}}\,\log\frac{1}{\epsilon}\,\sum_{k=1}^{\ell_{N}}\frac{1}{k}\;+\;\frac{C_{0}}{(\log N)^{\kappa-1}}\,\log\frac{1}{\epsilon}\;\le\;\frac{C_{0}}{\log N}\,\,\log\frac{1}{\epsilon}\;,
\]
as claimed.
\end{proof}
Decompose the tube $\mc{J}^{x,\,y}$ as
\begin{equation}
\mc{J}^{x,\,y}\;=\;\mc{K}^{x,\,y}\;\cup\;\mc{L}^{x,\,y}\;,\label{decj-1}
\end{equation}
where
\begin{gather*}
\mc{K}^{x,\,y}=\left\{ \eta\in\mc{J}^{x,\,y}:\eta_{x}\text{ or }\eta_{y}<6N\epsilon\right\} \;,\\
\mc{L}^{x,\,y}=\left\{ \eta\in\mc{J}^{x,\,y}:\eta_{x},\,\eta_{y}\ge6N\epsilon\right\} \;.
\end{gather*}
The next lemma asserts that we can remove the factor $\log(1/\epsilon)$
in the previous lemma replacing $\mc{J}^{x,\,y}$ by $\mc{K}^{x,\,y}$.
\begin{lem}
\label{lem0123} There exists a finite constant $C_{0}$ such that,
for all $x,\,y\in S$,
\[
\mu_{N}(\mc{K}^{x,\,y})\le\frac{C_{0}}{\log N}\;\cdot
\]
\end{lem}

\begin{proof}
Assume that $\eta_{x}\le6\epsilon N$, and let $\mc{K}_{m}^{x,\,y}=\mc{K}^{x,\,y}\cap\mc{I}_{m}^{x,\,y}$,
$N-\ell_{N}\le m\le N$. By (1) of Lemma \ref{lem07}, $\eta_{x}>\epsilon N$.
Hence, $\eta_{x}$ varies from $\epsilon N$ to $6\epsilon N$.

Proceed as in the previous lemma. In the formula for $\mu_{N}(\mc{K}_{m}^{x,\,y})$,
let $i$ represent $\eta_{x}$, so that \eqref{e0133} becomes
\[
\sum_{i=N\epsilon}^{6N\epsilon}\frac{1}{{a}(i)}\,\frac{1}{{a}(m-i)}\;\le\;\frac{C_{0}}{N}\;\cdot
\]
The rest of the argument is identical to the one of Lemma \ref{lem013}.
\end{proof}

\subsection{Construction of test functions\label{sec65}}

In this section, we introduce functions
$U_{x,\,y}:\mc{H}_{N}\rightarrow\bb{R}$, $x,\,y\in S$, to examine the
Resolvent equation \eqref{nfg03}. These functions are similar to the
ones introduced in the super-critical case in \textbf{\cite{BL3}} to
estimate the capacities between wells.

Fix $x,\,y\in S$ and a small parameter $\epsilon>0$. Let
$\phi_{\epsilon}:[0,1]\rightarrow[0,1]$ be a smooth, non-decreasing,
bijective function such that
\[
\phi_{\epsilon}(t)\,+\,\phi_{\epsilon}(1-t)\,=\,1\;,\quad t\in[0,\,1]\;,
\]
\[
\phi_{\epsilon}(t)\;=\;\begin{cases}
0 & t\in[0,\,3\epsilon]\;,\\
(t-4\epsilon)/(1-8\epsilon) & t\in[5\epsilon,\,1-5\epsilon]\;,\\
1 & t\in[1-3\epsilon,\,1]\;.
\end{cases}
\]
\[
\phi_{\epsilon}'(t)\;\le\;1\,+\,\epsilon^{1/2}\;,\quad t\in[0,\,1]\;.
\]
Although, the existence of such a function is straightforward, we
refer to \textbf{\cite[Section 7.3]{Seo} }for an explicit construction.

Define $\Phi_{\epsilon}:[0,\,1]\rightarrow[0,\,1]$ by
\[
\Phi_{\epsilon}(t)\coloneqq6\int_{0}^{\phi_{\epsilon}(t)}u\,(1-u)\,du\,=\,3\,\phi_{\epsilon}(t)^{2}-2\,\phi_{\epsilon}(t)^{3}\;.
\]
Note that
\[
\Phi_{\epsilon}(t)=\begin{cases}
0 & \text{if }t\in[0,\,3\epsilon]\\
1 & \text{if }t\in[1-3\epsilon,\,1]\;.
\end{cases}
\]

Recall from \eqref{PE} that $h_{x,\,y}=h_{\{x\},\,\{y\}}:S\rightarrow\bb{R}$
denotes the equilibrium potential between $x$ and $y$ for the random
walk $X(\cdot)$. Let
\begin{equation}
x\,=\,z_{1}\,,\,z_{2}\,,\,\dots\,,\,z_{\kappa}\,=\,y\label{enum1}
\end{equation}
be an enumeration of $S$ satisfying
\[
1\;=\;h_{x,\,y}(z_{1})\,\geq\,h_{x,\,y}(z_{2})\,\ge\,\cdots\,\ge\,h_{x,\,y}(z_{\kappa})\,=\,0\;.
\]

Define $U_{x,\,y}:\mc{H}_{N}\rightarrow\bb{R}$ by
\begin{equation}
U_{x,\,y}(\eta)\;=\;\sum_{j=1}^{\kappa-1}\big[\,h_{x,\,y}(z_{j})\,-\,h_{x,\,y}(z_{j+1})\,\big]\,\Phi_{\epsilon}\left(\,\frac{1}{N}\,\sum_{i=1}^{j}\eta_{z_{i}}\,\right).\label{uxy}
\end{equation}
The function $U_{x,\,y}$ approximates the equilibrium potential between
$\mc{V}^{x}$ and $\mc{V}^{y}$ in the tube $\mc{J}^{x,y}$.

\begin{rem}
\label{rm013}
Fix $x,\,y\in S$, and denote by $z_{1},\,z_{2},\,\dots,\,z_{\kappa}$
and $z_{1}',\,z_{2}',\,\dots,\,z_{\kappa}'$ the sequences \eqref{enum1}
associated to the functions $U_{x,\,y}$ and $U_{y,\,x}$, respectively.
We assume that $z_{i}=z_{\kappa+1-i}'$, $1\le i\le\kappa$. Clearly,
this conditions holds if $h_{x,\,y}(z)\neq h_{x,\,y}(w)$ for all
$z\not=w\in S$.
\end{rem}

Denote by $\Vert u\Vert_{\infty}$ the sup-norm of a function
$u:S\rightarrow\bb{R}$,
${\color{bblue}\Vert u\Vert_{\infty}=\max_{x\in S}\,|\,u(x)\,|}$. Let
$g_{N}:S\rightarrow\bb{R}$ be given by $g_N=f-f_N$. The sequence
$\{g_N:N\ge 1\}$ is uniformly bounded because, by Lemma \ref{np3}, so
is $\{f_N:N\ge 1\}$. Note that the computation below does not depend
on the specific form of the function $g_N$ but only on the fact that
it is uniformly bounded.

 We omit below the dependence of $g$ on $N$. We define a function
$V^{g}:\mc{H}_{N}\rightarrow\bb{R}$ in few steps. We first
construct it on $\mc{G}$, and then extend it to the whole set.
Recall from Lemma \ref{lem07}-(3) that the set $\mc{G}$ can
be represented as a disjoint union of the sets $\mc{V}^{x}$
and $\mc{J}^{x,y}$. Let
\begin{equation}
V^{g}(\eta)=\begin{cases}
g(x) & \text{if }\eta\in\mc{V}^{x},\;\;x\in S\;,\\
g(y)\,+\,[g(x)-g(y)]\,U_{x,\,y}(\eta) & \text{if }\eta\in\mc{J}^{x,\,y},\;\;x,\,y\in S\;.
\end{cases}\label{cv1}
\end{equation}

By Remark \ref{rm013}, we have $U_{y,\,x}=1-U_{x,\,y}$ on $\mc{J}^{x,\,y}$.
Hence,
\[
g(y)\,+\,[g(x)-g(y)]\,U_{x,\,y}(\eta)\;=\;g(x)\,+\,[g(y)-g(x)]\,U_{y,x}(\eta)\;,
\]
and $V^{g}$ is well-defined on $\mc{J}^{x,\,y}$.

The function $V^{g}$ is smooth enough on $\mc{G}$ in the following
sense.
\begin{lem}
\label{lem010} There exists a finite constant $C_{0}$ such that,
\[
\max_{\eta\in\mc{G}}|\,V^{g}(\eta)\,| \le C_0 \;\;\;\text{and}\;\;\;\quad\big|\,V^{g}(\sigma^{z,\,w}\eta)\,-\,V^{g}(\eta)\,\big|\,\le\,\frac{C_{0}}{N}
\]
for all $N\in\bb{N}$, $z,\,w\in S$, and configurations $\eta$
in $\mc{G}$ such that $\sigma^{z,\,w}\eta\in\mc{G}$.
\end{lem}

\begin{proof}
The first bound follows from the definition of $V^{g}$ and from the
fact that $U_{x,y}$ is bounded by $1$.

We turn to the second. From the definition of the sets $\mc{V}^{x}$, $\mc{J}^{x,y}$,
for a pair $(z,w)$ and configurations $\eta$ and $\sigma^{z,\,w}\eta$
in $\mc{G}$, there are three possibilities. Either $\eta$ and
$\sigma^{z,\,w}\eta$ belong to some set $\mc{V}^{x}$, or both
to some set $\mc{J}^{x,y}$ or $\eta$ belongs to some $\mc{V}^{x}$
and $\sigma^{z,\,w}\eta$ to some $\mc{J}^{x,y}$ {[}or the opposite{]}.
We consider separately the three cases.

The inequality is trivial if $\eta,\,\sigma^{z,\,w}\eta\in\mc{V}^{x}$
for some $x\in S$ since in this case $V^{g}(\sigma^{z,\,w}\eta)-V^{g}(\eta)=0$.

By definition of $\Phi_{\epsilon}$ and the bound on the derivative
of $\phi_{\epsilon}$,
\[
|\Phi_{\epsilon}'(t)|\;=\;6\,\big|\,\phi_{\epsilon}'(t)\,\phi_{\epsilon}(t)\,[1-\phi_{\epsilon}(t)]\,\big|\;\le\;6\,(1+\epsilon^{1/2})\;.
\]
Therefore, there exists a finite constant $C_{0}$ such that,
\begin{equation}
\big|\,U_{x,\,y}(\sigma^{z,\,w}\eta)-U_{x,\,y}(\eta)\,\big|\;\le\;\frac{C_{0}}{N}\label{e0011}
\end{equation}
for all $N\in\bb{N}$, $\eta\in\mc{H}_{N}$, and $z,\,w\in S$.
In particular, the inequality stated in Lemma \ref{lem010} holds
if $\eta$, $\sigma^{z,\,w}\eta\in\mc{J}^{x,\,y}$ for some $x,\,y\in S$.

Finally, assume that $\sigma^{z,\,w}\eta\in\mc{J}^{x,\,y}$ and
$\eta\in\mc{V}^{x}$ for some $x,\,y\in S$. The same argument
applies to the converse situation. In this case, $U_{x,\,y}(\eta)=1$
because $\eta_{x}\ge(1-2\epsilon)N$. Thus, by definition of $V^{g}$,
\[
V^{g}(\sigma^{z,\,w}\eta)\;-\;V^{g}(\eta)\;=\;[\,g(x)-g(y)\,]\,[\,U_{x,\,y}(\sigma^{z,\,w}\eta)-U_{x,\,y}(\eta)\,]\;,
\]
and the assertion of the lemma follows from \eqref{e0011}.
\end{proof}
To extend the function $V^{g}$ to $\mc{H}_{N}\setminus\mc{G}$,
let
\begin{equation}
V^{g}(\eta)=0\text{\;\;\;for}\;\eta\in\mc{H}_{N}\setminus\widehat{\mc{G}}\;.\label{cv2}
\end{equation}
On $\widehat{\mc{G}}\setminus\mc{G}$, smoothly interpolate the
construction \eqref{cv1} and \eqref{cv2} in such a way that
$\max_{\eta\in\widehat{\mc{G}}\setminus\mc{G}}|\,V^{g}(\eta)\,|\le C_0$ (where $C_0$ is the constant appeared in Lemma \ref{lem010}) and
\[
\big|\,V^{g}(\sigma^{z,\,w}\eta)-V^{g}(\eta)\,\big|\;\le\;\frac{C_{\epsilon}}{N}\,\quad\text{for all }\eta\in\widehat{\mc{G}}\setminus\mc{G}\;,\;\;\;z,\,w\in S\;,
\]
where $C_{\epsilon}$ is a constant independent of $N$. This is
possible in view of Lemma \ref{lem010} and since the distance between
$\mc{H}_{N}\setminus\widehat{\mc{G}}$ and $\mc{G}$ is of order
$N\epsilon$.

The next result summarizes the bounds obtained in the construction.
Recall that $\epsilon>0$ is fixed small parameter which appeared in
the construction of the function $\phi_\epsilon$ and that $V$ depends
on $\epsilon$ though the dependence does not appear in the notation.

\begin{lem}
\label{lem0101}
For each small $\epsilon>0$, there exist  finite
constants $C_0$ and  $C_{\epsilon}$ such that,
\[
\max_{\eta\in\mc{H}_{N}}|\,V^{g}(\eta)\,|
\;\le\; C_0 \quad \text{and}\quad
\big|\,V^{g}(\eta^{z,\,w})-V^{g}(\eta)\,\big|
\;\le\;\,\frac{C_{\epsilon}}{N}
\]
for all $N\ge 1$, $z,\,w\in S$, and $\eta\in\mc{H}_{N}$.
\end{lem}

We have now all elements to estimate the capacity between
$\mathcal{E}_{N}^{x}$ and $\breve{\mathcal{E}}_{N}^{x}$.

\begin{proof}[Proof of Proposition \ref{pp71}]
Fix $x\in S$.  By the Dirichlet principle \cite[equation
(B14)]{l-review},
\begin{equation*}
\Cap_{N}(\mathcal{E}_{N}^{x},\,\breve{\mathcal{E}}_{N}^{x})
\;\le\; \bb D_{N}(F)
\end{equation*}
for any function $F:\mc H_N \to \bb R$ such that $F\equiv1$ on
$\mathcal{E}_{N}^{x}$ and $F\equiv0$ on
$\breve{\mathcal{E}}_{N}^{x}$.

Let $\chi_x=\chi_{\{x\}}:S\rightarrow\mathbb{R}$ be the indicator
function on $x$, i.e.,
\begin{equation*}
\chi_x(y)=\mathbf{1}\{x=y\}\;, \quad y\in S\;,
\end{equation*}
and consider the test function $F=V^{\chi_x}$, where $V^{g}$ has been
introduced in \eqref{cv1}. Let $\mc C_x$ be the subset of $\mc H_N$
given by
\begin{equation*}
\mathcal{C}_{x} \;=\; \bigcup_{y\in S\setminus\{x\}}
\mathcal{J}^{x,\,y} \;\cup\;
\big(\, \widehat{\mathcal{G}}\setminus\mathcal{G} \,\big)\;.
\end{equation*}

Since
$F(\sigma^{z,\,w}\eta)=F(\eta)$ unless $\eta$ or $\eta^{x,\,y}$
belongs to $\mathcal{C}^{x}$,
\begin{align*}
\bb D_{N}(V^{\chi_x}) \;\le\; \theta_N\,
\sum_{z,\,w\in S}\sum_{\eta}
\mu_{N}(\eta) \, {g}(\eta_{z}) \, r(z,\,w)\,
\left[\, V^{\chi_x}(\sigma^{z,\,w}\eta)-V^{\chi_x}(\eta)\, \right]^{2}\;,
\end{align*}
where the second sum is performed over all $\eta$ at distance one or
less from $\mathcal{C}_{x}$. By Lemmata \ref{lem012}, \ref{lem013} and
\ref{lem0101},
\begin{align*}
\bb D_{N}(V^{\chi_x}) \; \le \; \frac{C_0\,\theta_N}{N^{2}}
\, \mu_{N}(\mathcal{C}_{x}) \;\le\;  {C_0} \,
\log \frac{1}{\epsilon}\;\cdot
\end{align*}
To completes the proof, it remains to fix some $0<\epsilon < 1$ and
observe that the sets $\mc E^z_N$ do not depend on $\epsilon$.
\end{proof}

\subsection{Proof of Proposition \ref{p4}\label{sec66}}

The proof is based on Proposition \ref{pma} stated below. Fix a
function $f:S\to\bb{R}$, and recall the definition of $F_{N}$,
introduced in \eqref{nfg03}, and the one of $\bb D_{N}$ given in
\eqref{dnfg}.

Let $D_{Z}(u,\,v)$ be the bilinear form given by
\[
D_{Z}(u,\,v)\;=\;\frac{1}{\kappa}
\sum_{x\in S}u(x)\,(-L_{Z}v)(x)\;=\;
\frac{1}{2\kappa}\,\sum_{x,\,y\in S}r_{Z}(x,y)\,(u(y)-u(x))\,(v(y)-v(x))\;,
\]
for $u,\,v:S\rightarrow\bb{R}$. Here, $L_{Z}$ is the generator
introduced in \eqref{genz}.  The next result is proven in Section
\ref{sec67}.

\begin{prop}
\label{pma}We have that
\begin{equation}
 \bb D_{N}(V^{g},\,F_{N})\;=\;D_{Z}(g,\,f_{N})\;+\; o_{N}(1)\,+\,o_{\epsilon}(1)\;.
 \label{epma-1}
\end{equation}
\end{prop}

\begin{prop}
\label{pma2}
We have that
\[
\langle\,V^{g},\,F_{N}\,\rangle_{\mu_{N}}\;=\;
\frac{1}{\kappa}\,\sum_{x\in S}\,g(x)\,f_{N}(x)\,
+\,o_{N}(1) \;.
\]
\end{prop}

\begin{proof}
Since $V^{g}(\eta) = g(x)$ for $\eta\in \mc{E}_{N}^{x}$, by Theorem
\ref{t26} and the first bound of Lemma \ref{np3},
\[
\sum_{\eta\in\mc{E}_{N}^{x}}V^{g}(\eta)\,F_{N}(\eta)\,
\mu_{N}(\eta)\;=\;\frac{1}{\kappa}\,g(x)\,f_{N}(x)\,
+\,o_{N}(1)   \;.
\]
As, by Lemma \ref{lem0101}, $|V^{g}(\eta)|\le C_0$, it remains to show that
\[
\sum_{\eta\in\Delta_{N}}\,|F_{N}(\eta)|\,\mu_{N}(\eta)=o_{N}(1)\;.
\]
By Lemma \ref{np3}, the sum is bounded by $C_0\, \mu_{N}(\Delta_{N})$
for some finite constant $C_0$.  By Theorem \ref{t26}, this expression
vanishes as $N\to\infty$.
\end{proof}

\begin{proof}[Proof of Proposition \ref{p4}]
The main idea of the proof is to compute $  \bb D_{N}(V^{g},\,F_{N})$
in two different ways. The first one is the estimate carried out in
Proposition \ref{pma}. The other, and simpler one, is presented below.

Multiply both sides of the equation \eqref{nfg03} by
$V^{g}(\eta)\mu_{N}(\eta)$ and sum over $\eta\in\mc{H}_{N}$ to
obtain that
\[
\lambda\,\langle\,V^{g},\,F_{N}\,\rangle_{\mu_{N}}
\,+\, \,\bb D_{N}(V^{g},\,F_{N})
\;=\; \langle V^{g},\,G_{N}\rangle_{\mu_{N}}.
\]
By Theorem \ref{t26} and Lemma \ref{lem0101}, since
$\mc{E}^{x}\subset\mc{V}^{x}$,
\begin{align}
\langle V^{g},\,G_N \rangle_{\mu_{N}}\;
& =\;\frac{1}{\kappa}\,\sum_{x\in S}g(x)\,
(\lambda f-L_{Z}f)(x)\,+\,o_{N}(1)\, \nonumber \\
& =\,\frac{\lambda}{\kappa}\sum_{x\in S}g(x)\,f(x)
\,+\,D_{Z}(g,\,f)\,+\,o_{N}(1)\;.
\label{compu2a}
\end{align}
The two previous equations along with Proposition \ref{pma2} yield
that
\[
 \bb D_{N}(V^{g},\,F_{N})
\;=\;\frac{\lambda}{\kappa}\sum_{x\in S}g(x)\,(f-f_{N})(x)
\;+\;D_{Z}(g,\,f)\,+\,\,
o_{N}(1)  \;.
\]

Thus, by Proposition \ref{pma},
\[
D_{Z}(g,\,f-f_{N})\;+\;\frac{\lambda}{\kappa}
\sum_{x\in S}g(x)\,(f-f_{N})(x)\;=\;
  o_{N}(1)\,+\,o_{\epsilon}(1)\;.
\]
Set $g=f-f_{N}$ to get that
\begin{equation}
\Vert f-f_{N}\Vert_{\infty}\;\le\;C_0 \,
\big\{\, o_{N}(1)\,+\,o_{\epsilon}(1)\,\big\}
\label{fsw1}
\end{equation}
for some finite constant $C_0 >0$.
Since both $f$ and $f_{N}$ do not depend on $\epsilon$, this implies
that $\lVert f-f_{N}\rVert_{\infty}=o_{N}(1)$, which completes the
proof.
\end{proof}

\subsection{Proof of Proposition \ref{pma}\label{sec67}}

Let $\mf{d}^{x}$, $x\in S$, be the configuration with one
particle at $x$ and no particles at the other sites.

For each set $\mc{A\subset\mc{H}}_{N}$, denote by $\mc{A}_{-},\,\mc{A}_{+}\subset\mc{H}_{N-1}$
the sets defined by
\begin{gather*}
\mc{A}_{-}\;=\;\{\xi\in\mc{H}_{N-1}:\xi+\mf{d}^{x}\in\mc{A}\;\;\;\forall\,x\in S\}\;,\\
\mc{A}_{+}\;=\;\{\xi\in\mc{H}_{N-1}:\exists\,x\in S\;\text{ s.t. }\;\xi+\mf{d}^{x}\in\mc{A}\}\;.
\end{gather*}

Recall from \eqref{1g} the definition of the subsets $\mc{G}$,
$\widehat{\mc{G}}$ of $\mc{H}_{N}$. We claim that
\begin{equation}
\mc{H}_{N-1}\;=\;\mc{G}_{-}\,\cup\,(\mc{H}_{N}\setminus\widehat{\mc{G}})_{-}\,\cup\,(\widehat{\mc{G}}\setminus\mc{G})_{+}\;.\label{ff02}
\end{equation}
It is clear that the right-hand set is contained in $\mc{H}_{N-1}$.
Fix $\xi\in\mc{H}_{N-1}$. Suppose that $\xi+\mf{d}^{x}$
belongs to $\widehat{\mc{G}}\setminus\mc{G}$ for some $x\in S$.
In this case, $\xi\in(\widehat{\mc{G}}\setminus\mc{G})_{+}$.

Suppose, now, that $\xi+\mf{d}^{x}\not\in\widehat{\mc{G}}\setminus\mc{G}$
for all $x\in S$. Fix $x_{0}\in S$. Since $\mc{H}_{N}=\mc{G}\,\cup\,(\mc{H}_{N}\setminus\widehat{\mc{G}})\,\cup\,(\widehat{\mc{G}}\setminus\mc{G})$,
$\xi+\mf{d}^{x_{0}}\in\mc{G}\,\cup\,(\mc{H}_{N}\setminus\widehat{\mc{G}})$.
Suppose that $\xi+\mf{d}^{x_{0}}\in\mc{G}$. The argument
applies to the other possibility. Fix $y\in S\setminus\{x_{0}\}$.
Since $\xi+\mf{d}^{x_{0}}$ and $\xi+\mf{d}^{y}$ are
neighbors, $\xi+\mf{d}^{y}$ can not belong to $\mc{H}_{N}\setminus\widehat{\mc{G}}$.
As it also does not belong to $\widehat{\mc{G}}\setminus\mc{G}$,
$\xi+\mf{d}^{y}$ is in $\mc{G}$. Hence, $\xi+\mf{d}^{y}\in\mc{G}$
for all $y\in S$, so that $\xi\in\mc{G}_{-}$, as claimed in
\eqref{ff02}.

As the sets on the right-hand side of \eqref{ff02} are disjoint, this
identity provides a partition of the set $\mc{H}_{N-1}$.
\begin{lem}
\label{lemcomp} There exists a finite constant $C_{0}$ such that
\[
\mu_{N-1}(\mc{A}_{+})\;\le\;C_{0}\,\mu_{N}(\mc{A})\;
\]
for all $\mc{A}\subset\mc{H}_{N}$ and $N\ge3$.
\end{lem}

\begin{proof}
Note that
\[
\mc{A}_{+}\;=\;\bigcup_{x\in S}\{\eta-\mf{d}^{x}:\eta\in\mc{A}\text{ with }\eta_{x}\ge1\}\;.
\]
Therefore,
\[
\mu_{N-1}(\mc{A}_{+})\le\sum_{x\in S}\sum_{\eta\in\mc{A}:\eta_{x}\ge1}\mu_{N-1}(\eta-\mf{d}^{x})\;.
\]
In particular, it is enough to show that there exists a finite constant
$C_{0}$ such that
\[
\frac{\mu_{N-1}(\eta-\mf{d}^{x})}{\mu_{N}(\eta)}\;\le\;C_{0}\;\;\;\text{for all }\eta\in\mc{H}_{N}\text{ with }\eta_{x}\ge1\;.
\]
By definition of the measure $\mu_{N}$ and Proposition \ref{p31},
\begin{align*}
\frac{\mu_{N-1}(\eta-\mf{d}^{x})}{\mu_{N}(\eta)}\; & =\;\frac{N-1}{Z_{N-1,\,S}\,[\log(N-1)]^{\kappa-1}}\,\frac{Z_{N,\,S}\,(\log N)^{\kappa-1}}{N}\,\frac{\mb{a}(\eta)}{\mb{a}(\eta-\mf{d}^{x})}\\
 & \le C_{0}\,\frac{(\log N)^{\kappa-1}}{[\log(N-1)]^{\kappa-1}}\;.
\end{align*}
This proves the bound and the lemma.
\end{proof}
Denote by ${\color{bblue}\bb D_{N}(F,\,G;\mc{A})}$, $\mc{A}\subset\mc{H}_{N-1}$,
the bilinear form given by
\[
\frac{\theta_N\,a_{N}}{2}\sum_{\xi\in\mc{A}}\sum_{x,\,y\in S}\mu_{N-1}(\xi)\,r(x,\,y)\,\left[F(\xi+\mf{d}^{y})-F(\xi+\mf{d}^{x})\right]\,\left[G(\xi+\mf{d}^{y})-G(\xi+\mf{d}^{x})\right]\;,
\]
for $F$, $G:\mc{H}_{N}\to\bb{R}$. In this formula,
\begin{equation}
a_{N}\;=\;\frac{Z_{N-1,\,S}\,[\log(N-1)]^{\kappa-1}}{N-1}\,\frac{N}{Z_{N,\,S}(\log N)^{\kappa-1}}\;=\;1\,+\,o_{N}(1)\;.\label{ff05}
\end{equation}

Since
\[
\mu_{N}(\xi+\mf{d}^{x})\,{g}(\xi_{x}+1)\;=\;a_{N}\,\mu_{N-1}(\xi)\;,
\]
a change of variables shows that
\[
\bb D_{N}(F,\,G)\;=\;\bb D_{N}(F,\,G;\mc{H}_{N-1})\;.
\]

The proof Proposition \ref{pma} is divided in several lemmata. We
start by restricting the computation to the set $\mc{G}$.

\begin{lem}
\label{lemc1}
We have that
\[
\bb D_{N}(V^{g},\,F_{N})\;=\;\,\bb D_{N}(V^{g},\,F_{N};\mc{G}_{-})\;+\;  o_{N}(1)\,+\,o_{\epsilon}(1)\;.
\]
\end{lem}

\begin{proof}
Since \eqref{ff02} is a partition of $\mc{H}_{N-1}$, we can
write $\bb D_{N}(V^{g},\,F_{N})$ as
\begin{align*}
\bb D_{N}(V^{g},\,F_{N};\mc{G}_{-})\;+\;\bb D_{N}(V^{g},\,F_{N};(\mc{H}_{N}\setminus\widehat{\mc{G}})_{-})\;+\;\bb D_{N}(V^{g},\,F_{N};(\widehat{\mc{G}}\setminus\mc{G})_{+})\;.
\end{align*}

On the one hand,
\[
\bb D_{N}(V^{g},\,F_{N};(\mc{H}_{N}\setminus\widehat{\mc{G}})_{-})\;=\;0
\]
because $V^{g}(\xi+\mf{d}^{z})=0$ for all $\xi\in(\mc{H}_{N}\setminus\widehat{\mc{G}})_{-}$,
$z\in S$.

On the other hand, by the Cauchy-Schwarz inequality and Lemma
\ref{np3},
\begin{equation}\label{rev91}
\bb D_{N}(V^{g},\,F_{N};(\widehat{\mc{G}}\setminus\mc{G})_{+})^{2}\;\le\;C_{0}\, \bb D_{N}(V^{g},\,V^{g};(\widehat{\mc{G}}\setminus\mc{G})_{+})\;
\end{equation}
for some finite constant $C_{0}$. By Lemmata \ref{lem0101}, \ref{lemcomp}
and the bound on $a_{N}$, the previous Dirichlet form is less than
or equal to
\begin{align*}
 & C_{0} \,\theta_N \,\sum_{\xi\in(\widehat{\mc{G}}\setminus\mc{G})_{+}}\sum_{z,\,w\in S}\mu_{N-1}(\xi)\,\left[V^{g}(\xi+\mf{d}^{z})-V^{g}(\xi+\mf{d}^{w})\right]^{2}\\
 & \quad\le\; \frac{C_{\epsilon}\,\theta_N}{N^{2}}\,\mu_{N-1}((\widehat{\mc{G}}\setminus\mc{G})_{+})\;
 \le\; \,\frac{C_{\epsilon}\,\theta_N}{N^{2}}\,\mu_{N}(\widehat{\mc{G}}\setminus\mc{G})\;.
\end{align*}
Hence, by Lemma \ref{lem012},
\[
\bb D_{N}(V^{g},\,V^{g};(\widehat{\mc{G}}\setminus\mc{G})_{+})\;\le\;o_{N}(1)\,+\,o_{\epsilon}(1)\;.
\]
Inserting this to \eqref{rev91}  completes the proof of the lemma.
\end{proof}
Recall, from \eqref{decj-1}, the definition of the sets $\mc{K}^{x,\,y}$,
$\mc{L}^{x,\,y}$, and, from \eqref{enum1}, the definition of
the sequence $(z_{i})_{i=1}^{\kappa}$.
\begin{lem}
\label{lem014} Fix $x\,\not=\,y\in S$. There exists a constant $C_{0}$
such that for all $\eta\in\mathscr{\mc{J}}^{x,\,y}$ and $1\le m<\kappa$,
\[
0\,\le\,\Phi_{\epsilon}\bigg(\,\frac{1}{N}\sum_{i=1}^{m}\eta_{z_{i}}\,+\,\frac{1}{N}\,\bigg)\,-\,\Phi_{\epsilon}\bigg(\,\frac{1}{N}\sum_{i=1}^{m}\eta_{z_{i}}\,\bigg)\;\le\;\frac{C_{0}}{N}\,\Big[\,\frac{\eta_{x}\eta_{y}}{N^{2}}\,+\,o_{N}(1)\,\Big]\;.
\]
Moreover, for all $\eta\in\mathscr{\mc{L}}^{x,\,y}$ and $1\le m<\kappa$,
\[
\Phi_{\epsilon}\bigg(\,\frac{1}{N}\sum_{i=1}^{m}\eta_{z_{i}}\,+\,\frac{1}{N}\,\bigg)\,-\,\Phi_{\epsilon}\bigg(\,\frac{1}{N}\sum_{i=1}^{m}\eta_{z_{i}}\,\bigg)\;=\;\frac{6}{N}\,\Big[\,\frac{\eta_{x}\eta_{y}}{N^{2}}\,+\,o_{N}(1)\,+\,O(\epsilon)\,\Big]\;,
\]
where $O(\epsilon)$ is a constant which depends on $\epsilon$ and
whose absolute value is bounded by $C_{0}\,\epsilon$.
\end{lem}

\begin{proof}
Fix $\eta\in\mathscr{\mc{J}}^{x,\,y}$. As $\Phi_{\epsilon}$
is non-decreasing, the first inequality holds. We consider the second
one. By definition of $\Phi_{\epsilon}$ and the mean-value theorem,
\begin{equation}
\Phi_{\epsilon}\bigg(\,\frac{1}{N}\sum_{i=1}^{m}\eta_{z_{i}}\,+\,\frac{1}{N}\,\bigg)\,-\,\Phi_{\epsilon}\left(\,\frac{1}{N}\sum_{i=1}^{m}\eta_{z_{i}}\,\right)\;=\;\frac{6}{N}\,\phi_{\epsilon}'(c)\,\phi_{\epsilon}(c)\,[\,1-\phi_{\epsilon}(c)\,]\;,\label{e0141}
\end{equation}
where $c\,=\,N^{-1}\sum_{1\le i\le m}\eta_{z_{i}}\,+\,(\delta/N)$
for some $0\le\delta\le1$. By definition of $\phi_{\epsilon}$, $1-\phi_{\epsilon}(c)=\phi_{\epsilon}(1-c)$.
Thus, as
\begin{gather*}
0\;\le\;c\,-\,\frac{\eta_{x}}{N}\;\le\;\frac{N-\eta_{x}-\eta_{y}+1}{N}\;\le\;\frac{\ell_{N}+1}{N}\;,\\
0\;\le\;(1-c)-\frac{\eta_{y}}{N}\;\le\;\frac{N-\eta_{x}-\eta_{y}}{N}\;\le\;\frac{\ell_{N}}{N}\;,
\end{gather*}
it follows from the uniform bound on $\Vert\phi_{\epsilon}'\Vert_{\infty}$,
that
\begin{align*}
 & \Phi_{\epsilon}\left(\,\frac{1}{N}\sum_{i=1}^{m}\eta_{z_{i}}\,+\,\frac{1}{N}\,\right)\,-\,\Phi_{\epsilon}\bigg(\,\frac{1}{N}\sum_{i=1}^{m}\eta_{z_{i}}\,\bigg)\\
 & \qquad\le\;\frac{C_{0}}{N}\,\Big[\,\phi_{\epsilon}\Big(\frac{\eta_{x}}{N}\Big)\,+\,o_{N}(1)\,\Big]\,\Big[\,\phi_{\epsilon}\Big(\frac{\eta_{y}}{N}\Big)\,+\,o_{N}(1)\,\Big]\;.
\end{align*}
Since $\phi'_{\epsilon}(t)\le1+\epsilon^{1/2}$ for all $0\le t\le1$
and $\phi_{\epsilon}(0)=0$, $\phi_{\epsilon}(t)\le(1+\epsilon^{1/2})t\le2t$.
This completes the proof of the first assertion of the lemma, as $\eta_{z}/N\le1$.

We turn to the second one. Fix a configuration $\eta$ in $\mc{L}^{x,\,y}$.
Since
\[
6\epsilon\;\le\;\frac{\eta_{x}}{N}\;\le\;\frac{1}{N}\sum_{i=1}^{m}\eta_{z_{i}}\;\le\;\frac{N-\eta_{y}}{N}\;\le\;1-6\epsilon\;,
\]
the constant $c$ belongs to the interval $[6\epsilon,1-(11/2)\epsilon]$
{[}provided $1/N\le\epsilon/2${]}, and $\phi_{\epsilon}'(c)=1/(1-8\epsilon)$.
On the other hand, since $\phi_{\epsilon}$ is linear on the interval
$[5\epsilon,1-5\epsilon]$ and $\eta_{x}$, $\eta_{y}\ge6\epsilon N$,
\begin{gather*}
\phi_{\epsilon}(c)\;=\;\phi_{\epsilon}\left(\frac{\eta_{x}}{N}\right)\,+\,o_{N}(1)\;=\;\frac{1}{1-8\epsilon}\,\Big(\frac{\eta_{x}}{N}\,-\,4\epsilon\,\Big)\,+\,o_{N}(1)\;,\\
\phi_{\epsilon}(1-c)\;=\;\phi_{\epsilon}\left(\frac{\eta_{y}}{N}\right)\,+\,o_{N}(1)\;=\;\frac{1}{1-8\epsilon}\,\Big(\frac{\eta_{y}}{N}\,-\,4\epsilon\,\Big)\,+\,o_{N}(1)\;.
\end{gather*}
To complete the proof of the second assertion, it remains to report
these estimates to the right-hand side of \eqref{e0141}.
\end{proof}
By the definitions of $V^{g}$, and $U_{x,y}$, given in \eqref{cv1},
\eqref{uxy}, respectively, for $\eta\in\mc{K}^{x,\,y}$, there
exists a finite constant $C_{0}$ such that for all $z$, $w\in S$,
\begin{equation}
\big|\,V^{g}(\sigma^{z,\,w}\eta)\,-\,V^{g}(\eta)\,\big|\;\le\;\frac{C_{0}}{N}\, o_{\epsilon}(1)\;.\label{ff03}
\end{equation}

The next result asserts that it is enough to estimate the Dirichlet
form on the sets $\mc{L}_{-}^{x,\,y}$, $x,\,y\in S$.

\begin{lem}
\label{lemc2}
We have that
\[
 \bb D_{N}(V^{g},\,F_{N};\mc{G}_{-})\;=\;\sum_{x,\,y\in S} \bb D_{N}(V^{g},\,F_{N};\mc{L}_{-}^{x,\,y})\;+\; o_{\epsilon}(1)   \;.
\]
\end{lem}

\begin{proof}
An argument, similar to the one presented to derive \eqref{ff02},
yields that the set $\mc{G}_{-}$ can be decomposed as
\[
\mc{G}_{-}\;=\;\bigcup_{x,\,y\in S}\mc{L}_{-}^{x,\,y}\;\cup\;\bigcup_{x\in S}\mc{V}_{-}^{x}\;\cup\;\bigcup_{x,\,y\in S}(\,\mc{K}_{+}^{x,\,y}\cap\mc{G}_{-}\,)\;.
\]

On the one hand,
\[
\bb D_{N}(V^{g},\,F_{N};\mc{V}_{-}^{x})\;=\;0
\]
because $V^{g}(\xi+\mf{d}^{z})=g(x)$ for all $\xi\in\mc{V}_{-}^{x}$
and $z\in S$.

On the other hand, by Schwarz inequality and the bound on $a_{N}$,
\begin{align*}
 & \bb D_{N}(V^{g},\,F_{N};\mc{K}_{+}^{x,\,y}\cap\mc{G}_{-})^{2}\\
 & \quad\le
 \;C_{0}\,\bb D_{N}(F_{N})\,\times \,\theta_N \sum_{\xi\in\mc{K}_{+}^{x,\,y}}\sum_{z,\,w\in S}\mu_{N-1}(\xi)\,\big[\,V^{g}(\xi+\mf{d}^{z})\,-\,V^{g}(\xi+\mf{d}^{w})\,\big]^{2}
\end{align*}
for some finite constant $C_{0}$. By Lemma \ref{np3} and \eqref{ff03},
this expression is bounded from above by
\[
\frac{C_{0}}{N^{2}}\,\theta_N  \,o_{\epsilon}(1) \,\mu_{N-1}(\mc{K}_{+}^{x,\,y})\;.
\]
By Lemmata \ref{lem0123} and \ref{lemcomp}, $\mu_{N-1}(\mc{K}_{+}^{x,\,y})\le C_{0}\,\mu_{N}(\mc{K}^{x,\,y})\le C_{0}/\log N$
for some finite constant $C_{0}$.

Putting together the previous estimates yields that
\[
\bb D_{N}(V^{g},\,F_{N};\mc{K}_{+}^{x,\,y}\cap\mc{G}_{-})\;\le\;C_{0} \,o_{\epsilon}(1)\, \;.
\]
This completes the proof of the lemma.
\end{proof}
It remains to compute the Dirichlet form on $\mc{L}_{-}^{x,\,y}$.
The proof of the next lemma is given in Section \ref{sec68}.
\begin{lem}
\label{lemc3} For $x,\,y\in S$,
\begin{equation*}
    \,\bb D_{N}(U_{x,\,y},\,F_{N};\mc{L}_{-}^{x,\,y})
  \;=\;\frac{r_{Z}(x,y)}{\kappa}\,\big[\,f_{N}(x)-f_{N}(y)\,\big]\,+\,o_{N}(1) ,+\,o_{\epsilon}(1)\;.
\end{equation*}
\end{lem}

\begin{proof}[Proof of Proposition \ref{pma}]
By definition \eqref{cv1} of $V^{g}$ on $\mc{J}^{x,\,y}$, we have
\[
 \,\bb D_{N}(V^{g},\,F_{N};\mc{L}_{-}^{x,\,y})\,=\, \,\big[g(x)-g(y)\big]\,\bb D_{N}(U_{x,\,y},\,F_{N};\mc{L}_{-}^{x,\,y})\;.
\]
By Lemma \ref{lemc3}, the right-hand side can be rewritten as
\begin{align*}
   \frac{r_{Z}(x,y)}{\kappa}\,\big[g(x)-g(y)\big]\,\big[\,f_{N}(x)-f_{N}(y)\,\big]
   \,+\, \,o_{N}(1)\,+\,o_{\epsilon}(1)\;.
\end{align*}
It remains to combine this estimate with Lemmata \ref{lemc1} and
\ref{lemc2}.
\end{proof}

\subsection{Proof of Lemma \ref{lemc3}\label{sec68}}

We start with a simple lemma which allows to bound a covariance between
two functions $F,\,G:\mc{E}_{N}^{x}\rightarrow\bb{R}$ in
terms of the Dirichlet form of one of them and the $L^{\infty}$-norm
of the other.
\begin{lem}
\label{lem025} There exists a finite constant $C_{0}$ such that,
for all $x\in\bb{R}$ and $F,\,G:\mc{E}_{N}^{x}\rightarrow\bb{R}$,
\[
\big|\,\mb{E}^N_{\mu_{N}^{x}}[F\,G]\,-\,\mb{E}^N_{\mu_{N}^{x}}[F]\,\mb{E}^N_{\mu_{N}^{x}}[G]\,\big|^{2}\;
\le\;\frac{C_{0}}{(\log N)^3} \,\Vert G\Vert_{\infty}^{2}\,\bb D_{N}(F)\;.
\]
\end{lem}

\begin{proof}
This lemma is a simple consequence of the local spectral gap estimate.
By the Cauchy-Schwarz inequality,
\begin{align*}
\big|\,\mb{E}^N_{\mu_{N}^{x}}[F\,G]\,-\,\mb{E}^N_{\mu_{N}^{x}}[F]\,\mb{E}^N_{\mu_{N}^{x}}[G]\,\big|^{2}\;\le\;\Vert G\Vert_{\infty}^{2}\text{Var}_{\mu_{N}^{x}}(F)\;,
\end{align*}
where the variance has been introduced in \eqref{ff04}. To complete
the proof, it remains to recall the local spectral gap, stated in
Theorem \ref{t29}.
\end{proof}
For $x,\,y\in S$, define

\begin{equation}
\mc{B}^{x,\,y}\;=\;\{\zeta\in\bb{N}^{S\setminus\{x,\,y\}}:|\zeta|\le\ell_{N}-1\}\;.\label{bbn}
\end{equation}

\begin{lem}
\label{lem026}For $x,\,y\in S$,
\[
\sum_{\zeta\in\mc{B}^{x,\,y}}\frac{1}{\mb{a}(\zeta)}\;=\;\big[\,1\,+\,o_{N}(1)\,\big]\,(\log N)^{\kappa-2}\;.
\]
\end{lem}

\begin{proof}
Set $\xi=(N-|\zeta|,\,\zeta)\in\mc{H}_{N,\,S\setminus\{x\}}$.
By Theorem \ref{t26},
\[
\lim_{N\rightarrow\infty}\frac{N}{Z_{N,\,S\setminus\{x\}}\,(\log N)^{\kappa-2}}\,\sum_{\zeta\in\mc{B}^{x,\,y}}\frac{1}{N-|\zeta|}\,\frac{1}{\mb{a}(\zeta)}\;=\;\frac{1}{\kappa-1}\;\cdot
\]
The assertion of the lemma follows from Proposition \ref{p31}.
\end{proof}

For $\zeta\in\bb{N}^{S\setminus\{x,\,y\}}$ with $|\zeta|\le N$, define
${\color{bblue}\eta_{\zeta}^{(i)}\in\mc{H}_{N}}$,
$0\le i\le N-|\zeta|$, as the configuration on $S$ with $N-i-|\zeta|$
particles at site $x$, $i$ particles at site $y$, and $\zeta_{z}$
particles at $z\in S\setminus\{x,\,y\}$:
\[
\eta_{\zeta}^{(i)}=(N-i-|\zeta|,\,i,\,\zeta)\;.
\]

\begin{lem}
\label{lem018} For all $x\not=y\in S$,
\begin{gather*}
\frac{1}{(\log N)^{\kappa-2}}\sum_{\zeta\in\mc{B}^{x,\,y}}\frac{1}{\mb{a}(\zeta)}\,F_{N}(\eta_{\zeta}^{(6\epsilon N)})\;=\;\big[\,1+o_{N}(1)\,\big]\,f_{N}(x)\;+\;o_{N}(1)\;,\\
\frac{1}{(\log N)^{\kappa-2}}\sum_{\zeta\in\mc{B}^{x,\,y}}\frac{1}{\mb{a}(\zeta)}\,F_{N}(\eta_{\zeta}^{(N-|\zeta|-6\epsilon N)})\;=\;\big[\,1+o_{N}(1)\,\big]\,f_{N}(y)\;+\;o_{N}(1)\;.
\end{gather*}
\end{lem}

\begin{proof}
We prove the first assertion, as the second one can be obtained by
symmetry. Fix $x,\,y\in S$. For $k\in\bb{N}$, let
\begin{equation}
\mc{B}_{k}^{x,y}\;=\;\{\zeta\in\bb{N}^{S\setminus\{x,\,y\}}:|\zeta|=k\}\;.\label{bxyk}
\end{equation}
For $0\le k<\ell_{N}$, define
\begin{equation}
c_{N}(k)\;=\;\frac{\log N}{N}\,\left(\,\sum_{i=0}^{\ell_{N}-k}\frac{1}{{a}(N-k-i)\,{a}(i)}\,\right)^{-1}\;,\label{cnk}
\end{equation}
and set $c_{N}(\ell_{N})=0$. Note that there exists a finite constant
$C_{0}$ such that
\begin{equation}
|c_{N}(k)|\,\le\,C_{0}\,\log N\text{ for all }0\le k\le\ell_{N}\;.\label{cbdd}
\end{equation}

Define
\[
\widetilde{f}_{N}(x)\;=\;\sum_{\eta\in\mc{E}_{N}^{x}}\mu_{N}(\eta)\,c_{N}(N-\eta_{x}-\eta_{y})\,F_{N}(\eta)\;.
\]
We claim that
\begin{equation}
\widetilde{f}_{N}(x)\;=\;\Big[\,\frac{1}{\kappa}\,+\,o_{N}(1)\,\Big]\,f_{N}(x)\,+\,o_{N}(1)\;,\label{cye1}
\end{equation}
and that
\begin{equation}
\frac{1}{Z_{N,\,S}\,(\log N)^{\kappa-2}}\sum_{\zeta\in\mc{B}^{x,\,y}}\frac{1}{\mb{a}(\zeta)}\,F_{N}(\eta_{\zeta}^{(6\epsilon N)})\;-\;\widetilde{f}_{N}(x)\;=\;o_{N}(1)\;.\label{cye2}
\end{equation}
The assertion of the lemma follows from these two identities and Proposition
\ref{p31}.

To prove the first claim, let
\[
d_{N}\;=\;\sum_{\eta\in\mc{E}_{N}^{x}}c_{N}(N-\eta_{x}-\eta_{y})\,\mu_{N}(\eta)\;.
\]
By definition of $\mu_{N}$ and $c_{N}$, as $c_{N}(\ell_{N})=0$, we can rewrite $d_N$ as
\begin{align*}
 d_N  &\;=\; \frac{N}{Z_{N,\,S}\,(\log N)^{\kappa-1}}\sum_{k=0}^{\ell_{N}}\sum_{\zeta\in\mc{B}_{k}^{x,y}}\sum_{i=0}^{\ell_{N}-k}\frac{c_{N}(k)}{{a}(N-k-i)\,{a}(i)\,\mb{a}(\zeta)}\\
  &\quad=\;\frac{N}{Z_{N,\,S}\,(\log N)^{\kappa-1}}\sum_{k=0}^{\ell_{N}-1}\sum_{\zeta\in\mc{B}_{k}^{x,y}}\frac{\log N}{N}\frac{1}{\mb{a}(\zeta)}
 \;=\;\frac{1}{Z_{N,\,S}\,(\log N)^{\kappa-2}}\sum_{\zeta\in\mc{B}^{x,\,y}}\frac{1}{\mb{a}(\zeta)}\;\cdot
\end{align*}
Hence, by Proposition \ref{p31} and Lemma \ref{lem026},
\[
d_{N}\;=\;\frac{1}{\kappa}\,+\,o_{N}(1)\;.
\]

To prove \eqref{cye1}, it remains to show that
\begin{equation}
\widetilde{f}_{N}(x)\,-\,d_{N}\,f_{N}(x)\,=\,o_{N}(1)\;.\label{mjm}
\end{equation}
Define $U:\mc{E}_{N}^{x}\rightarrow\bb{R}$ as $U(\eta)=c_{N}(N-\eta_{x}-\eta_{y})$,
so that
\[
\frac{1}{\mc{\mu}_{N}(\mc{E}_{N}^{x})}\,\big\{\,\widetilde{f}_{N}(x)\,-\,d_{N}f_{N}(x)\,\big\}\;=\;\mb{E}^N_{\mu_{N}^{x}}[F_{N}U]\,-\,\mb{E}^N_{\mu_{N}^{x}}[F_{N}]\,\mb{E}^N_{\mu_{N}^{x}}[U]\;.
\]
Thus, by Lemma \ref{lem025} and \eqref{cbdd},
\[
\left[\widetilde{f}_{N}(x)\,-\,d_{N}\,f_{N}(x)\right]^{2}\;\le\;\frac{C_{0}}{\log N}  \, \bb D_{N}(F_{N})\;
\]
for some finite constant $C_{0}$. By Lemma \ref{np3}, this
expression is bounded by $C_{0}/\log N$, which proves \eqref{mjm}
and \eqref{cye1}.

We turn to \eqref{cye2}. By definition, $\widetilde{f}_{N}(x)$ is
equal to
\[
\frac{N}{Z_{N,\,S}(\log N)^{\kappa-1}}\sum_{k=0}^{\ell_{N}-1}\sum_{\zeta\in\mc{B}_{k}^{x,y}}\sum_{i=0}^{\ell_{N}-k}\frac{c_{N}(k)}{\mb{a}(\zeta){a}(N-k-i){a}(i)}F_{N}(\eta_{\zeta}^{(i)})\;.
\]
On the other hand, by definitions of $\mc{B}^{x,\,y}$, $\mc{B}_{k}^{x,\,y}$
and $c_{N}(k)$, given in \eqref{bbn}, \eqref{bxyk} and \eqref{cnk},
respectively, we have
\begin{align*}
 & \frac{1}{Z_{N,\,S}(\log N)^{\kappa-2}}\sum_{\zeta\in\mc{B}^{x,\,y}}\frac{1}{\mb{a}(\zeta)}\,F_{N}(\eta_{\zeta}^{(6\epsilon N)})\\
 & \quad=\;\frac{N}{Z_{N,\,S}(\log N)^{\kappa-1}}\sum_{k=0}^{\ell_{N}-1}\sum_{\zeta\in\mc{B}_{k}^{x,y}}\frac{\log N}{N}\,\frac{1}{\mb{a}(\zeta)}\,F_{N}(\eta_{\zeta}^{(6\epsilon N)})\\
 & \quad=\;\frac{N}{Z_{N,\,S}(\log N)^{\kappa-1}}\sum_{k=0}^{\ell_{N}-1}\sum_{\zeta\in\mc{B}_{k}^{x,y}}\sum_{i=0}^{\ell_{N}-k}\frac{c_{N}(k)}{\mb{a}(\zeta){a}(N-k-i){a}(i)}\,F_{N}(\eta_{\zeta}^{(6\epsilon N)})\;.
\end{align*}
Therefore, the left-hand side of \eqref{cye2} is equal to
\begin{align*}
\sum_{\zeta\in\mc{B}^{x,\,y}}\sum_{i=0}^{\ell_{N}-|\zeta|}c_{N}(|\zeta|)\,\mu_{N}(\eta_{\zeta}^{(i)})\big[\,F_{N}(\eta_{\zeta}^{(6\epsilon N)})-F_{N}(\eta_{\zeta}^{(i)})\,\big]\;.
\end{align*}

In view of the previous expression, by the Cauchy-Schwarz inequality
and \eqref{cbdd}, the square of the left-hand side of \eqref{cye2}
is bounded by
\begin{equation}
C_{0}\,(\log N)^{2}\,\sum_{\zeta\in\mc{B}^{x,\,y}}\sum_{i=0}^{\ell_{N}-|\zeta|}\mu_{N}(\eta_{\zeta}^{(i)})\big[\,F_{N}(\eta_{\zeta}^{(6\epsilon N)})-F_{N}(\eta_{\zeta}^{(i)})\,\big]^{2}\label{f7}
\end{equation}
for some finite constant $C_{0}$. By the Cauchy-Schwarz inequality
again, the square inside the previous sum is less than or equal to
\begin{align*}
 & \sum_{j=i}^{6\epsilon N-1}\frac{1}{{a}(j)\,{a}(N-|\zeta|-j)}\big[\,F_{N}(\eta_{\zeta}^{(j)})\,-\,F_{N}(\eta_{\zeta}^{(j+1)})\,\big]^{2}\sum_{j=i}^{6\epsilon N-1}{a}(j)\,{a}(N-|\zeta|-j)\\
 & \quad\le\;C_{0}\,N^{2}\,\sum_{j=i}^{6\epsilon N-1}\frac{1}{{a}(j)\,{a}(N-|\zeta|-j)}\big[\,F_{N}(\eta_{\zeta}^{(j)})-F_{N}(\eta_{\zeta}^{(j+1)})\,\big]^{2}
\end{align*}
for some finite constant $C_{0}$. The sum \eqref{f7} is thus bounded
above by
\begin{align*}
 & C_{0}\,(N\,\log N)^{2}\,\sum_{\zeta\in\mc{B}^{x,\,y}}\sum_{i=0}^{\ell_{N}-|\zeta|}\sum_{j=i}^{6\epsilon N-1}\frac{\mu_{N}(\eta_{\zeta}^{(j)})}{{a}(i)\,{a}(N-|\zeta|-i)}\big[\,F_{N}(\eta_{\zeta}^{(j)})-F_{N}(\eta_{\zeta}^{(j+1)})\,\big]^{2}\\
 & \quad\le\;C_{0}\,N\,(\log N)^{2}\sum_{\zeta\in\mc{B}^{x,\,y}}\sum_{i=0}^{\ell_{N}-|\zeta|}\sum_{j=i}^{6\epsilon N-1}\frac{1}{{a}(i)}\,\mu_{N}(\eta_{\zeta}^{(j)})\big[\,F_{N}(\eta_{\zeta}^{(j)})-F_{N}(\eta_{\zeta}^{(j+1)})\,\big]^{2}\;.
\end{align*}
Changing the order of summations this expression becomes
\[
C_{0}\,N\,(\log N)^{2}\sum_{\zeta\in\mc{B}^{x,\,y}}\sum_{j=0}^{6\epsilon N-1}\mu_{N}(\eta_{\zeta}^{(j)})\,\big[\,F_{N}(\eta_{\zeta}^{(j)})-F_{N}(\eta_{\zeta}^{(j+1)})\,\big]^{2}\sum_{i=0}^{A_{N}(j,\zeta)}\frac{1}{{a}(i)}\,\;,
\]
where $A_{N}(j,\zeta)=\min\{j,\,\ell_{N}-|\zeta|\}$. The last summation over $i$
is bounded by $C_{0}\,\log\ell_{N}\le C_{0}\,\log N$. Hence, by Lemma
\ref{lemb}, this expression is less than or equal to
$C_{0} \,(\log N)^{2}/N$ $\bb D_{N}(F_{N})$, which, by
Lemma \ref{np3}, is bounded by $C_{0}\,(\log N)^{2}/N$, which
proves \eqref{cye2}.
\end{proof}
\begin{proof}[Proof of Lemma \ref{lemc3}]
Fix $x,\,y\in S$, and recall the definition of $U_{x,\,y}$ given in
\eqref{uxy} and the one of the sequence $(z_{i})_{i=1}^{\kappa}$
introduced in \eqref{enum1}. With this notation, we can write
$ \bb D_{N}(U_{x,\,y},\,F_{N};\mc{L}_{-}^{x,\,y})$ as
\[
\frac{\theta_{N}a_{N}}{2}\,\sum_{i,\,j=1}^{\kappa}\,\sum_{\xi\in\mc{L}_{-}^{x,\,y}}\mu_{N-1}(\xi)\,r(z_{i},\,z_{j})\,(T_{i,j}U_{x,y})(\xi)\,(T_{i,j}F_{N})(\xi)\;,
\]
where $(T_{i,j}G)(\xi)=G(\xi+\mf{d}^{z_{j}})-G(\xi+\mf{d}^{z_{i}})$.

Assume that $i>j$. By definition, we can write $T_{i,j}U_{x,y}(\xi)$
as
\[
\sum_{n=j}^{i-1}\,[\,h_{x,\,y}(z_{n})-h_{x,\,y}(z_{n+1})\,]\,\bigg[\,\Phi_{\epsilon}\bigg(\,\frac{1}{N}\sum_{k=1}^{n}\xi_{z_{k}}+\frac{1}{N}\bigg)\,-\,\Phi_{\epsilon}\bigg(\frac{1}{N}\sum_{k=1}^{n}\xi_{z_{k}}\bigg)\bigg]\;.
\]
By the second assertion of Lemma \ref{lem014}, this sum is equal
to
\begin{align*}
 & \frac{6}{N}\sum_{n=j}^{i-1}\,[\,h_{x,\,y}(z_{n})-h_{x,\,y}(z_{n+1})\,]\,\Big(\,\frac{\xi_{x}\xi_{y}}{N^{2}}\,+\,o_{N}(1)\,+\,O(\epsilon)\,\Big)\\
 & \quad=\;\frac{6}{N}\,\,[\,h_{x,\,y}(z_{j})-h_{x,\,y}(z_{i})\,]\,\Big(\,\frac{\xi_{x}\xi_{y}}{N^{2}}\,+\,o_{N}(1)\,+\,O(\epsilon)\,\Big)\;.
\end{align*}
A similar identity holds for $i<j$.

Therefore,
\[
 \bb D_{N}(U_{x,\,y},\,F_{N};\mc{L}_{-}^{x,\,y})\;=\;I_{1}\;+\;I_{2}\;,
\]
where
\begin{gather*}
I_{1}\;=\;\frac{3\theta_{N}a_{N}}{N^{3}}\sum_{\xi\in\mc{L}_{-}^{x,\,y}}\sum_{i,\,j=1}^{\kappa}\mu_{N-1}(\xi)\,r(z_{i},z_{j})\,\xi_{x}\xi_{y}\,\,[\,h_{x,\,y}(z_{j})-h_{x,\,y}(z_{i})\,]\,(T_{i,j}F_{N})(\xi)\;,\\
I_{2}\;=\;\frac{[o_{N}(1)+O(\epsilon)]\,\theta_{N}}{N}\sum_{\xi\in\mc{L}_{-}^{x,\,y}}\sum_{i,\,j=1}^{\kappa}\mu_{N-1}(\xi)\,[\,h_{x,\,y}(z_{j})-h_{x,\,y}(z_{i})\,]\,(T_{i,j}F_{N})(\xi)\;.
\end{gather*}

The second term is easy to estimate. By the Cauchy-Schwarz inequality,
its square is bounded by
\[
\frac{[\,o_{N}(1)+O(\epsilon)\,]\,\theta_{N} }{N^{2}}\,
\mu_{N-1}(\mc{L}_{-}^{x,\,y})\,\bb D_{N}(F_{N})\;.
\]
By definition of $\mc{L}_{-}^{x,\,y}$, $\mc{L}_{+}^{x,\,y}$,
and by Lemmata \ref{lem013} and \ref{lemcomp},
\[
\mu_{N-1}(\mc{L}_{-}^{x,\,y})\;\le\;\mu_{N-1}(\mc{L}_{+}^{x,\,y})
\;\le\;C_{0}\,\mu_{N}(\mc{L}^{x,\,y})
\;\le\;\frac{C_{0}}{\log N}\log\frac{1}{\epsilon}
\]
for some finite constant $C_{0}$. Hence, by Lemma \ref{np3},
\[
I_{2}\;=\;o_{N}(1)\,+\,o_{\epsilon}(1)\;.
\]

We turn to $I_{1}$. Write $\xi$ as $(\xi_{x},\,\xi_{y},\,\zeta)$
for $\zeta\in\bb{N}^{S\setminus\{x,\,y\}}$. Then, $I_{1}$ is
equal to
\begin{align*}
 & \frac{3\,\theta_{N}(N-1)\,a_{N}}{N^{3}\,Z_{N-1,\,S}\,[\log(N-1)]^{\kappa-1}}\\
 & \quad\times\;\sum_{\xi\in\mc{L}_{-}^{x,\,y}}\sum_{i,\,j=1}^{\kappa}\frac{1}{\mb{a}(\zeta)}\,r(z_{i},z_{j})\,\{\,h_{x,\,y}(z_{j})-h_{x,\,y}(z_{i})\,\}\,(T_{i,j}F_{N})(\xi)\;.
\end{align*}
By Proposition \ref{p31}, by definition of $\theta_{N}$ and by
\eqref{ff05}, we may rewrite this expression as
\begin{align*}
 & \frac{6\,[1+o_{N}(1)]}{\kappa\,(\log N)^{\kappa-2}}\sum_{\xi\in\mc{L}_{-}^{x,\,y}}\sum_{i=1}^{\kappa}\frac{F_{N}(\xi+\mf{d}^{z_{i}})}{\mb{a}(\zeta)}\sum_{j=1}^{\kappa}r(z_{i},z_{j})\,\left\{ h_{x,\,y}(z_{i})-h_{x,\,y}(z_{j})\right\} \\
 & \quad=\;\frac{6\,[1+o_{N}(1)]}{\kappa\,(\log N)^{\kappa-2}}\sum_{\xi\in\mc{L}_{-}^{x,\,y}}\sum_{i=1}^{\kappa}\frac{F_{N}(\xi+\mf{d}^{z_{i}})}{\mb{a}(\zeta)}(-L_{X}h_{x,\,y})(z_{i})\;.
\end{align*}
By \eqref{ff20},
\[
(L_{X}h_{x,\,y})(z)\;=\;\begin{cases}
-\,\kappa\,\Cap_{X}(x,\,y) & \text{if }z=x\\
\kappa\,\Cap_{X}(x,\,y) & \text{if }z=y\\
0 & \text{otherwise\;.}
\end{cases}
\]
Thus, by the definition \eqref{bexy} of $r_{Z}(x,y)$,
\begin{align*}
I_{1}\;=\;\frac{r_{Z}(x,y)[1+o_{N}(1)]}{\kappa\,(\log N)^{\kappa-2}}\sum_{\xi\in\mc{L}_{-}^{x,\,y}}\frac{1}{\mb{a}(\zeta)}\,\left[F_{N}(\xi+\mf{d}^{x})-F_{N}(\xi+\mf{d}^{y})\right]\;.
\end{align*}

Recall from \eqref{bbn} the definition of the set $\mc{B}^{x,\,y}$
and that $\eta_{\zeta}^{(i)}$ represents the configuration $(N-|\zeta|-i,\,i,\,\zeta)\in\mc{H}_{N}$.
With this notation, the set $\mc{L}_{-}^{x,\,y}$ can be represented
as
\[
\mc{L}_{-}^{x,\,y}\,=\,\big\{(N-1-|\zeta|-i,\,i,\,\zeta):\zeta\in\mc{B}^{x,\,y}\;\;,\;\;6N\epsilon\le i<N-|\zeta|-6N\epsilon\big\}\;.
\]
Therefore, we can write $I_{1}$ as
\begin{align*}
 & \frac{r_{Z}(x,y)[1+o_{N}(1)]}{\kappa\,(\log N)^{\kappa-2}}\sum_{\zeta\in\mc{B}^{x,\,y}}\frac{1}{\mb{a}(\zeta)}\sum_{i=6N\epsilon}^{N-|\zeta|-6N\epsilon-1}\left[F_{N}(\eta_{\zeta}^{(i)})-F_{N}(\eta_{\zeta}^{(i+1)})\right]\\
 & \quad=\;\frac{r_{Z}(x,y)[1+o_{N}(1)]}{\kappa\,(\log N)^{\kappa-2}}\sum_{\zeta\in\mc{B}^{x,\,y}}\frac{1}{\mb{a}(\zeta)}\left[F_{N}(\eta_{\zeta}^{(6N\epsilon)})-F_{N}(\eta_{\zeta}^{(N-|\zeta|-6N\epsilon)})\right]\;.
\end{align*}
Thus, by Lemma \ref{lem018},
\[
I_{1}=\frac{r_{Z}(x,y)}{\kappa}\,[\,f_{N}(x)-f_{N}(y)\,]\,+\,o_{N}(1)\,+\,o_{N}(1)\;,
\]
which completes the proof of the lemma.
\end{proof}

\section{Attractor sets in the wells}
\label{sec8}

The proof of Proposition \ref{p74} is divided in two steps. We first
show that starting from a configuration $\eta$ in
$\mathcal{E}_{N}^{x}$, the process hits the set $\mathcal{D}_{N}^{x}$
before it leaves the large well $\mathcal{W}_{N}^{x}$. The proof of
this result requires the construction of a super-harmonic function on
$\mathcal{W}_{N}^{x} \setminus \mathcal{E}_{N}^{x}$, a technical and
difficult step presented in the next section.  Then, we show that
starting from $\mathcal{D}_{N}^{x}$, the process visits all
configurations of this set before hitting a new well
$\mathcal{E}_{N}^{y}$.

\subsection{Deep wells are attractors}
\label{sec81}

The next result asserts that starting from $\mc E^x_N$ the process
hits the deep well $\mc D^x_N$ before leaving $\mc W^x_N$.

\begin{prop}
\label{p81}
For all $x\in S$,
\begin{equation*}
\lim\limits _{N\rightarrow\infty}\inf
\limits _{\eta\in\mathcal{E}_{N}^{x}}
\mathbf{P}_{\eta}^{N} [\, \tau_{\mathcal{D}_{N}^{x}}
\,<\, \tau_{(\mathcal{W}_{N}^{x})^{c}}\, ] \;=\; 1\;.
\end{equation*}
\end{prop}

The proof of this proposition is based on the existence of a
super-harmonic function in
$\mathcal{W}_{N}^{x}\setminus\mathcal{D}_{N}^{x}$, presented in the
next section.

\begin{thm}
\label{p82}
Fix $x\in S$. There exist positive, finite constants $c_0$, $c_{1}$,
$c_{2}$ and a function
$G_{N}^{x}:\mathcal{H}_{N}\rightarrow\mathbb{R}$ such that,
\begin{gather}
(\mathscr{L}_{N} G_{N}^{x}) (\eta) \;\le
-\, \frac{c_0 \theta_N}{N-\eta_x } <\; 0\;, \label{e821}\\
\text{and} \;\;
c_{1}\, (N-\xi_{x}) \;\le\; G_{N}^{x}(\xi)\label{e822}
\;\le\; c_{2}\, (N-\xi_{x})
\end{gather}
for all $\eta\in \mathcal{W}_{N}^{x}\setminus\mathcal{D}_{N}^{x}$,
$\xi\in \overline{\mathcal{W}_{N}^{x}\setminus\mathcal{D}_{N}^{x}}$
and large enough $N$.
\end{thm}

\begin{proof}[Proof of Proposition \ref{p81}.]
In view of Theorem \ref{p82} and Lemma \ref{nl2}, it is enough to
check that \eqref{n14} holds. By \eqref{e822} and the definition of
the sets $\mathcal{D}_{N}^{x}$, $\mathcal{E}_{N}^{x}$ and
$\mathcal{W}_{N}^{x}$, the ratio in \eqref{n14} is bounded by
\begin{align*}
\frac{c_{2} \, \ell_N \,-\, c_{1}\,N^{\gamma}}
{c_{1} \, N/(\log N)^{\beta}
\,-\, c_{1}\, N^{\gamma}}\;, \quad
\eta\in\mathcal{E}_{N}^{x}\setminus\mathcal{D}_{N}^{x}\;.
\end{align*}
Since $\ell_{N}= N/\log N$, $0< \beta<1$, and $0<\gamma<1$,
the previous expression vanishes as $N\to\infty$. This completes the
proof of the proposition.
\end{proof}

\subsection{Visiting points  in deep wells}
\label{sec82}

The main result of this section, Proposition \ref{p83}, asserts that
starting from a deep well $\mathcal{D}_{N}^{x}$ the process visits
all configurations in $\mathcal{D}_{N}^{x}$ before hitting a new well
$\mathcal{E}_{N}^{y}$.  This result is a weak version of Proposition
\ref{p74}, as it requires the process to start from
$\mathcal{D}_{N}^{x}$ instead of $\mathcal{E}_{N}^{x}$.

The proof of Proposition \ref{p83} is based on a classical bound of
equilibrium potentials in terms of capacities. We first provide a
lower bound on the capacities between configurations in
$\mathcal{D}_{N}^{x}$.

\begin{lem}
\label{lem84} Fix $x\in S$. There exists a positive constant $c_0$
such that for all $\xi,\,\eta\in\mathcal{D}_{N}^{x}$ and $N\ge 1$,
\begin{equation*}
\Cap_{N}(\xi,\,\eta) \;\geq\;
\frac{c_0 \theta_N }{N^{\gamma\kappa}\, (\log N)^{\kappa -1}}\;\cdot
\end{equation*}
\end{lem}

\begin{proof}
Fix $\xi,\,\eta\in\mathcal{D}_{N}^{x}$. Consider a sequence
$\xi=\zeta^{(0)},\,\zeta^{(1)},\,\dots,\,\zeta^{(p)}=\eta$ in
$\mathcal{D}_{N}^{x}$ such that
$\zeta^{(k+1)}=\sigma^{x_{k},\,y_{k}}\zeta^{(k)}$ for some
$x_{k},\,y_{k}\in S$ satisfying $r(x_{k},\,y_{k})>0$. Since there are
at most $N^\gamma$ particles on $S\setminus \{x\}$, there exists such
a sequence with length bounded by $C_0 N^\gamma$:
\begin{equation*}
p\le C_0\, N^{\gamma}
\end{equation*}
for some finite  constant $C_0$.

Let $F:\mathcal{H}_{N}\rightarrow\mathbb{R}$ be a function such
that $F(\xi)=0$ and $F(\eta)=1$. By Cauchy-Schwarz inequality,
there exists a finite constant $C_0$ such that
\begin{equation*}
1 \;=\; \bigg\{\,
\sum_{k=0}^{p-1} \,[\, F(\zeta^{(k+1)}) \,-\, F(\zeta^{(k)})\,]\,
\, \bigg\}^{2} \;\leq\;
 C_0\, \theta_N^{-1}\, \bb D_{N}(F)\,
\sum_{k=0}^{p-1}\frac{1}{\mu_{N}(\zeta^{(k)})}\;\cdot
\end{equation*}
Thus, by the Dirichlet principle,
\begin{equation*}
\Cap_{N}(\xi,\,\eta) \;\geq\; c_0\,\theta_N\,
\bigg(\, \sum_{k=0}^{p-1}\frac{1}{\mu_{N}(\zeta^{(k)})}
\, \bigg)^{-1}\;.
\end{equation*}

By definition of the set $\mathcal{D}_{N}^{x}$,
$\mathbf{a}(\zeta) \;\le\; N\, (N^{\gamma})^{\kappa-1} \;=\;
N^{1+(\kappa-1)\gamma}$ for $\zeta\in\mathcal{D}_{N}^{x}$.  Hence, by
the explicit formula for the invariant measure and Proposition
\ref{p31}, there exists a positive constant $c_0$ such that
\begin{equation*}
\mu_{N}(\zeta) \;\ge\; c_0\, \frac{N}{(\log N)^{\kappa-1}}
\, \frac{1}{N^{1+(\kappa-1)\gamma}} \;=\;
\frac{1}{N^{\gamma(\kappa-1)}(\log N)^{\kappa-1}}\; \cdot
\end{equation*}
To complete the proof, it remains to put together all previous
estimates.
\end{proof}

The bound produced by this argument in the case where $\xi$ belongs to
$\mathcal{D}_{N}^{x}$ and $\eta$ to $\mathcal{E}_{N}^{x}$ is too crude
to prove Proposition \ref{p83} below with
$\eta\in \mathcal{E}_{N}^{x}$, instead of
$\eta\in \mathcal{D}_{N}^{x}$.

\begin{prop}
\label{p83}
For all $x\in S$,
\begin{equation*}
\lim_{N\rightarrow\infty}
\inf_{\eta\in\mathcal{D}_{N}^{x}}
\inf_{\zeta\in\mathcal{D}_{N}^{x}}
\mathbf{P}_{\eta}^{N} \big[
\tau_{\zeta}<\tau_{\check{\mathcal{E}}_{N}^{x}}\,\big]
\;=\; 1\;.
\end{equation*}
\end{prop}

\begin{proof}
By \cite[equation (3.3)]{LL} and the monotonicity of the capacity,
\begin{equation*}
\mathbf{P}_{\eta}^{N}\big[\, \tau_{\zeta} \,>\,
\tau_{\check{\mathcal{E}}_{N}^{x}}\big] \;\le\;
\frac{\Cap_{N}(\eta,\,\check{\mathcal{E}}_{N}^{x})}
{\Cap_{N}(\eta,\,\zeta)} \;\le\;
\frac{\Cap_{N}(\mathcal{E}_{N}^{x},\,\check{\mathcal{E}}_{N}^{x})}
{\Cap_{N}(\eta,\,\zeta)}\;.
\end{equation*}
Thus, by Proposition \ref{pp71} and Lemma \ref{lem84},
\begin{equation*}
\mathbf{P}_{\eta}^{N}\big[\, \tau_{\zeta} \,>\,
\tau_{\check{\mathcal{E}}_{N}^{x}}\big] \;\le\;
C_0 \frac{N^{\gamma\kappa}(\log N)^{\kappa -1}}{\theta_N}\;.
\end{equation*}
Since, by hypothesis, $\gamma<2/\kappa$, this expression vanishes as
$N\to\infty$, as claimed.
\end{proof}

\begin{rem}
Proposition \ref{p83} is enough to derive Condition (C2), the full
content of Proposition \ref{p74} is not needed. See the proof at the
end of Section \ref{sec7}. The statement of Proposition \ref{p74}, is
however interesting, as it asserts that starting from a well
$\mc E^x_N$, the process visits all points of the deep well
$\mc D^x_N$ before hitting a new well.
\end{rem}

\begin{proof}[Proof of Proposition \ref{p74}]
Fix $x\in S$, $\eta\in \mc E^x_N$, $\zeta\in \mc D^x_N$.
By the strong Markov property,
\begin{align*}
\mathbf{P}_{\eta}^{N} [\,
\tau_{\zeta} \,<\, \tau_{\check{\mathcal{E}}_{N}^{x}}\,]
\; & \ge\; \mathbf{P}_{\eta}^{N}[\, \tau_{\mathcal{D}_{N}^{x}}
\,<\, \tau_{\check{\mathcal{E}}_{N}^{x}} \,,\,
\tau_{\zeta} \,<\, \tau_{\check{\mathcal{E}}_{N}^{x}} \, ] \\
& \geq\, \mathbf{P}_{\eta}^{N}[\, \tau_{\mathcal{D}_{N}^{x}}
\,<\, \tau_{\check{\mathcal{E}}_{N}^{x}} \, ]
\, \inf_{\xi\in\mathcal{D}_{N}^{x}}
\mathbf{P}_{\xi}^{N}[\, \tau_{\zeta}<\tau_{\check{\mathcal{E}}_{N}^{x}}\,]\;.
\end{align*}
Optimizing  over $\eta\in\mathcal{E}_{N}^{x}$ yields that
\begin{equation*}
\inf_{\eta\in\mathcal{E}_{N}^{x}}
\mathbf{P}_{\eta}^{N} [\,
\tau_{\zeta} \,<\, \tau_{\check{\mathcal{E}}_{N}^{x}}\,] \;\ge\;
\inf_{\eta\in\mathcal{E}_{N}^{x}}
\mathbf{P}_{\eta}^{N}[\, \tau_{\mathcal{D}_{N}^{x}}
\,<\, \tau_{(\mathcal{W}_{N}^{x})^{c}} \, ] \,
\inf_{\xi\in\mathcal{D}_{N}^{x}}
\mathbf{P}_{\xi}^{N}[\, \tau_{\zeta}<\tau_{\check{\mathcal{E}}_{N}^{x}}\,]\;.
\end{equation*}
because
$\tau_{(\mathcal{W}_{N}^{x})^{c}}<\tau_{\check{\mathcal{E}}_{N}^{x}}$.
To complete the proof, it remains to recall the statements of
Propositions \ref{p81} and  \ref{p83}.
\end{proof}

\section{A super-harmonic function}
\label{sec9}

In this chapter, we prove Theorem \ref{p82}. For the convenience of
notation, we will now work with the generator $\mathscr{A}_N$ of the
original zero-range process, instead of the speeded-up generator
$\mathscr{L}_N$.

The super-harmonic function $G_N^x$ is introduced in Section
\ref{sec94}. We explain below the ideas behind its construction. To
propose candidates, one interprets the zero-range process as a random
walk on the simplex $\mc H_N$.

Fix $x_0\in S$, and denote by $\color{bblue} \mc H^{x_0}_N$ the subset
of $\mc H_N$ of all configurations such that
$N- \alpha_N \le \eta_{x_0} \le N - \beta_N$, where
$\beta_N \ll \alpha_N \ll N$ are two sequences. This means that all
coordinates $\eta_x$ are much smaller than $\eta_{x_0}$ on the set
$\mc H^{x_0}_N$.

One wishes to show that $\sum_{x\not = x_0} \eta_x$ decreases with
time in this set. This is done by constructing an increasing function
$F:\bb N \to \bb R$ such that
$(\ms{A}_N F)(\sum_{x\not = x_0} \eta_x) \le 0$ on the set
$\mc H^{x_0}_N$. In fact, it is not difficult to find functions which
are super-harmonic in the interior of $\mc H^{x_0}_N$ [the points
$\eta$ in this set such that $\eta_x>0$ for all $x$]. Indeed, in the
interior, it is clear that $\sum_{x\not = x_0} \eta_x$ decreases in
time because the rate of a jump from $x_0$ to $x$ is strictly smaller
than the rate of a jump from $x$ to $x_0$. The problem occurs at the
boundary. The sum $\sum_{x\not = x_0} \eta_x$ may increase due to a
jump from $x_0$ to a site $x$ such that $\eta_x=0$, and the reverse
jumps are forbidden.

In the diffusive scale the random walk should converge weakly to a
diffusion on a continuous simplex. Denote by $\xi^{x_0}$ the corner of
this simplex which corresponds to the configuration in which all
particles sit at site $x_0$. One can write down the drift of this
diffusion and define a $(|S|-2)$-dimensional manifold with the
property that at any point of this manifold the scalar product of the
drift of the diffusion with the normal vector to the manifold [which
point towards the corner] is positive. For $|S|=3$ or $4$, one can
draw pictures of the vector field induced by the drift to create an
intuition.  We refer to Figure \ref{fig2} for an illustration of case
$|S|=3$.

\begin{figure}
\protect \includegraphics[scale=0.16]{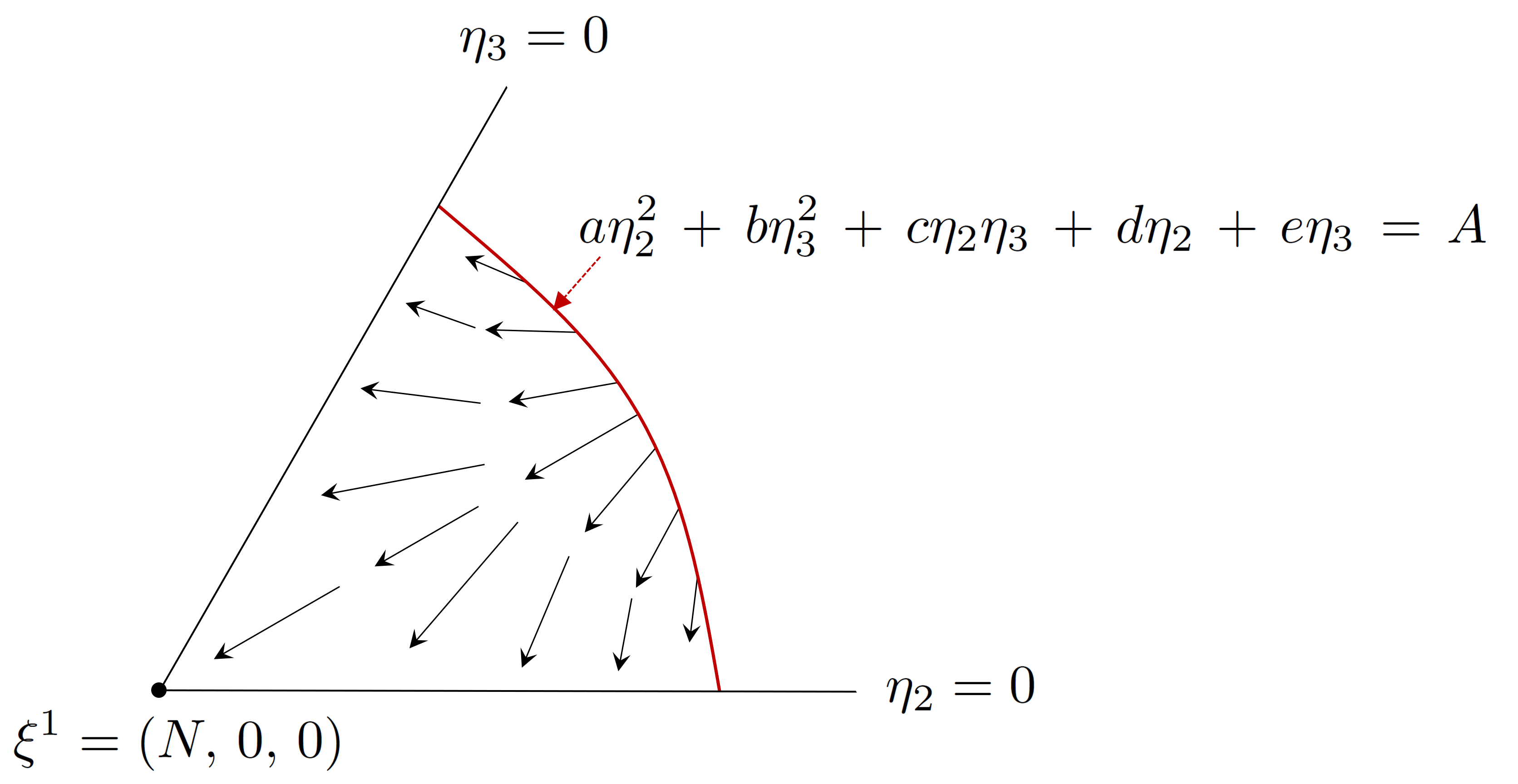}\protect
\caption{ An illustration of the drift of the diffusion which
approximates the zero-range dynamics when $N$ is large in the case
where $S=\{1,\,2,\,3\}$.
The red curve represents the manifold $\mc M_A$.
\label{fig2}}
\end{figure}

A good choice for this manifold is the one given by
\begin{equation*}
\mc M_A \;=\; \bigg\{\, \eta \in \mc H^{x_0}_N :
\sum_{x,y\in S\setminus \{x_0\}} a_{x,y} \eta_x\, \eta_y
\;+\; \sum_{x\not = x_0} b_{x} \eta_x \,=\,  A\, \bigg\}
\end{equation*}
for appropriate coefficients. Each value of $A$ gives a different
manifold. The corresponding function should be constant on each
manifold and a natural candidate emerges:
$F = F(\sum_{x,y\in S\setminus \{x_0\}} a_{x,y} \eta_x\, \eta_y +
\sum_{x\not = x_0} b_{x} \eta_x)$.

This is how the function $P$, introduced in Lemma \ref{lem99},
emerges.  By Proposition \ref{p97} and the proof of Proposition
\ref{p917},
\begin{equation*}
(\ms {A}_{N}P^{1/2})(\eta)
\;\le\; \frac{1}{P(\eta)^{1/2}}
\, \Big\{\, -\, 1 \,+\,
\sum_{x\in S_{0}}\mathbf{1}\{\eta_{x}=1\} \,\Big\}\;,
\end{equation*}
where $\color{bblue} S_0 = S \setminus \{x_0\}$.  Thus, $P^{1/2}$ is
super-harmonic except when there is more than one coordinate with only
one particle.

To modify this function at the boundary, we introduce functions
$P^{A}$, $A\subset S_0$, which, by the second assertion of Proposition
\ref{p97}, eliminate the positive part of $(\ms {A}_{N}P^{1/2})(\eta)$
if the configuration $\eta$ has two or more particles at the sites in
$A^c$.  More precisely, in Section \ref{sec94} and below, we prove
that there exists a constant $c>0$ such that
\begin{equation*}
(\ms {A}_{N}(P-P^{A}+a )^{1/2})
(\eta) \;\le\;\frac{-\, c}{P(\eta)^{1/2}}
\end{equation*}
for all $a>0$, provided that $\eta_{x}\ge2$ for all $x\in A^c$.

Therefore, the functions $(P-P^{A}+a)^{1/2}$ are super-harmonic in
different regions of the space, and the union of these regions
contains the annulus $\mc H^{x_0}_N$. We use these functions to define
one on $\mc H^{x_0}_N$. The problem occurs at the boundary of these
regions. This obstacle is circumvented by averaging these functions
over the free constant $a$.

\subsection{Potential theory of underlying random walk}
\label{sec91}

Recall the definition of equilibrium potential \eqref{PE} and the one
of capacity \eqref{CAP} for the underlying random walk.

\begin{lem}
\label{lem91}
Let $B$ be a non-empty subset of $S$ and let $x,\,y\in S\setminus B$.
Then,
\begin{equation*}
\frac{h_{x,\,B}(y)}{\Cap_{X}(x,\,B)} \;=\;
\frac{h_{y,\,B}(x)}{\Cap_{X}(y,\,B)}\;.
\end{equation*}
\end{lem}

\begin{proof}
Recall that we denote by $\mathbb{P}_{x}$ the probability on the path
space $D(\bb R_+, S)$ induced by the random walk $X(t)$ starting from
$x$, and by $\bb E_x$ the expectation with respect to $\mathbb{P}_{x}$.

By \cite[Proposition 6.10]{BL1},
\begin{gather*}
\mathbb{E}_{x}\Big[\, \int_{0}^{\tau_{B}} \chi_{\{y\}} (X(t))\, \ dt\,
\Big]
\;=\; \frac{\langle \chi_{\{y\}} \,,\, h_{x,\,B} \rangle _{m}}
{\Cap_{X}(x ,\,B)} \;=\; \frac{m(y)\,h_{x,\,B}(y)}
{\Cap_{X}(x ,\,B)}\;,\\
\mathbb{E}_{y}\Big[\, \int_{0}^{\tau_{B}} \chi_{\{x\}}(X(t))
\ dt\,\Big] \;=\;\frac{\langle \chi_{\{x\}} \,,\, h_{y,\,B}\rangle
  _{m}}
{\Cap_{X}(y,\,B)} \;=\; \frac{m(x)\,h_{y,\,B}(x)}{\Cap_{X}(y,\,B)}\;.
\end{gather*}

It remains to show that
\begin{equation*}
m(x)\, \mathbb{E}_{x}\Big[\, \int_{0}^{\tau_{B}}
\chi_{\{y\}} (X(t)) \ dt\,\Big] \;=\;
m(y)\, \mathbb{E}_{y}\Big[\, \int_{0}^{\tau_{B}}
\chi_{\{x\}} (X(t)) \ dt\,\Big]\;.
\end{equation*}
Denote by $(Y(n))_{n\in\mathbb{N}}$ the embedded, discrete-time Markov
chain. Recall that $Y(n)$ is a $S$-valued chain which jumps from $x$
to $y$ with probability $p(x\,,\,y)=r(x\,,\,y)/\lambda(x)$, where
$\lambda(x)=\sum_{y\in S}r(x,\,y)$, and that its invariant measure,
denoted by $M$, is given by $M(x) \, =\, m(x)\, \lambda(x)$.

Let {\color{bblue} $\mf e_{k}$, $k\ge 0$}, be a sequence of independent,
mean-one exponential random variables, independent of the chain
$Y(n)$.  Denote by $\color{bblue}\mathbb{E}^{Y,\mf e}_{x}$ the expectation with
respect to the chain $Y(n)$ starting from $x\in S$ and the sequence
$(\mf e_{n})_{n\in\mathbb{N}}$. With this notation,
\begin{equation*}
\mathbb{E}_{x}\left[\, \int_{0}^{\tau_{B}}\chi_{\{y\}} (X(t))\
dt\,\right]
\;=\; \mathbb{E}^{Y, \mf e}_{x}\left[\, \sum_{n=0}^{\tau_{B}-1}
\bs 1\{Y(n)=y\}\, \frac{\mf e_{n}}{\lambda(Y(n))}\,\right]\;.
\end{equation*}
Replacing in the denominator $Y(n)$ by $y$, and then integrating over
$\mf e_k$, yields that the right-hand side is equal to
\begin{equation*}
\frac{1}{\lambda(y)}\, \mathbb{E}^{Y, \mf e}_{x}\bigg[\,
\sum_{n\ge 0} \bs 1 \{Y(n)=y,\,n<\tau_{B}\}\, \mf e_{n} \,\bigg]
\;=\; \frac{1}{\lambda(y)}\, \sum_{n=0}^{\infty}
\mathbb{P}^{Y, \mf e}_{x}\big[\, Y(n)=y,\,n<\tau_{B} \,\big]\;.
\end{equation*}

We are left to show that for all $n\ge 0$
\begin{equation*}
M(x)\, \mathbb{P}^{Y, \mf e}_{x}\big[\, Y(n)=y,\,n<\tau_{B} \,\big]
\;=\;
M(y)\, \mathbb{P}^{Y, \mf e}_{y}\big[\, Y(n)=x,\,n<\tau_{B} \,\big]\;,
\end{equation*}
which follows from the reversibility of the chain $Y(n)$ with respect
to the stationary measure $M$.
\end{proof}

Note that we did not use in this proof the fact that the stationary
measure $m$ of the random walk $X$ is the uniform measure. This result
holds for general reversible dynamics, and a version for
non-reversible ones can be obtained along the same lines.

\smallskip We conclude this section with an identity used many times
in this article.  Let $A$, $B$ be two non-empty, disjoint subsets of
$S$.  Since $ L_X h_{A,B} \,=\,-\, L_X h_{B,A}$, by the last displayed
equation in the proof of \cite[Lemma B9]{l-review},
\begin{equation}
\label{ff20}
\Cap_{S}(A,B) \;=\;
-\, \sum_{x\in A} m(x)\, (L_X h_{A,B}) (x)
\;=\; \sum_{x\in A} m(x)\, (L_X h_{B,A}) (x)\;.
\end{equation}

\subsection{Coefficients of a quadratic function}
\label{sec92}

The super-harmonic function is, essentially, the square root of a
quadratic function. We introduce in this section the coefficients of
this quadratic function.

Fix $x_{0}\in S$, and recall that $S_{0}=S\setminus\{x_{0}\}$.  For
each non-empty subset $A$ of $S_{0}$, define the coefficients
$(b_{x,\,y}^{A})_{x,\,y\in S}$ by
\begin{equation}
\label{ff09}
b_{x,\,y}^{A} \;=\; \frac{1}{\kappa}\,
\frac{h_{x,\,A^{c}}(y)}{\Cap_{X}(x,\,A^{c})}\;,
\quad x\;, y\, \in\, A \;,
\end{equation}
and let $b_{x,\,y}^{A}=0$ otherwise.

\begin{lem}
\label{lem92}
For each non-empty subset $A$ of $S_{0}$ and for all $x,\,y\in S$,
$b_{x,\,y}^{A}=b_{y,\,x}^{A}$.
\end{lem}

\begin{proof}
For $x,\,y\in A$ this identity follows from Lemma \ref{lem91}.  If
either $x\notin A$ or $y\notin A$ (or both),
$b_{x,\,y}^{A}=b_{y,\,x}^{A}=0$ by definition.
\end{proof}

We present below some properties of this sequence.

\begin{lem}
\label{lem93}
For two non-empty subsets $A,\,B$ of $S_{0}$ satisfying $A\subset B$,
$b_{x,\,y}^{A}\le b_{x,\,y}^{B}$ for all $x,\,y\in S$.
\end{lem}

\begin{proof}
As the coefficients are non-negative, it is enough to check the
inequality for $x,y\in A$. In this case, since the measure $m$ is the
uniform measure, by \cite[Proposition 6.10]{BL1},
\begin{gather*}
b_{x,\,y}^{A} \;=\;
\frac{m(y)\,h_{x,\,A^{c}}(y)}{\Cap_{X}(x,A^{c})} \;=\;
\mathbb{E}_{x}\Big[\, \int_{0}^{\tau_{A^{c}}} \chi_{\{y\}} (X(t)) \
dt\, \Big]\;,\\
b_{x,\,y}^{B} \;=\;
\frac{m(y)\,h_{x,\,B^{c}}(y)}{\Cap_{X}(x,\,B^{c})}
\;=\; \mathbb{E}_{x}\Big[\, \int_{0}^{\tau_{B^{c}}} \chi_{\{y\}} (X(t)) \
dt\, \Big]\;.
\end{gather*}
The first expectation is bounded by the second since
$\tau_{A^{c}}\le\tau_{B^{c}}$.
\end{proof}

For a non-empty subset $A$ of $S_0$, let $z_{x}^{A}$, $x\in S$, be
given by
\begin{equation}
\label{ff08}
z_{x}^{A} \;=\; \frac{1}{2}\, \sum_{y\in S}
r(x,\,y)\, [\, b_{x,\,x}^{A} \,+\,
b_{y,\,y}^{A} \,-\, 2\, b_{x,\,y}^{A} \,]\;.
\end{equation}

\begin{lem}
\label{lem94}
For each non-empty $A\subset S_{0}$, we have that
\begin{equation*}
\frac{1}{2}\, \sum_{y\in S}r(x,\,y)\, (b_{y,\,y}^{A}-b_{x\,,x}^{A})
\;=\;
\begin{cases}
\text{z}_{x}^{A} \,-\, 1 & \text{for }x\in A\;,\\
\text{z}_{x}^{A} & \text{for }x\in A^{c}\;.
\end{cases}
\end{equation*}
\end{lem}

\begin{proof}
If $x\in A^{c}$ the result follows because $b_{x,\,y}^{A}=0$ for all
$y\in S$. On the other hand, if $x\in A$, by \eqref{ff20}
\begin{align*}
\sum_{y\in S}r(x\,,\,y) \, (b_{x,\,x}^{A}-b_{x,\,y}^{A}) \;& =\;
\frac{m(x)}{\Cap_{S}(x,\,A^{c})}\, \sum_{y\in S}r(x,\,y)\,
(h_{x,\,A^{c}}(x)-h_{x,\,A^{c}}(y))\\
\:& =\; -\, \frac{m(x)(L_{X}h_{x,\,A^{c}})(x)}{\Cap_{X}(x,\,A^{c})}
\;=\; 1\;.
\end{align*}
Hence,
\begin{align*}
z_{x}^{A} \,-\, 1 \;& =\; \frac{1}{2}\,
\sum_{y\in S}r(x,\,y)\, [\, b_{x,\,x}^{A} \,+\,
b_{y,\,y}^{A} \,-\, 2\, b_{x,\,y}^{A}\,]
\;-\; \sum_{y\in S} r(x,\,y)\, (b_{x,\,x}^{A} \,-\, b_{x,\,y}^{A} )\\
& =\; \frac{1}{2} \sum_{y\in S} r(x\,,\,y)\,
[b_{y,\,y}^{A} \,-\, b_{x,\,x}^{A}]\;,
\end{align*}
as claimed.
\end{proof}

\subsection{Linear and quadratic functions}
\label{sec93}

For a subset $\mathcal{C}$ of $\mathcal{H}_{N}$, let
\begin{gather*}
\text{int}\,\mathcal{C} \;=\;
\{\, \eta\in\mathcal{C}: \sigma^{x,\,y}\eta\in\mathcal{C}
\text{ for all }x,\,y\text{ with }r(x,\,y)>0\, \}\;, \\
\partial\mathcal{C} \;=\; \mathcal{C}\setminus\text{int }\mathcal{C}\;,\\
\overline{\mathcal{C}} \;=\; \{\eta\in\mathcal{H}_{N}:
\eta\in\mathcal{C}\text{ or }\sigma^{x,\,y}\eta\in\mathcal{C}
\text{ for some }x,\,y\text{ with }r(x,\,y)>0\}\;.
\end{gather*}

To prove Theorem \ref{p82}, it suffices to construct a function
$G^{x_0}_N$ on
$\overline{\mathcal{W}_{N}^{x_{0}} \setminus
  \mathcal{D}_{N}^{x_{0}}}$, satisfying the conditions of the
proposition in the set
$\mathcal{W}_{N}^{x_{0}}\setminus\mathcal{D}_{N}^{x_{0}}$, and to
extend it arbitrarily to $\mathcal{H}_{N}$.

Let
\begin{equation*}
\mathcal{U}_{N}^{x_{0}} \;=\;
\overline{\mathcal{W}_{N}^{x_{0}}\setminus\mathcal{D}_{N}^{x_{0}}}
\quad\text{so that}\quad
\text{int }\mathcal{U}_{N}^{x_{0}}=\mathcal{W}_{N}^{x_{0}}
\setminus\mathcal{D}_{N}^{x_{0}}\;.
\end{equation*}

Define the quadratic function
$Q^{A}:\mathcal{U}_{N}^{x_{0}}\rightarrow\mathbb{R}$,
$A\subset S_{0}$, and the linear function
$U^{A}:\mathcal{U}_{N}^{x_{0}}\rightarrow\mathbb{R}$ as
\begin{gather*}
Q^{A}(\eta) \;=\; \frac{1}{2}\, \sum_{x,\,y\in S}
b_{x,\,y}^{A} \, \eta_{x}\, \eta_{y}
\;=\; \frac{1}{2}\, \sum_{x\in A} b_{x,\,x}^{A}\,
\eta_{x}^{2} \,+\, \sum_{\{x,\,y\}\subset A} b_{x,\,y}^{A}
\, \eta_{x}\, \eta_{y}\;,\\
U^{A}(\eta) \;=\; \frac{1}{2} \,
\sum_{x\in S}b_{x,\,x}^{A}\, \eta_{x} \;=\;
\frac{1}{2}\, \sum_{x\in A} b_{x,\,x}^{A}\, \eta_{x}\;.
\end{gather*}
In the last sum of the first line, each pair $\{x,y\}$ appears only
once.  Let
\begin{gather*}
P^{A}(\eta) \;=\; Q^{A}(\eta) \;-\; U^{A}(\eta)
\;=\; \frac{1}{2}\, \sum_{x\in A} b_{x,\,x}^{A} \,
\eta_{x}\, (\eta_{x}-1) \;+\;\sum_{\{x,\,y\}\subset A}
b_{x,\,y}^{A} \, \eta_{x}\, \eta_{y}\;.
\end{gather*}
Note that $P^\varnothing (\eta) =0$ for all $\eta$.

Fix $A\subset S_0$, $x\in S$ and
$\eta\in \text{int }\mc U_{N}^{x_{0}}$ such that $\eta_x \ge 1$. An
elementary computation yields that
\begin{equation}
\label{ff16}
U^{A}(\sigma^{x,\,y}\eta) \;-\; U^{A}(\eta)
\;=\; \frac{1}{2}\, \Big\{\,
b^A_{y,y} \, \bs 1\{y\in A\} \,-\,
b^A_{x,x} \, \bs 1\{x\in A\}\,\Big\}\;.
\end{equation}

\begin{lem}
\label{lem95} For $x\in A$, $y\in S\setminus\{x\}$, and
$\eta\in\text{int }\,\mathcal{U}_{N}^{x_{0}}$,
\begin{equation*}
Q^{A}(\sigma^{x,\,y}\eta) \;-\; Q^{A}(\eta)
\;=\; \sum_{z\in A}\eta_{z}\, [\, b_{z,\,y}^{A}
\,-\, b_{z,\,x}^{A}\,] \;+\;
\frac{1}{2}\, [\, b_{x,\,x}^{A} \;+\; b_{y,\,y}^{A}
\;-\; 2\, b_{x,\,y}^{A}\,]\;.
\end{equation*}
\end{lem}

\begin{proof}
First, fix $y\in A$, $y\neq x$. By definition of $Q_A$,
$Q^{A}(\sigma^{x,\,y}\eta) \,-\, Q^{A}(\eta)$ is equal to
\begin{align*}
& \frac{1}{2}\, b_{x,\,x}^{A} \,
[\, (\eta_{x}-1)^{2} \,-\, \eta_{x}^{2}\,]
\;+\; \frac{1}{2} \, b_{y,\,y}^{A} \, [\, (\eta_{y}+1)^{2} \,-\, \eta_{y}^{2}\,]\\
&\quad  +\; b_{x,\,y}^{A}\, [\, (\eta_{x}-1)(\eta_{y}+1)
\,-\, \eta_{x}\eta_{y}\,]
\;+\; \sum_{z\in A\setminus\{x,\,y\}} b_{x,\,z}^{A}
[\, (\eta_{x}-1)\, \eta_{z} \,-\, \eta_{x}\, \eta_{z}\,] \\
&\quad  +\; \sum_{z\in A\setminus\{x,y\}} b_{z,y}^{A}
[\, \eta_{z}(\eta_{y}+1) \,-\, \eta_{z}\eta_{y} \,] \;.
\end{align*}
We may rewrite this sum as
\begin{align*}
& \frac{1}{2}\, b_{x,\,x}^{A}\, [\, 1\,-\, 2\, \eta_{x}\, ]
\;+\; \frac{1}{2}\, b_{y,\,y}^{A}\, [\, 2\, \eta_{y} \,+\, 1]
\; +\; b_{x,\,y}^{A}\, [\, \eta_{x} \,-\, \eta_{y} \,-\, 1]\\
& \quad +\; \sum_{z\in A\setminus\{x,\,y\}} \eta_{z}\, [\, b_{z,\,y}^{A}
\,-\, b_{x,\,z}^{A}\,]\;.
\end{align*}
Since $b_{z,\,w}^{A}$ is symmetric  by Lemma \ref{lem92}, if $y\in A$,
\begin{align*}
Q^{A}(\sigma^{x,\,y}\eta) \,-\, Q^{A}(\eta)
\;=\;
\sum_{z\in A}\eta_{z}\,[ \, b_{z,\,y}^{A} \,-\, b_{z,\,x}^{A}\,]
\;+\; \frac{1}{2}\, [\, b_{x,\,x}^{A} \,+\, b_{y,\,y}^{A}
\,-\, 2\, b_{x,\,y}^{A}\, ]\;,
\end{align*}
as claimed.

Assume now that $y$ belongs to $A^{c}$. In this case, by definition of
$Q_A$, $Q^{A}(\sigma^{x,\,y}\eta) \,-\, Q^{A}(\eta)$ is equal to
\begin{equation*}
\frac{1}{2}\, b_{x,\,x}^{A}\, [\, (\eta_{x}-1)^{2}-\eta_{x}^{2}\,]
\;+\sum_{z\in A\setminus\{x\}} b_{x,\,z}^{A}\,
[\, (\eta_{x}-1)\eta_{z} \,-\, \eta_{x}\eta_{z}\,]
\;=\; -\, \sum_{z\in A} b_{z,\,x}^{A}\, \eta_{z}
\;+\; \frac{1}{2}\, b_{x,\,x}^{A}\;.
\end{equation*}
To complete the proof, it remains to recall that $b_{z,\,y}=0$ for all
$z\in S$.
\end{proof}

Fix $A\subset S_0$, $x\not \in A$, $y\not = x$ and
$\eta\in\text{int }\mathcal{U}_{N}^{x_{0}}$ such that $\eta_{x}\ge 1$.
A similar computation yields that
\begin{equation*}
Q^{A}(\sigma^{x,\,y}\eta) \;-\; Q^{A}(\eta)
\;=\; \Big\{\, \frac{1}{2}\, (2\eta_y+1) \,  b_{y,y}^{A}
\,+\, \sum_{z\in A \,,\, z\not = y} b_{y,z}^{A} \, \eta_z
\,\Big\}\,
\bs 1\{\, y\,\in\, A\,\}\;.
\end{equation*}
It follows from \eqref{ff16}, Lemma \ref{lem95} and the previous
estimate that there exists a constant $C_0$ such that
\begin{equation}
\label{ff18}
\big| \, P^{A}(\sigma^{x,\,y}\eta) \;-\; P^{A}(\eta)\, \big|
\;\le\; C_0\, \bigg\{ \, 1\;+\; \sum_{z\in A} \eta_z \, \bigg\}
\end{equation}
for all subsets $A$ of $S_0$, $x$, $y\in S$, $y\not = x$ and
$\eta\in\text{int }\mathcal{U}_{N}^{x_{0}}$ such that $\eta_{x}\ge 1$.

Let $u_{x,\,y}^{A}$, $x\in A$, $y \in A^{c}$, be given by
\begin{equation}
\label{uij}
u_{x,\,y}^{A} \;=\;
\frac{m(x)\,(L_{X}h_{x,\,A^{c}})(y)}{\Cap_{X}(x,\,A^{c})} \;.
\end{equation}
Since $m(x)=m(y)$ and $L_{X} h_{x,\,A^{c}} = -\, L_{X} h_{A^{c},x}$,
by \eqref{ff20},
\begin{equation}
\label{sumu}
\sum_{y\in A^{c}}u_{x,\,y}^{A} \;=\; 1\;, \quad
\text{for all }x\in A\;.
\end{equation}
Observe that this identity holds only because $m$ is the uniform
measure, as we replaced $m(x)$ by $m(y)$.

\begin{lem}
\label{lem96}
Fix $A\subset S_0$ and $\eta\in\text{int }\mathcal{U}_{N}^{x_{0}}$. If
$x\in A$ and $\eta_{x}\ge 1$, then
\begin{equation*}
\sum_{y\in S}r(x,\,y)\, [\, P^{A}(\sigma^{x,\,y}\eta)
\,-\, P^{A}(\eta)\,] \;=\; -\, \eta_{x} \;+\; 1\;.
\end{equation*}
On the other hand, if $x\in A^{c}$ and $\eta_{x}\ge1$, then
\begin{equation*}
\sum_{y\in S}r(x,\,y)\,
[\, P^{A}(\sigma^{x,\,y}\eta) \,-\, P^{A}(\eta)\,]
\;=\; \sum_{z\in A} u_{z,\,x}^{A}\, \eta_{z}\;.
\end{equation*}
\item In particular, for $A=S_{0}$,
\begin{equation*}
\sum_{y\in S}r(x_{0},y)\, [\, P^{S_{0}}(\sigma^{x_0,\,y}\eta)
\,-\, P^{S_{0}}(\eta)\,] \;=\; \sum_{z\in S_0}\eta_{z}\;.
\end{equation*}
\end{lem}

\begin{proof}
We consider $Q^{A}$ and $U^{A}$ separately. By Lemma \ref{lem95},
\begin{align*}
& \sum_{y\in S} r(x,\,y) \, [\,
Q^{A}(\sigma^{x,\,y}\eta)-Q^{A}(\eta)\, ]  \\
& \quad =\; \sum_{y\in S} r(x,\,y) \, \Big(\,
\sum_{z\in A} \eta_{z} [\, b_{z,\,y}^{A} - b_{z,\,x}^{A}\,]
\,+\, \frac{1}{2}\, [\, b_{x,\,x}^{A} +
b_{y,\,y}^{A}-2\, b_{x,\,y}^{A}\,] \,\Big) \;.
\end{align*}
By definition of $b_{x,\,y}^{A}$ and $z_{x}^{A}$ given in \eqref{ff09} and
\eqref{ff08}, respectively, and by changing the order of summation yield
that the previous expression is equal to
\begin{equation*}
\sum_{z\in A} \eta_{z}\, \frac{m(z)}{\Cap_{X}(z,\,A^{c})}
\, \sum_{y\in S} r(x,\,y) \, [\, h_{z,A^{c}}(y) \,-\,
h_{z,A^{c}}(x)\, ] \;+\; z_{x}^{A}\;.
\end{equation*}
Thus, by definition of $u^A_{x,y}$, introduced in \eqref{uij}, we have that
\begin{equation}
\label{ff10}
\sum_{y\in S} r(x,\,y) \, [\,
Q^{A}(\sigma^{x,\,y}\eta)-Q^{A}(\eta)\, ] \;=\;
\sum_{z\in A} \eta_{z}\, u^A_{z,x} \;+\; z_{x}^{A}\;.
\end{equation}
On the other hand, by definition of $U^{A}$, and since
$\eta_{x}\ge 1$,
\begin{equation*}
\sum_{y\in S} r(x,\,y)\, [\, U^{A}(\sigma^{x,\,y}\eta)
\,-\, U^{A}(\eta)\,]
\;=\; \frac{1}{2}\, \sum_{y\in S} r(x,\,y)\, [\, b_{y,\,y}^{A}
\,-\, b_{x,\,x}^{A}\,]\;.
\end{equation*}

Assume that $x\in A$ and $\eta_{x}\ge 1$. In this case, by definition
of $u^A_{x,\,y}$ and \eqref{ff20},
\begin{equation*}
u^A_{z,\,x} \,=\, 0 \;\;\text{for}\;\; z\,\in\,  A\setminus\{x\}
\;\;\text{and}\;\; u^A_{x,\,x} \,=\, -\, 1\;.
\end{equation*}
Therefore, in this case,
\begin{equation*}
\sum_{y\in S} r(x,\,y) \, [\,
Q^{A}(\sigma^{x,\,y}\eta)-Q^{A}(\eta)\, ]
\;=\; -\, \eta_{x} \;+\; {z}_{x}^{A} \;.
\end{equation*}
Moreover,  by Lemma \ref{lem94},
\begin{equation*}
\sum_{y\in S} r(x,\,y)\, [\, U^{A}(\sigma^{x,\,y}\eta)
\,-\, U^{A}(\eta)\,]
\;=\; {z}_{x}^{A} \;-\; 1\;.
\end{equation*}
This completes the first part of the proof, in view of the definition
of $P^A$.

Assume that $x\in A^{c}$ and $\eta_{x}\ge1$. In this case, by Lemma
\ref{lem94},
\begin{equation*}
\sum_{y\in S} r(x,\,y)\, [\, U^{A}(\sigma^{x,\,y}\eta)
\,-\, U^{A}(\eta)\,]
\;=\; {z}_{x}^{A} \;.
\end{equation*}
This identity together with \eqref{ff10} completes the proof of the
second assertion of the lemma.

For the last assertion of the lemma, we have to check that
$u_{z,\,x_{0}}^{S_{0}}=1$ for all $z\in S_{0}$. By \eqref{uij},
and since $h_{z,\,x_{0}} = 1\, - \, h_{x_0,z}$ and $m(z) = m(x_0)$,
\begin{equation*}
u_{z,\,x_{0}}^{S_{0}} \;=\;
\frac{m(z)\,(L_{X}h_{z,\,x_{0}})(x_{0})}{\Cap_{X}(z,\,x_{0})}
\;=\; -\,
\frac{m(x_0)\,(L_{X}h_{x_0,z})(x_{0})}{\Cap_{X}(z,\,x_{0})}\;\cdot
\end{equation*}
By \eqref{ff20}, this expression is equal to $1$, which completes the
proof of the lemma.
\end{proof}

The next result is a consequence of Lemma \ref{lem96}.  Fix a function
$J:\mathcal{U}_{N}^{x_{0}}\mapsto\mathbb{R}$. In the remaining part of
the current section, we write $\color{bblue} J(\eta)=o_{N}(1)$ if
\begin{equation*}
\limsup_{N\mapsto\infty} \,\sup_{\eta\in\mathcal{U}_{N}^{x_{0}}}
\, \big|\, J(\eta) \,\big| \;=\; 0\;.
\end{equation*}

\begin{prop}
\label{p97}
Fix a non-empty subset $A$ of $S_{0}$ and
$\eta \in \text{\rm int }\mathcal{U}_{N}^{x_{0}}$. Then,
\begin{equation*}
(\mathscr{A}_{N}P^{A})(\eta)\; =\;
\sum_{x\in A} {g}( \eta_{x})\, [\,1 \,-\, \eta_{x}\,]
\;+\; \sum_{x\in A^{c}} {g}(\eta_{x})\,
\sum_{z\in A}u_{z,\,x}^{A} \, \eta_{z}\;.
\end{equation*}
If $\eta_{x}\ge 2$ for all $x\in A^{c}$, then
\begin{equation*}
(\mathscr{A}_{N}P^{A}) (\eta)\;\ge\;
\sum_{x\in A} \bs 1 \{\eta_{x}=1\}\;.
\end{equation*}
Finally, if  $A=S_{0}$,
\begin{equation*}
(\mathscr{A}_{N} P^{S_{0}})(\eta)
\;=\; \sum_{x\in S_{0}}\bs 1\{\eta_{x}=1\} \;+\; o_{N}(1)\;.
\end{equation*}
\end{prop}

\begin{proof}
The first assertion is a consequence of Lemma \ref{lem96}. For the
second one, since $\eta_{x}\ge 2$ for all $x\in A^{c}$, we obtain from
the first part that
\begin{equation}
\label{ff15}
(\mathscr{A}_{N}P^{A}) (\eta) \;\ge\;
\sum_{x\in A} {g}(\eta_{x})\, [\, 1\,-\, \eta_{x}\,]
\;+\; \sum_{x\in A^{c}} \sum_{z\in A} u_{z,\,x}^{A}\, \eta_{z}
\end{equation}
because $g(\eta_x)\ge 1$ for $x\in A^c$.  By \eqref{sumu}, the second
term on the right-hand side is equal to $\sum_{z\in A} \eta_{z}$, so
that
\begin{equation*}
(\mathscr{A}_{N}P^{A})(\eta) \;\ge\;
\sum_{x\in A} \big\{ \, \eta_{x} \,-\,{g}(\eta_{x})\, [\,
\eta_{x}-1\,] \big\} \;=\; \sum_{x\in A}\bs 1 \{\eta_{x}=1\}\;,
\end{equation*}
because $n \,-\, {g}(n)(n-1) \,=\, \bs 1 \{n=1\}$. This proves the
second assertion of the proposition.

We turn to the last claim. By the first assertion of this proposition
and the last one of Lemma \ref{lem96},
\begin{equation*}
(\mathscr{A}_{N}P^{S_{0}})(\eta)  \;=\;
\sum_{x\in S_{0}} {g}(\eta_{x})\, [\, 1\,-\, \eta_{x}\,]
\;+\; {g}(\eta_{x_0})\, \sum_{z\in S_{0}} \eta_{z} \;.
\end{equation*}
As $\eta$ belongs to $\text{int }\mathcal{U}_{N}^{x_{0}}$,
$\sum_{z\in S_{x_{0}}}\eta_{z} /[\eta_{x_0} -1]= o_N(1)$. Hence,
writing $g (\eta_{x_0})$ as $1 + [\eta_{x_0}-1]^{-1}$, the previous
identity becomes
\begin{equation*}
(\mathscr{A}_{N}P^{S_{0}})(\eta)  \;=\;
\sum_{x\in S_{0}} \big\{ \, \eta_{x} \,-\,{g}(\eta_{x})\, [\,
\eta_{x}-1\,] \big\}  \;+\; o_N(1) \;.
\end{equation*}
To complete the argument, it remain to recall that $n \,-\,
{g}(n)(n-1) \,=\, \bs 1 \{n=1\}$.
\end{proof}

Let us write $\color{bblue} P = P^{S_{0}}$.

\begin{lem}
\label{lem99}
  There exist constants $c_1,\,c_2>0$ such that
\begin{equation*}
c_{1}\, \Big( \sum_{x\in S_{0}}\eta_{x}\Big)^{2} \;\le\;
P(\eta) \;\le\;  c_{2}\, \Big(\sum_{x\in S_{0}}\eta_{x}\Big)^{2}
\end{equation*}
for all $\eta\in\mathcal{U}_{N}^{x_{0}}$.
\end{lem}

\begin{proof}
The upper bound follows from the definition of $P$. To prove the lower
bound, note that there exists constants $0<\lambda<\varLambda<\infty$
such that $\lambda<b_{x,\,x}^{S_{0}}<\varLambda$ for all $x\in S_{0}$.
Thus, by definition of $P$, since $b^{S_0}_{x,y} \ge 0$ for all $x$,
$y\in S$, and by the Cauchy-Schwarz inequality,
\begin{equation*}
P(\eta) \; \ge\; \frac{\lambda}{2} \sum_{x\in S_{0}}
\eta_{x}^{2} \;-\; \frac{\varLambda}{2} \sum_{x\in S_{0}}
\eta_{x} \;\ge\; \frac{\lambda}{2(\kappa-1)}\,
\Big(\sum_{x\in S_{0}}\eta_{x}\Big)^{2} \;-\;
\frac{\varLambda}{2} \sum_{x\in S_{0}}\eta_{x}\;.
\end{equation*}
To complete the proof, it remains to recall that
$\sum_{x\in S_{0}} \eta_{x}\,\ge\,  N^{\gamma}\gg1$ for
$\eta\in\mathcal{U}_{N}^{x_{0}}$.
\end{proof}

\begin{lem}
\label{lem910}
There exists a positive constant $c_0>0$ such that
\begin{equation*}
\sum_{x\in S} {g}(\eta_{x})\, \sum_{y\in S}
r(x,\,y)\, [\, P(\sigma^{x,\,y}\eta) \,-\, P(\eta)\,]^{2}
\;\ge\; c_0\, P(\eta)
\end{equation*}
for all $\eta\in\text{\rm int }\mathcal{U}_{N}^{x_{0}}$.
\end{lem}

\begin{proof}
Since ${g}(\eta_{x_0})\ge 1$, it suffices to show that
\begin{equation*}
\sum_{y\in S} r(x_{0},\,y)\, [\, P(\sigma^{x_0,\,y}\eta) \,-\, P(\eta)\,
]^{2} \;\ge\; c_0 \, P(\eta) \;.
\end{equation*}
By the Cauchy-Schwarz inequality,
\begin{align*}
& \sum_{y\in S}r(x_{0},\,y)\, \sum_{y\in S}r(x_{0},\,y) \,
[\, P(\sigma^{x_{0},\,y}\eta) \,-\, P(\eta)\,]^{2} \\
& \quad \ge\; \Big( \sum_{y\in S} r(x_{0},\,y)\,
[\, P(\sigma^{x_{0},\,y}\eta) \,-\, P(\eta)]\,\Big)^{2}
\;=\; \Big(\sum_{z\in S_{0}} \eta_{z}\Big)^{2}\;,
\end{align*}
where the last identity follows from the third assertion of Lemma
\ref{lem96}. To complete the proof, recall the upper bound of Lemma
\ref{lem99}.
\end{proof}

\subsection{A super-harmonic function}
\label{sec94}

Some notations are required. We first claim that for each non-empty
subset $B$ of $S_{0}$, there exist positive constants $\alpha_{B}$,
$\beta_{B}>0$ such that
\begin{equation}
\label{ebbc1}
\frac{1}{2}\, \sum_{x\in B_{0}}b_{x,\,x}^{B}\,
\eta_{x}\, (\eta_{x}-1)
\;+\; \sum_{\{x,\,y\}\in B_{0}} b_{x,\,y}^{B}\, \eta_{x}\, \eta_{y}
\;<\, \alpha_{B}\, \sum_{x\in B_{0}} \eta_{x}\, (\eta_{x}-1)
\;+\; \beta_{B}
\end{equation}
for all $\varnothing\neq B_{0}\subset B$ and for all
$\eta\in\mathcal{U}_{N}^{x_{0}}$.

To prove this claim, let {\color{bblue}
$\mf a = \max\,\{ b^C_{z,w} \}$}, where the maximum is performed over
all nonempty subsets $C$ of $S_0$, and all $z$, $w\in C$. Clearly,
there exists a finite constant $C_0$, depending only on $\kappa$, such
that the left-hand side of \eqref{ebbc1} is bounded by
\begin{equation*}
C_0 \, \mf a  \, \sum_{x\in B_{0}} \eta^2_{x} \;\le\;
2 C_0 \, \mf a\, \sum_{x\in B_{0}} \eta_{x} \, (\eta_{x}-1)
\;+\;  C_0\, \kappa\,\mf a\;,
\end{equation*}
because $t^{2}\le2t(t-1)+1$. This proves the claim.
Clearly, we may assume that
\begin{equation}
\label{conalb}
\alpha_{B} \;>\; b_{x,\,x}^{B} \quad \text{for all }x\in B\;.
\end{equation}

We assign a positive constant $c_{A}>0$ to each proper, non-empty
subset $A$ of $S_{0}$, i.e., $\varnothing\subsetneq A\subsetneq S_0$, as follows.

If $A$ is a singleton, $|A|=1$, set $c_{A}>0$ arbitrarily. Fix
$2\le k \le |S_0|-1$, and suppose that $c_{A}$ has been assigned to
all sets $A\neq\varnothing$ such that $|A|<k$. Fix a subset $B$ of
$S_0$ such that $|B|=k$, and let
\begin{equation}
\label{ebbc2}
c_{B}^{0} \;=\;
\max_{A}\,
\max_{x\in A}\left\{ \frac{2\, \alpha_{B}\, c_{A}}{b_{x,\,x}^{A}}
\;+\; \beta_{B}\right\} \;,
\end{equation}
where the first maximum is performed over all proper, non-empty
subsets $A$ of $B$. Then, select a constant $c_{B}$ larger than $c_{B}^{0}$ and such that
\begin{equation}
\label{ff17}
c_{A}\neq c_{B}\;\text{for all }A,\,B\subsetneq
S_{0}\text{ with }A\neq B\;.
\end{equation}

Fix a positive integer $\ell\ge2$.  For each proper, non-empty subset
$A$ of $S_{0}$, let
$P_{\ell}^{A}:\mathcal{U}_{N}^{x_{0}}\rightarrow\mathbb{R}$ be given by
\begin{equation}
\label{ff11}
P_{\ell}^{A}(\eta) \;=\; P^{A}(\eta) \;-\; c_{A}\, \ell^{2}\;.
\end{equation}

Clearly,
\begin{equation*}
(\mathscr{A}_{N}P_{\ell}^{A})(\eta)
\;=\; (\mathscr{A}_{N}P^{A})(\eta)
\quad \text{for all }\eta\in\mathcal{U}_{N}^{x_{0}}\;.
\end{equation*}

Let $\color{bblue} P_{\ell}^{\varnothing}(\eta)=0$ for all
$\eta\in\mathcal{U}_{N}^{x_{0}}$, and define the correction function
$W_{\ell}:\mathcal{U}_{N}^{x_{0}}\mapsto\mathbb{R}$ by
\begin{equation*}
W_{\ell}(\eta)=\min\left\{ P_{\ell}^{A}(\eta):
A\subset S_{0} \,,\, A \not = S_0 \, \right\} \;.
\end{equation*}

\begin{lem}
\label{lem913}
There exists a constant $C_0<\infty$ such that, for all
$\eta\in\mathcal{U}_{N}^{x_{0}}$,
\begin{equation*}
-\, C_0 \, \ell^{2} \;\le\; W_{\ell}(\eta) \;\le\; 0\;.
\end{equation*}
\end{lem}

\begin{proof}
Since $W_{\ell}(\eta) \le P_{\ell}^{\varnothing}(\eta)\le 0$, the
upper bound is clear. We turn to the lower bound.  Since $P^A$,
$A\subsetneq S_{0}$, is a non-negative function,
$P_{\ell}^{A}(\eta)\ge-c_{A}\ell^{2}$. It remains to set $C_0$ as
$\max_{A\subsetneq S_{0}}c_{A}$.
\end{proof}

By Lemmata \ref{lem99} and \ref{lem913}, we get $P(\eta) - W_{\ell}(\eta)>0$
for all $\eta\in \mathcal{U}_{N}^{x_{0}}$.  Let
$F_{m}:\mathcal{U}_{N}^{x_{0}}\rightarrow\mathbb{R}$, $m\ge 2$, be
defined by
\begin{equation*}
F_{m}(\eta) \;=\;
\sum_{\ell=2}^{m} \frac{1}{\ell} \,[\, P(\eta) \,-\,
W_{\ell}(\eta)\,]^{1/2}\; .
\end{equation*}

\begin{thm}
\label{t915}
There exists $m\in \bb N$, $c_0>0$ and $N_0\ge 1$, such that for all
$N\ge N_0$, the function $F_m$ satisfies
$$
(\mathscr{A}_N F_m) (\eta)  \,\le\, -\, \frac{c_0}{N-\eta_{x_0}}
\;\;\;\;\text{for all } \eta \in\textrm{\rm int }\,
\mathcal{U}_{N}^{x_{0}}\;.
$$
\end{thm}

The proof of this theorem is given in Section \ref{sec95}.  This
result, as well as the majority of the next ones, are asymptotic in
$N$. This means that they may fail for small $N$, but that there
exists a constant $N_{0}$, which may depend only on $\kappa$ and
$\ell$, such that the assertion holds for $N>N_{0}$.

\begin{proof}[Proof of Theorem \ref{p82}]
Let
\begin{equation*}
G_{N}^{x_{0}}(\eta)=\begin{cases}
F_{m}(\eta) & \eta\in\mathcal{U}_{N}^{x_{0}}\;,\\
0 & \text{otherwise\,.}
\end{cases}
\end{equation*}
The requirement \eqref{e821} follows from Theorem \ref{t915} and the
fact that $\mathscr{L}_N = \theta_N \mathscr{A}_N$, while
\eqref{e822} follows from Lemmata \ref{lem99} and \ref{lem913}.
\end{proof}

\subsection{The corrector  $W_{\ell}$}

Let $\mathcal{D}_{\ell}(A)$, $A\subsetneq S_{0}$, be the set given by
\begin{equation}
\label{ff14}
\mathcal{D}_{\ell}(A) \;=\;
\{\eta\in\mathcal{U}_{N}^{x_{0}}: P_{\ell}^{A}(\eta)
\,=\, W_{\ell}(\eta)\}\;.
\end{equation}
Note that some configurations may belong to several
$\mathcal{D}_{\ell}(A)$'s.

Let $A_{0}$ be a proper, non-empty subset of $S_{0}$ and let $\eta$ be
a configuration in $\mathcal{U}_{N}^{x_{0}}$ such that $\eta_{x}=0$
for all $x\in A_{0}$. The next lemma states that
$W_\ell(\eta) = P_{\ell}^{A}(\eta)$ for some set $A$ which contains
$A_{0}$.

\begin{lem}
\label{lem02}
Fix a proper, non-empty subset $A_{0}$ of $S_{0}$ and $\eta$
in $\mathcal{U}_{N}^{x_{0}}$ such that $\eta_{x}=0$ for all
$x\in A_{0}$. Suppose that
\begin{equation*}
P_{\ell}^{A}(\eta) \;=\; W_\ell(\eta)
\end{equation*}
for some $A\subsetneq S_0$. Then, $A\supset A_{0}$ provided $N$ is
large enough.
\end{lem}

\begin{proof}
Fix $A\subsetneq S_0$, and assume that
\begin{equation*}
P_{\ell}^{A}(\eta) \;=\; W_\ell(\eta)
\;=\;  \min\{P_{\ell}^{B}(\eta):B\subsetneq S_{0}\}\;.
\end{equation*}
In particular, since $\eta_{x}=0$ for all $x\in A_{0}$,
\begin{equation}
\label{condp1}
P_{\ell}^{A}(\eta) \;\le\; P_{\ell}^{A_{0}}(\eta)
\;=\; -\, c_{A_{0}}\, \ell^{2} \;<\; 0\;.
\end{equation}
We consider separately three cases.

Suppose that $A\subsetneq A_{0}$. By definition,
$c_{A_{0}}>c^0_{A_0}$. By \eqref{ebbc2} and Lemma \ref{lem93}, this
constant is larger than
$2\alpha_{A_0} c_A/b^A_{x,x} \ge 2\alpha_{A_0} c_A/b^{A_0}_{x,x}$. By
\eqref{conalb}, we get $2\alpha_{A_0} c_A/b^{A_0}_{x,x} \ge 2c_A \ge
c_A$. This proves that $c_{A_{0}}>c_{A}$.  Thus, as $\eta_x=0$ for all
$x\in A_0$,
\begin{equation*}
P_{\ell}^{A}(\eta) \;=\; -\, c_{A}\, \ell^{2} \;>\;
-\, c_{A_{0}}\, \ell^{2} \;=\; P_{\ell}^{A_{0}}(\eta)\;,
\end{equation*}
in contradiction with \eqref{condp1}.

\smallskip Assume that $A_{0}\not\subset A$, $A\not\subset A_{0}$ and
$A_{0}\cup A\neq S_{0}$. We claim that
\begin{equation}
\label{condp2}
P^{A_{0}\cup A}(\eta) \;<\; P_{\ell}^{A}(\eta) \;.
\end{equation}
Since $\eta_{x}=0$ for all $x\in A_{0}$ and since
$A_{0}\cup A\neq S_{0}$,
\begin{gather*}
P_{\ell}^{A}(\eta) \;=\;
\frac{1}{2}\, \sum_{x\in A\setminus A_{0}} b_{x,\,x}^{A}\,
\eta_{x}\, (\eta_{x}-1) \;+\;
\sum_{x,\,y\in A\setminus A_{0}} b_{x,\,y}^{A}\, \eta_{x}\,\eta_{y}
\;-\; c_{A}\, \ell^{2} \;,\\
P_{\ell}^{A_{0}\cup A}(\eta) \;=\; \frac{1}{2}\,
\sum_{x\in A\setminus A_{0}} b_{x,\,x}^{A_{0}\cup A}\, \eta_{x}\,
(\eta_{x}-1) \;+\; \sum_{x,\,y\in A\setminus A_{0}}
b_{x,\,y}^{A_{0}\cup A}\, \eta_{x}\, \eta_{y} \;-\;
c_{A_{0}\cup A}\ell^{2} \;.
\end{gather*}
These sums are carried over a set which is not empty because we
assumed that $A\not\subset A_{0}$.  Let
\begin{equation*}
M \;=\; \max_{x\in A} \frac{\alpha_{A_{0}\cup A}}
{b_{x,\,x}^{A}} \;. \quad \text{By Lemma \ref{lem93} and
  \eqref{conalb}}\;,
\;\; M \,>\, 1\;.
\end{equation*}
By \eqref{ebbc1} and the explicit formula for $P_{\ell}^{A_{0}\cup
  A}(\eta)$,
\begin{equation*}
P_{\ell}^{A_{0}\cup A}(\eta) \;<\;
\alpha_{A_{0}\cup A}\, \sum_{x\in A\setminus A_{0}}
\eta_{x}\, (\eta_{x}-1) \;+\; \beta_{A_{0}\cup A} \;-\;
c_{A_{0}\cup  A}\, \ell^2\;.
\end{equation*}
By definition of $M$, this expression is bounded by
\begin{equation*}
M\, \sum_{x\in A\setminus A_{0}} b_{x,\,x}^{A}\,
\eta_{x}\, (\eta_{x}-1) \;+\; \beta_{A_{0}\cup A}
\;-\; c_{A_{0}\cup  A}\, \ell^2\;.
\end{equation*}
By definition of $P_{\ell}^{A}$, this sum is less than or equal to
\begin{equation*}
2\, M\, P_{\ell}^{A}(\eta) \;+\;
2\, M\, c_{A} \, \ell^2 \;+\;
\beta_{A_{0}\cup A} \;-\; c_{A_{0}\cup A} \, \ell^2 \;.
\end{equation*}
By definition, $c_{A_{0}\cup A}>c_{A_{0}\cup A}^{0}$. Since
$A_0\not\subset A$, $A \subsetneq A_0 \cup A$. Thus, by \eqref{ebbc2}
and by definition of $M$,
$c_{A_{0}\cup A}^{0} \ge 2 M\, c_{A}+\beta_{A_{0}\cup A}$. Hence, by
the previous estimates, and since $\ell\ge 1$,
\begin{equation*}
P_{\ell}^{A_{0}\cup A}(\eta) \;<\, 2\,M\, P_{\ell}^{A}(\eta)
\;\le\; P_{\ell}^{A}(\eta)
\end{equation*}
because $M>1$ and $P_{\ell}^{A}(\eta)<0$ by \eqref{condp1}. This proves \eqref{condp2}
and contradicts the fact that $P_{\ell}^{A}(\eta) = W_{\ell}(\eta)$.

\smallskip
Assume, finally, that $A_{0}\cup A=S_{0}$. Since both are proper
subsets of $S_0$, $A_{0}\not\subset A$ and $A\not\subset A_{0}$.

The set $S_{0}$ can be decomposed into
$S_{0}=A_{0}\cup(A\setminus A_{0})$. Since $\eta\in \mc U^{x_0}_N$,
and $\eta_{x}=0$ for all $x\in A_0$,
\begin{equation*}
\sum_{x\in A\setminus A_{0}}\eta_{x} \;=\;
\sum_{x\in S_{0}}\eta_{x} \;\ge\; N^{\gamma}\;.
\end{equation*}
Since $b^A_{x,x}>0$ for all $x\in A$, a similar computation to the one
presented in the proof of Lemma \ref{lem99} yields that
\begin{equation*}
\frac{1}{2}\, \sum_{x\in A\setminus A_{0}}b_{x,x}^{A}\,
\eta_{x}\, (\eta_{x}-1) \;\ge\; c_0 \, N^{2\gamma}
\end{equation*}
for some positive constant $c_0$. Thus,
\begin{equation*}
P_{\ell}^{A}(x) \;>\; c_0\, N^{2\gamma} \;-\; c_{A}\,\ell^2
\;\ge \; 0
\end{equation*}
for large enough $N$, which contradicts \eqref{condp1}.

In conclusion, none of the previous three cases can be in force, so
that $A\supset A_0$, as claimed.
\end{proof}

\begin{cor}
\label{p07}
Fix a proper, non-empty subset $A$ of $S_{0}$. Then, for all
$x\in S_{0}\setminus A$, $\eta\in\mathcal{D}_{\ell}(A)$, we have that
$\eta_{x}\neq 0$.
\end{cor}

\begin{proof}
Fix a proper, non-empty subset $A$ of $S_{0}$ and
$\eta\in\mathcal{D}_{\ell}(A)$.  Let
\begin{equation*}
A_{0} \;=\; \{\, x\in S_{0}:\eta_{x}=0\, \}\;.
\end{equation*}
As $\eta\in\mathcal{D}_{\ell}(A)$,
$P_{\ell}^{A}(\eta) = W_{\ell}(\eta)$.  Hence, by Lemma \ref{lem02},
$A_{0}\subset A$, as claimed.
\end{proof}

\subsection{The set $\mathcal{D}_{\ell}(A)$}

The crucial point in the proof of Theorem \ref{t915} is to estimate
$W_{\ell}$. This is relatively easy in each set
$\text{int }\,\mathcal{D}_{\ell}(A)$ because $W_{\ell}$ is equal to
$P_{\ell}^{A}$. In contrast, its behavior at the boundary
$\partial\mathcal{D}_{\ell}(A)$ is problematic.

The next result states that $\sum_{x\in A} \eta_x$ can not be too
large for configurations $\eta$ in $\mathcal{D}_{\ell}(A)$.

\begin{prop}
\label{p917}
There exists $\gamma_1>0$ such that, for all proper, non-empty subsets
$A$ of $S_{0}$,
\begin{equation*}
\mathcal{D}_{\ell}(A)\;\subset\; \mathcal{G}_{\ell}^{\gamma_1}(A)
\;:=\;
\{\eta\in\mathcal{U}_{N}^{x_{0}}:
\eta_{x}<\gamma_1\ell \;\text{ for all }\; x\in A\}\;.
\end{equation*}
\end{prop}

\begin{proof}
Fix $\eta\in\mathcal{D}_{\ell}(A)$.  By definition of
$\mathcal{D}_{\ell}(A)$,
\begin{equation*}
P_{\ell}^{A}(\eta) \;\le\;  P_{\ell}^{\varnothing}(\eta) \;=\; 0\;.
\end{equation*}
On the other hand, by definition of $P_{\ell}^{A}$, there exists
$\gamma_{A}>0$ such that
\begin{equation*}
P_{\ell}^{A}(\xi) \;>\; 0 \;\text{ if }\;
\xi_{x} \;\ge\; \gamma_{A}\, \ell \;\text{ for some }\;
x\in A\;.
\end{equation*}
It follows from the two previous remarks that
$\mathcal{D}_{\ell}(A) \subset \mathcal{G}_{\ell}^{\gamma_{A}}(A)$.
To complete the proof, it remains to set $\gamma_1=\max\gamma_{A}$.
\end{proof}

\begin{prop}
\label{p010}
Fix a proper, non-empty subset $A$ of $S_{0}$ and
$\eta\in\text{\ensuremath{\textrm{\rm int } }}\mathcal{D}_{\ell}(A)$.
Then,
\begin{equation*}
(\mathscr{A}_{N}W_{\ell}) (\eta) \;=\;
(\mathscr{A}_{N}P^{A})(\eta) \;\ge\;
\sum_{x\in S_{0}}\mathbf{1}\{\eta_{x}=1\}\;.
\end{equation*}
\end{prop}

\begin{proof}
Fix $\eta\in\text{\ensuremath{\textrm{\rm int
    }}}\mathcal{D}_{\ell}(A)$, so that
\begin{equation*}
W_{\ell}(\eta) \;=\; P_{\ell}^{A}(\eta) \quad \text{and}\quad
W_{\ell}(\sigma^{x,\,y}\eta) \;=\; P_{\ell}^{A}(\sigma^{x,\,y}\eta)
\end{equation*}
for all $x$, $y$ in $S$ with $r(x,\,y)>0$.  Thus,
\begin{equation*}
(\mathscr{A}_{N}W_{\ell}) (\eta)
\;=\; (\mathscr{A}_{N}P_{\ell}^{A}) (\eta)
\;=\; (\mathscr{A}_{N}P^{A}) (\eta)\;.
\end{equation*}

We turn to the second assertion.  By Corollary \ref{p07},
$\eta_{x}\neq 0$ for all $x\in S_{0}\setminus A$. If $\eta_{x}=1$ for
some $x\in S_{0}\setminus A$, by Corollary \ref{p07},
$\sigma^{x,\,y}\eta\notin\mathcal{D}_{\ell}(A)$ for any $y\in S$ with
$r(x,\,y)>0$, so that
$\eta \notin \text{\ensuremath{\textrm{int}}}\,\mathcal{D}_{\ell}(A)$
as well. Therefore, $\eta_{x}\ge2$ for all $x\in S_{0}\setminus A$,
and the second claim follows from the second assertion of Proposition
\ref{p97}.
\end{proof}

\begin{lem}
\label{lem920}
Fix $x\not = y \in S$, and proper subsets $A$, $B$ of $S_0$,
$A\neq B$.  There exists a   constant $C_0>0$ such that
\begin{equation*}
\left| \, P_{\ell}^{B}(\eta)-P_{\ell}^{A}(\eta) \,\right|
\;\le\;  C_0\, \ell  \;\;\;\text{and\;\;\;}
\left|\, P_{\ell}^{B}(\sigma^{x,\,y}\eta) \,-\,
P_{\ell}^{A}(\sigma^{x,\,y}\eta)\, \right|
\;\le\; C_0 \, \ell
\end{equation*}
for all $\eta\in\text{int }\,\mathcal{U}_{N}^{x_{0}}$
such that $\eta\in\mathcal{D}_{\ell}(A)$ and
$\sigma^{x,\,y}\eta\in\mathcal{D}_{\ell}(B)$,
\end{lem}

\begin{proof}
Since the proof for these two estimates are identical, we only focus
on the first one. We regard $P_{\ell}^{A}$ and $P_{\ell}^{B}$ as
quadratic functions on $\mathbb{R}^{\kappa-1}$ whose restriction to
$\bb N^{\kappa-1}$ is given by \eqref{ff11}.

As $\eta$ belongs to $\mathcal{D}_{\ell}(A)$ and $\sigma^{x,\,y}\eta$
to $\mathcal{D}_{\ell}(B)$,
\begin{equation*}
P_{\ell}^{A}(\eta) \; \le\;  P_{\ell}^{B}(\eta)
\;\;,\;\;\;
P_{\ell}^{A}(\sigma^{x,\,y}\eta) \;\ge\;
P_{\ell}^{B}(\sigma^{x,\,y}\eta)\;.
\end{equation*}
Hence, by the intermediate value theorem, there exists
$\text{w}_{0}\in\mathbb{R}^{\kappa-1}$ belonging to the line segment
connecting $\eta$ and $\sigma^{x,\,y}\eta$ such that
\begin{equation*}
(P_{\ell}^{A}-P_{\ell}^{B})(\text{w}_{0})=0\;.
\end{equation*}

Since
\begin{equation*}
|\, \eta \,-\, \text{w}_{0}\,| \;\le\;
|\, \eta \,-\, \sigma^{x,\,y}\eta\,|
\;=\; \sqrt{2}\;,
\end{equation*}
by the Taylor expansion, there exists a finite constant $C_0$ such
that
\begin{equation*}
\left|\, P_{\ell}^{A}(\eta) \,-\,
P_{\ell}^{A}(\text{w}_{0})\, \right|
\;\le\; C_0\,
\big\{\, |\, \nabla P_{\ell}^{A}(\eta)\,|
\;+\; \Vert\, \nabla^{2}P_{\ell}^{A} \,
\Vert_{L^{\infty}(\mathbb{R}^{\kappa - 1})}\, \big\}\;.
\end{equation*}
As $P_{\ell}^{A}$ is a quadratic function,
$\Vert\, \nabla^{2}P_{\ell}^{A}\,
\Vert_{L^{\infty}(\mathbb{R}^{\kappa-1})}\le C_0$.  On the other hand,
since $\eta$ belongs to $\mathcal{D}_{\ell}(A)$, by Proposition
\ref{p917}, $\eta_{z}\le\gamma_1 \ell$ for all $z\in A$. Hence, there
exists a finite constant $C_0$ such that
$|\, \nabla P_{\ell}^{A}(\eta)\, | \,<\, C_0 \ell$, and the
previous displayed equation becomes
\begin{equation}
\label{ff13}
\left|\, P_{\ell}^{A}(\eta) \,-\,
P_{\ell}^{A}(\text{w}_{0})\, \right|
\;\le\; C_0\, \ell\;.
\end{equation}

To use the same argument to estimate
$P_{\ell}^{B}(\eta) \,-\, P_{\ell}^{B}(\text{w}_{0})$
we only need to show that $\eta_z \le C_0 \ell$ for all $z\in B$.
Since $\sigma^{x,\,y}\eta\in\mathcal{D}_{\ell}(B)$, by Proposition
\ref{p917}, $(\sigma^{x,\,y} \eta)_{z}\le\gamma_1 \ell$ for all $z\in
B$. Thus, as $|\, (\sigma^{x,\,y} \eta)_{z} - \eta_z \,|\le 1$,
$\eta_z \le (\sigma^{x,\,y} \eta)_{z}  + 1 \le \gamma_1\, \ell + 1$
for all $z\in B$. This proves \eqref{ff13} with $A$ replaced by $B$.

Putting together the previous estimates yields that
\begin{equation*}
\left|\, P_{\ell}^{A}(\eta) \,-\,
P_{\ell}^{B}(\eta)\, \right|
\;\le\; \left|\, P_{\ell}^{A}(\eta)
\,-\, P_{\ell}^{A}(\text{w}_{0})\, \right|
\;+\; \left|\, P_{\ell}^{B}(\eta) \,-\,
P_{\ell}^{B}(\text{w}_{0})\, \right| \;\le\;  C_0 \, \ell\; ,
\end{equation*}
as claimed
\end{proof}

\begin{lem}
\label{lem924}
If $\eta$ belongs to $\partial\mathcal{D}_{\ell}(A)$,
there exists a constant $C_0$ such that
\begin{equation*}
\left|\, (\mathscr{A}_{N}W_{\ell})(\eta)
\,-\, (\mathscr{A}_{N}P^{A})(\eta)\, \right|
\;\le\; C_0 \, \ell\;.
\end{equation*}
\end{lem}

\begin{proof}
Assume that $\eta$ belongs to $\partial\mathcal{D}_{\ell}(A)$. It is
enough to show that there exists a   constant $C_0$ such that for
all $\eta\in\partial\mathcal{D}_{\ell}(A)$ and $x,\,y\in S$ with
$r(x,\,y)>0$,
\begin{equation*}
\left|\, W_{\ell}(\sigma^{x,\,y}\eta) \,-\,
P_{\ell}^{A}(\sigma^{x,\,y}\eta)\, \right|
\;\le\; C_0\, \ell\;.
\end{equation*}
This inequality holds clearly when
$\sigma^{x,\,y}\eta\in\mathcal{D}_{\ell}(A)$.  Assume that
$\sigma^{x,\,y}\eta\in\mathcal{D}_{\ell}(B)$ for some $B\neq A$. Then,
\begin{equation*}
\left|\, W_{\ell}(\sigma^{x,\,y}\eta)
\,-\, P_{\ell}^{A}(\sigma^{x,\,y}\eta)\, \right|
\;=\; \left|\, P_{\ell}^{B}(\sigma^{x,\,y}\eta)
\,-\, P_{\ell}^{A}(\sigma^{x,\,y}\eta)\, \right|\;.
\end{equation*}
By Lemma \ref{lem920}, this quantity is bounded by $C_0\, \ell$.
\end{proof}

The following proposition is crucial in the proof of Theorem
\ref{t915}. It is here that condition \eqref{ff17} plays a role.
Let
\begin{equation*}
\partial\mathcal{D}_{\ell}
\;=\; \bigcup_{A\subsetneq S_{0}}
\partial\mathcal{D}_{\ell}(A) \;.
\end{equation*}

\begin{prop}
\label{p921}
There exists a constant $\gamma_{2}>0$ such that, for all
$\eta\in\mathcal{U}_{N}^{x_{0}}$,
\begin{equation*}
\sum_{\ell\ge 2} \bs 1\{\eta\in\partial\mathcal{D}_{\ell}\}
\;\le\; \gamma_{2}\;.
\end{equation*}
In other words, each configuration $\eta\in\mathcal{U}_{N}^{x_{0}}$
belongs to a boundary set $\partial\mathcal{D}_{\ell}(A)$ at most
$\gamma_{2}$ times.
\end{prop}

\begin{proof}
Fix $\eta\in\partial\mathcal{D}_{\ell}(A)$, so that there exists
$x,\,y\in S$ with $r(x,\,y)>0$ such that
$\sigma^{x,\,y}\eta\in\mathcal{D}_{\ell}(B)$ for some $B\neq
A$. By Lemma \ref{lem920}, there exists $C_{0}>0$ such that
\begin{equation*}
\left|P_{\ell}^{A}(\eta)-P_{\ell}^{B}(\eta)\right|\le C_{0}\ell\;.
\end{equation*}
Therefore, it suffices to prove that there exists a finite constant
$C_1$ such that
\begin{equation*}
\sum_{\ell=1}^{\infty}\mathbf{1}\left\{\,
\left|P_{\ell}^{A}(\eta)-P_{\ell}^{B}(\eta)\right|
\;\le\; C_{0}\ell\, \right\} \;\le\; C_1\;.
\end{equation*}
Recall that
\begin{equation*}
P_{\ell}^{A} (\eta) \;-\;
P_{\ell}^{B}(\eta) \;=\; P^{A}(\eta) \;-\;
P^{B} (\eta) \;-\; (c_{A}-c_{B})\, \ell^{2}\;.
\end{equation*}
Since $c_A \not = c_B$, the left-hand side of the penultimate displayed
equation can be written as
\begin{equation*}
\sum_{\ell=1}^{\infty} \mathbf{1}
\Big\{ \, \Big| \, \ell^{2} \,-\,
\frac{(P^{A}-P^{B})(\eta)}{c_{A}-c_{B}}\,\Big|
\;\le\; \frac{C_{0}\ell}{|c_{A}-c_{B}|}\, \Big\} \;.
\end{equation*}
By Lemma \ref{lem922} below, this sum is bounded by a constant which
only depends on $c_{A},\,c_{B}$ and $C_{0}$, as claimed.
\end{proof}

\begin{lem}
\label{lem922}
For $\alpha>0$ and $t\in\mathbb{R}$, the set
\begin{equation*}
A_{\alpha,\,t} \;=\; \big\{\, x\in\mathbb{R} :
x^{2}-2\alpha x+t\le0\le x^{2}+2\alpha x+t \,\big\}
\end{equation*}
is either an empty set or a closed interval of length at most
$2\alpha$.
\end{lem}

\begin{proof}
If $t>\alpha^{2}$, the inequality $x^{2}+2\alpha x+t<0$ cannot hold
and the set $A_{\alpha,\,t}$ is empty. We may, therefore, assume that
$t\le\alpha^{2}$. In this case, let
\begin{equation*}
u^{\pm} \;=\; \alpha \,\pm\,
\sqrt{\alpha^{2}-t}\;, \qquad
v^{\pm} \;=\; -\, \alpha \;\pm\; \sqrt{\alpha^{2}-t}\;,
\end{equation*}
so that
\begin{equation*}
A_{\alpha,\,t} \;=\; [\, u^{-},\,u^{+}\,]
\;\setminus\; (v^{-},\,v^{+})\;.
\end{equation*}
This set is a closed sub-interval of $[v^{+},\,u^{+}]$ and
$u^{+} - v^{+} = 2\alpha$. This completes the proof.
\end{proof}

\subsection{The function $h_\ell$}

Fix $\ell\ge 2$, and let
\begin{equation*}
h_{\ell}(\eta) \;:=\; P(\eta) \,-\, W_{\ell} (\eta)\;.
\end{equation*}

The next result is the main step in the construction of a
super-harmonic function.

\begin{prop}
\label{p923}
There exist positive constants $c_{1}$, $c_{2}$ such that
\begin{equation*}
(\mathscr{A}_{N}h_{\ell}^{1/2})(\eta) \;\le\;
\frac{1}{P(\eta)^{1/2}} \,
\big\{ -\, c_{1} \;+\; c_{2}\, \ell\,
\mathbf{1}\{\eta\in\partial\mathcal{D}_{\ell}\, \}\, \big\}
\end{equation*}
for all $\ell\ge2$, large enough $N$ and
$\eta\in\text{\rm int }\mathcal{U}_{N}^{x_{0}}$.
\end{prop}

To prove this proposition, we first investigate
$\mathscr{A}_{N} h_{\ell}$.

\begin{lem}
\label{lem925}
There exists a finite constant $C_0$ such that
\begin{equation*}
(\mathscr{A}_{N} h_{\ell})(\eta) \;\le\;
C_0\, \ell\, \mathbf{1}\{\eta\in\partial\mathcal{D}_{\ell}\}
\,+\, o_{N}(1)
\end{equation*}
for all $\eta\in\text{\rm int }\ensuremath{\mathcal{U}_{N}^{x_{0}}}$.
\end{lem}

\begin{proof}
Suppose that $\eta\in\textrm{int}\,\mathcal{D}_{\ell}(A)$ for some
proper subset $A$ of $S_{0}$. By Proposition \ref{p010} and the third
assertion of Proposition \ref{p97},
\begin{equation*}
(\mathscr{A}_{N} h_{\ell}) (\eta) \;=\;
(\mathscr{A}_{N}P) (\eta) \;-\;
(\mathscr{A}_{N}W^{\ell})m(\eta) \;\le\; o_{N}(1)\;.
\end{equation*}

Assume that $\eta\in\partial\mathcal{D}_{\ell}(A)$, for some proper
subset $A$ of $S_{0}$. By Lemma \ref{lem924},
\begin{equation*}
(\mathscr{A}_{N}h_{\ell}) (\eta) \;\le\;
(\mathscr{A}_{N}P)(\eta) \;-\;
(\mathscr{A}_{N} P^A)(\eta) \;+\; C_0\,\ell
\end{equation*}
for some finite constant $C_0$.  By the first assertion of Proposition
\ref{p97},
\begin{equation*}
\mathscr{A}_{N}P^{A}(\eta) \;\ge\; -\, C_0 \sum_{x\in A}\eta_{x}
\end{equation*}
for some finite constant $C_0$. By Proposition \ref{p917}, this
expression is bounded below by
$-\, C_0\, \gamma_1\, \ell \,=\, -\, C_0\, \ell$. On the other hand,
by the third assertion of Proposition \ref{p97},
$(\mathscr{A}_{N}P)(\eta) \le \kappa + o_N(1)$. This completes the
proof of the proposition.
\end{proof}

The next result is an extension of Lemma \ref{lem910}.

\begin{lem}
\label{lem926}
There exists a positive constant $c_0$ such that
\begin{equation*}
\sum_{x\in S} {g}(\eta_{x})\,
\sum_{y\in S} r(x,\,y)\, \left[\, h_{\ell}(\sigma^{x,\,y}\eta)
\,-\, h_{\ell}(\eta)\, \right]^{2} \;\ge\; c_0\, P(\eta)
\end{equation*}
for all $\eta\in(\text{\rm int }
\mathcal{U}_{N}^{x_{0}}) \setminus \partial\mathcal{D}_{\ell}$.
\end{lem}

\begin{proof}
Since ${g}(\eta_{x_{0}})>0$, it suffices to show that
\begin{equation*}
\sum_{y\in S}r(x_{0},\,y)\, \left[\, h_{\ell}(\sigma^{x_{0},\,y}\eta)
\,-\, h_{\ell}(\eta)\, \right]^{2} \;\ge\; c_0\, P(\eta)\;.
\end{equation*}
By the Cauchy-Schwarz inequality, the square of the left-hand side is
bounded below by
\begin{equation*}
\Big\{ \sum_{y\in S}r(x_{0},\,y) \Big\}^{-1}
\Big\{ \, \sum_{y\in S} r(x_{0},\,y)\,
[\, h_{\ell}(\sigma^{x_{0},\,y}\eta)-h_{\ell}(\eta)\, ]
\Big\}^{2}\;.
\end{equation*}
Thus, by Lemma \ref{lem99}, it is enough to show that
\begin{equation}
\label{ebb}
\sum_{y\in S} r(x_{0},\,y)\,
[\, h_{\ell}(\sigma^{x_{0},\,y}\eta) \,-\, h_{\ell}(\eta) \,]
\;\ge\; c_0\, \sum_{x\in S_{0}}\eta_{x}\;.
\end{equation}

Since $\eta\not\in\partial\mathcal{D}_{\ell}$, $\eta$ belongs to
$\text{int}\,\mathcal{D}_{\ell}(A)$ for some $A\subsetneq S_{0}$. In
particular, $\eta$ and $\sigma^{x_{0},\,y}\eta$ belong to
$\mathcal{D}_{\ell}(A)$ for all $y$ such that $r(x_{0},\,y)>0$. The
left-hand side of the previous displayed equation is thus equal to
\begin{equation*}
\sum_{y\in S} r(x_{0},\,y) \,
\big[\, (P-P_{\ell}^{A})(\sigma^{x_{0},\,y}\eta)
\,-\, (P-P_{\ell}^{A})(\eta)\,\big]\;.
\end{equation*}

If $A=\varnothing$, then
$P_{\ell}^{A}(\sigma^{x_{0},\,y}\eta)=P_{\ell}^{A}(\eta)=0$. In this
case, \eqref{ebb} follows from the third assertion of Lemma
\ref{lem96}.  If $A\neq\varnothing$,
\begin{equation*}
P_{\ell}^{A}(\sigma^{x_{0},\,y}\eta) \,-\,
P_{\ell}^{A}(\eta) \;=\;
P^{A}(\sigma^{x_{0},\,y}\eta) \,-\, P^{A}(\eta)\;.
\end{equation*}
By the second and third statements of Lemma \ref{lem96}, the left-hand
side of \eqref{ebb} is equal to
\begin{equation*}
\sum_{z\in S_{0}}\eta_{z} \;-\;
\sum_{z\in A}u_{z,\,x_{0}}^{A}\, \eta_{z}\;.
\end{equation*}
On the set $\mc U^{x_0}_N$, $\sum_{z\in S_{0}}\eta_{z}\ge N^{\gamma}$
and, by Proposition \ref{p917}, $\eta_{z}\le\gamma_{1}\ell$ for all
$z\in A$. In particular, the previous expression is greater than
$(1/2)\, \sum_{z\in S_{0}}\eta_{z}$ for $N$ large enough. This
completes the proof.
\end{proof}

\begin{proof}[Proof of Proposition \ref{p923}]
By definition,
\begin{equation*}
(\mathscr{A}_{N}h_{\ell}^{1/2})(\eta)
\;=\; \sum_{x\in S}{g}(\eta_{x})\,
\sum_{y\in S}r(x,\,y)\,
\big[\, h_{\ell}^{1/2}(\sigma^{x,\,y}\eta)
\,-\, h_{\ell}^{1/2}(\eta) \, \big]\;.
\end{equation*}

By Lemmata \ref{lem99} and \ref{lem913}, there exists a positive
constant $c_0$ such that
\begin{equation*}
h_{\ell} (\eta) \;\ge\; P(\eta) \;\ge\; c_0 \, \Big(\sum_{z\in S_0}
\eta_z \,\Big)^2
\end{equation*}
for all $\eta\in \text{int } \mc U^{x_0}_N$. On the other hand, by
definition of $h_\ell$, \eqref{ff18} [for $A=S_0$ and $A$ a proper
subset of $S_0$] and Lemma \ref{lem920}, there exists a finite
constant $C_0$ such that
\begin{equation*}
\big| \, h_\ell (\sigma^{x,\,y}\eta) \;-\; h_\ell(\eta)\, \big|
\;\le\; C_0\, \Big\{ \, \ell \;+\; \sum_{z\in S_0} \eta_z \, \Big\}
\end{equation*}
for all $x$, $y\in S$, $y\not = x$ and
$\eta\in\text{int }\mathcal{U}_{N}^{x_{0}}$ such that $\eta_{x}\ge 1$.
This expression is bounded by $C_0\, \sum_{z\in S_0} \eta_z$ for $N$
sufficiently large.
Since $\sum_{z\in S_0} \eta_z \ge N^\gamma$ on $\text{int } \mc
U^{x_0}_N$, it follows from the two previous estimates that there
exists a finite constant $C_0$ such that
\begin{equation}
\label{ff19}
\frac{\big| \, h_\ell (\sigma^{x,\,y}\eta) \;-\; h_\ell(\eta)\, \big|}
{h_\ell(\eta)} \;\le\; C_0 \, N^{-\gamma}
\end{equation}
for all $x$, $y\in S$, $y\not = x$ and
$\eta\in\text{int }\mathcal{U}_{N}^{x_{0}}$ such that $\eta_{x}\ge 1$.

A second order Taylor expansion and the previous bound yield that
$(\mathscr{A}_{N}h_{\ell}^{1/2})(\eta)$ is equal to
\begin{align*}
\frac{(\mathscr{A}_{N} h_{\ell})(\eta)}{2 h_{\ell}(\eta)^{1/2}}
\;-\; [\, 1 \,+\, c_N\,]\,
\frac{1}{8h_{\ell}(\eta)^{3/2}}
\sum_{x\in S}{g}(\eta_{x})\, \sum_{y\in S}r(x,\,y)\,
[\, h_{\ell}(\sigma^{x,\,y}\eta) \,-\,
h_{\ell}(\eta)\,]^{2}\;,
\end{align*}
where $c_N$ is bounded by $C_0 N^{-\gamma}$. Hence, by Lemmata
\ref{lem925} and \ref{lem926}, there exist a finite constant $C_0$ and
a positive constant $c_0$ such that
\begin{equation*}
(\mathscr{A}_{N} \, h_{\ell}^{1/2})(\eta)
\;\le\;
\frac{C_0 \, \ell\, \mathbf{1}\{\eta\in\partial\mathcal{D}_{\ell}\}}
{h_{\ell}(\eta)^{1/2}}
\;+\; \frac{o_{N}(1)}{h_{\ell}(\eta)^{1/2}}\
\;-\; \mathbf{1}\{\eta\not \in\partial\mathcal{D}_{\ell}\}\,
\frac{c_{0} \, P(\eta)}{h_{\ell}(\eta)^{3/2}} \;\cdot
\end{equation*}
Write $\mathbf{1}\{\eta\not \in\partial\mathcal{D}_{\ell}\}$ as
$ 1 \,-\, \mathbf{1}\{\eta \in\partial\mathcal{D}_{\ell}\}$. Since
$h_{\ell}(\eta) \ge P(\eta)$ and, by Lemma \ref{lem913},
$P(\eta) \ge (1/2) h_{\ell}(\eta)$.
\begin{equation*}
(\mathscr{A}_{N} \, h_{\ell}^{1/2})(\eta)
\;\le\;
\frac{C_0 \, \ell\, \mathbf{1}\{\eta\in\partial\mathcal{D}_{\ell}\}}
{P(\eta)^{1/2}}
\;-\; \frac{c_{0}}{P(\eta)^{1/2}} \;,
\end{equation*}
as claimed.
\end{proof}

\subsection{Proof of Theorem \ref{t915}}
\label{sec95}

\begin{proof}[Proof of Theorem \ref{t915}]
The function $F_{m}$ can be written as
\begin{equation*}
F_{m}(\eta) \;=\; \sum_{\ell=2}^{m}
\frac{1}{\ell}\, h_{\ell}(\eta)^{1/2}\;.
\end{equation*}
By Proposition \ref{p923},
\begin{equation*}
P(\eta)^{1/2}\,(\mathscr{A}_{N}F_{m})(\eta)
\;\le\; -\, c_{0}\sum_{\ell=2}^{m}\frac{1}{\ell}
\;+\; C_0 \sum_{\ell=2}^{m}
\mathbf{1}\{\eta\in\partial\mathcal{D}_{\ell}\}\;.
\end{equation*}
By Proposition \ref{p921}, this expression is bounded by
\begin{equation*}
-\, c_{0} \log m \;+\; C_{0}\, \gamma_2\;.
\end{equation*}
Thus, by taking $m$ large enough, there exists $c'_0>0$ such that
$P(\eta)^{1/2}\, (\mathscr{A}_{N}F_{m})(\eta)<-c'_0<0$ for all
$\eta\in\text{int }\,\mathcal{U}_{N}^{x_{0}}$. It remains to recall
the statement of Lemma \ref{lem99} to complete the proof.
\end{proof}

\smallskip\noindent{\bf Acknowledgments.}
The authors wish to thank M. Loulakis and S. Grosskinsky for
references on the Efron-Stein inequality.

C. L. has been partially supported by FAPERJ CNE E-26/201.207/2014, by
CNPq Bolsa de Produtividade em Pesquisa PQ 303538/2014-7, by
ANR-15-CE40-0020-01 LSD of the French National Research Agency.
I.S. was supported by the National Research Foundation of Korea (NRF)
grant funded by the Korea government (MSIT) (No. 2018R1C1B6006896 and
No. 2017R1A5A1015626), the Samsung Science and Technology Foundation (Project
Number SSTF-BA1901-03), and POSCO Science Fellowship of POSCO TJ Park
Foundation. D. M. has received financial support from CNPq during the
development of this paper.

Part of this work has been done when the first two authors were at the
Seoul National University. The warm hospitality is acknowledged.

\end{document}